\documentclass{mcom-l}

\usepackage{stmaryrd}
\SetSymbolFont{stmry}{bold}{U}{stmry}{m}{n}
\usepackage{mathtools}
\usepackage{booktabs}
\usepackage{multirow}
\usepackage{bm}
\usepackage{enumitem}
\usepackage{caption}

\usepackage{tikz}
\usepackage[caption=false]{subfig}

\newtheorem{theorem}{Theorem}[section]
\newtheorem{lemma}[theorem]{Lemma}
\newtheorem{corollary}[theorem]{Corollary}

\theoremstyle{definition}

\newtheorem{assumption}[theorem]{Assumption}

\theoremstyle{remark}
\newtheorem{remark}[theorem]{Remark}

\numberwithin{equation}{section}

\DeclarePairedDelimiter\jump{\llbracket}{\rrbracket}

\def\tens#1{\pmb{\mathsf{#1}}}
\def\vec#1{\boldsymbol{#1}}

\def\R{\mathbb{R}}
\def\sym{\mathop{\mathrm{sym}}\nolimits}
\def\tr{\mathop{\mathrm{tr}}\nolimits}
\def\Rds{\mathbb{R}^{d \times d}_{\sym}}
\def\Rdst{\mathbb{R}^{d \times d}_{\sym,\tr}}

\def\supp{\mathop{\mathrm{supp}}\nolimits}
\def\macrostar{\mathop{\mathrm{macrostar}}\nolimits} %identity

\def\diver{\mathop{\mathrm{div}}\nolimits} %\divergence
\def\twodot{\,{:}\,}

\newcommand{\Cinfty}[3][\Omega]{C^{\infty}_{#2}(#1)^{#3}}
\newcommand{\Cinfd}{\Cinfty{0}{d}}

\newcommand{\Lebs}[4][\Omega]{L^{#2}_{#3}(#1)^{#4}}
\newcommand{\Lp}[1]{\Lebs{#1}{}{}}
\newcommand{\Lmean}[1]{\Lebs{#1}{0}{}}
\newcommand{\Lsym}[1]{\Lebs{#1}{\sym}{d \times d}}

\newcommand{\Lsymtr}[1]{\Lebs{#1}{\sym,\tr}{d \times d}}
\newcommand{\Sobs}[5][\Omega]{W^{#2, #3}_{#4}(#1)^{#5}}
\newcommand{\SobsH}[4][\Omega]{H^{#2}_{#3}(#1)^{#4}}
\newcommand{\Wonep}[2][d]{\Sobs{1}{#2}{}{#1}}
\newcommand{\Wmonep}[2][d]{\Sobs{-1}{#2}{}{#1}}
\newcommand{\Hone}[1][d]{\SobsH{1}{}{#1}}
\newcommand{\Hmone}[1][d]{\SobsH{-1}{}{#1}}

\newcommand{\Wonezdiv}[1]{\Sobs{1}{#1}{0,\diver}{d}}

\newcommand{\Wbdry}[4][\Gamma]{W^{1/#2',#2}_{#3}(#1)^{#4}}
\newcommand{\Hbdry}[3][\Gamma]{H^{1/2}_{#2}(#1)^{#3}}

\newcommand{\Wtrzero}[2][\Gamma]{\Wbdry[#1]{#2}{00}{}}

\newcommand{\Hhalfzero}[1][\Gamma]{\Hbdry[#1]{00}{}}
\def\be{\vec{e}}
\def\bu{\vec{u}}
\def\bv{\vec{v}}
\def\bw{\vec{w}}
\def\bz{\vec{z}}

\def\bsigma{\vec{\sigma}}
\def\btau{\vec{\tau}}

\def\BD{\tens{D}}
\def\BG{\tens{G}}
\def\BH{\tens{H}}

\def\BS{\tens{S}}
\def\Du{\BD(\bu)}
\def\Dcr{\bm{\mathcal{D}}}
\def\Scr{\bm{\mathcal{S}}}

\def\Pr{\mathrm{Pr}}
\def\Ra{\mathrm{Ra}}
\def\Re{\mathrm{Re}}
\def\Gr{\mathrm{Gr}}
\def\Di{\mathrm{Di}}
\def\Bn{\mathrm{Bn}}
\def\Pe{\mathrm{Pe}}
\def\Br{\mathrm{Br}}

\def\Vndiv{V^n_{\diver}}

\def\Signtr{\Sigma^n_{\tr}}

\newcommand{\txtb}[1]{{\color{black}#1}}%blue
\newcommand{\txtbl}[1]{{\color{black}#1}}
\begin{document}

\title[FEM for Heat-conducting Implicit Fluids]{Finite Element Approximation and Preconditioning for Anisothermal  \txtb{Flow of} Implicitly-constituted Non-Newtonian \txtb{Fluids}}


%    author one information
\author[P.~E.~Farrell]{Patrick Farrell}
\address{Mathematical Institute, University of Oxford, Oxford OX2 6GG, UK}
%\curraddr{}
\email{patrick.farrell@maths.ox.ac.uk}
\thanks{This research was supported by the Engineering and Physical Sciences Research Council grant EP/R029423/1, and by the EPSRC Centre for Doctoral Training in Partial Differential Equations: Analysis and Applications, grant EP/L015811/1. The second author was supported by CONACyT (Scholarship 438269)}

%    author two information
\author[P.~A.~Gazca~Orozco]{Pablo Alexei Gazca~Orozco}
\address{Mathematical Institute, University of Oxford, Oxford OX2 6GG, UK}
\curraddr{Department of Mathematics, FAU Erlangen-N\"{u}rnberg, 91058 Erlangen, Germany}
\email{alexei.gazca@math.fau.de}
%\thanks{}

%    author three information
\author[E.~S\"{u}li]{Endre S\"uli}
\address{Mathematical Institute, University of Oxford, Oxford OX2 6GG, UK}
%\curraddr{}
\email{endre.suli@maths.ox.ac.uk}
%\thanks{}
%    \subjclass is required.

\date{\today}

\dedicatory{}

%    Abstract is required.
\begin{abstract}
We devise 3-field and 4-field finite element approximations of a system describing the steady state of an incompressible heat-conducting fluid with implicit non-Newtonian rheology. We prove that the sequence of numerical approximations converges to a weak solution of the problem. We develop a block preconditioner based on augmented Lagrangian stabilisation for a discretisation based on the Scott--Vogelius finite element pair for the velocity and pressure. The preconditioner involves a specialised multigrid algorithm that makes use of a space decomposition that captures the kernel of the divergence and non-standard intergrid transfer operators. The preconditioner exhibits robust convergence behaviour when applied to the Navier--Stokes and power-law systems, including temperature-dependent viscosity, heat conductivity and viscous dissipation.%  When applied to more complex non-Newtonian problems, the Schur complement approximation remains satisfactory, but the multigrid algorithm loses its effectiveness.
\end{abstract}

\maketitle


\section{Introduction}
For $d\in\{2,3\}$, let $\Omega\subset \mathbb{R}^d$ be a bounded polytopal domain with a Lipschitz boundary. The steady form of the Oberbeck--Boussinesq \cite{Oberbeck1879,Boussinesq1903} approximation used in the modelling of natural convection reads:
\begin{subequations}\label{eq:PDE_OB}
	\begin{alignat}{2}
%		\begin{aligned}
%			-\diver \BS + \rho_0 \diver (\bu\otimes\bu)
		\txtb{-\diver (2\hat{\mu}(\theta)\Du)} + \rho_0 \diver (\bu\otimes\bu)
			&+ \nabla p = -\rho_0\beta g (\theta - \theta_C)\be_d
			\quad & &\text{ in }\Omega,\\
			\diver\bu &= 0\quad & &\text{ in }\Omega,\\
			-\diver (\hat{\kappa}(\theta)\nabla\theta)
			+ \rho_0 c_p \diver(\bu & \theta )
			+ \beta\rho_0 g \theta\bu\cdot \be_d
			= \txtb{2\hat{\mu} |\BD(\bu)|^2}\quad & &\text{ in }\Omega, \label{eq:PDE_OB_energy}
%			= \BS \twodot \BD(\bu)\quad & &\text{ in }\Omega, \label{eq:PDE_OB_energy}
%		\end{aligned}
	\end{alignat}
\end{subequations}
	where $\be_d$ is the unit vector pointing against gravity, $\BD(\bu) = \frac{1}{2}(\nabla\bu + \nabla\bu^\top)$ denotes the symmetric gradient, and the quantities appearing in the equations are as follows:

\vspace{0.5cm}
\noindent
\begin{center}
\begin{tabular}{cc}
\toprule
$\bu\colon\Omega\to \R^d$ & velocity field\\
%$\BS\colon \Omega\to \Rdst$ & \txtb{deviatoric} stress\\
$p \colon \Omega \to \R$ & pressure\\
$\theta \colon \Omega \to \R$ & temperature\\
$\hat{\kappa} \colon \R \to \R$ & heat conductivity\\
$\hat{\mu} \colon \R \to \R$ & \txtb{viscosity}\\
\bottomrule
\end{tabular}
\hspace{0.1cm}
\begin{tabular}{cc}
\toprule
$\beta$ & thermal expansion coefficient\\
$c_p$ & specific heat capacity\\
$g$ & acceleration due to gravity\\
$\rho_0$ & reference density\\
$\theta_C$ & reference temperature\\
\bottomrule
\end{tabular}
\end{center}
\vspace{0.5cm}

\txtb{The system \eqref{eq:PDE_OB} assumes that the fluid in question behaves according to the Newtonian constitutive relation $\BS=2\hat{\mu}(\theta)\Du$, where $\BS\colon \Omega \to \Rdst$ is the deviatoric stress; here $\Rdst$ denotes the set of $d\times d$ symmetric and traceless matrices.} The system is supplemented with the boundary conditions
%The symbol $\Rdst$ above denotes the set of $d\times d$ symmetric and traceless matrices. The system is supplemented with the boundary conditions
\begin{equation}\label{eq:bdry_conditions}
	\bu|_{\partial\Omega}= \bm{0}, \qquad \theta|_{\Gamma_D} = \theta_b,
	\qquad \hat{\kappa}(\theta)\nabla\theta \cdot\bm{n}|_{\partial\Omega \setminus \Gamma_D} = 0,
\end{equation}
where $\Gamma_D$ is a relatively open subset of $\partial\Omega$ with $|\Gamma_D|\neq 0$, $\bm{n}$ is the unit outward-pointing normal vector to the boundary, and $\theta_b$ is a given temperature distribution on $\Gamma_D$. In many applications the effects of viscous dissipation are ignored, i.e.\ only the first two terms in the temperature equation \eqref{eq:PDE_OB_energy} are kept. However, it has been observed that in some cases the effects of the viscous dissipation term $\BS\twodot\BD(\bu) = 2\hat{\mu}(\theta)|\Du|^2$ are non-negligible and should be taken into account \cite{Hewitt1975,Turcotte1974,Velarde1976,Ostrach1958}. Furthermore, as noted in \cite{Bayly1992,Turcotte1974}, the viscous dissipation must be balanced with the adiabatic heating term $\beta\rho_0 g \theta\bu\cdot\be_d$; for a mathematically rigorous derivation of the system \eqref{eq:PDE_OB} see \cite{Kagei2000}. The existence of distributional solutions of \eqref{eq:PDE_OB} with non-Newtonian rheology of power-law type was shown in \cite{Roubicek2001,Necas2001}.

{\color{black}
	A distinct way of modelling the effects of the temperature on the flow is to employ the following system describing a general (homogeneous) incompressible fluid \cite{Bulicek2009}:
\begin{subequations}\label{eq:PDE_Forced}
	\begin{alignat}{2}
%		\begin{aligned}
			-\diver \BS + \rho_* \diver (\bu\otimes\bu)
			&+ \nabla p =  \rho_*\bm{f}
			\quad & &\text{ in }\Omega,\\
			\diver\bu = &\:0\quad & &\text{ in }\Omega,\\
			-\diver (\hat{\kappa}(\theta)\nabla\theta)
			+ \rho_* \diver(c_v&\bu  \theta )
			= \BS \twodot \BD(\bu)\quad & &\text{ in }\Omega, \label{eq:PDE_Forced_energy}
%		\end{aligned}
	\end{alignat}
\end{subequations}
where now $\rho_*$ is the (constant) density of the fluid, $c_v$ is the specific heat capacity at constant volume, and $\bm{f}$ is the body force.
%{\color{blue}
	The systems \eqref{eq:PDE_OB} and \eqref{eq:PDE_Forced} have different origins; for instance, the temperature appearing in \eqref{eq:PDE_Forced} is an absolute temperature (and therefore strictly positive), while the one appearing in \eqref{eq:PDE_OB} is a perturbation with respect to a reference temperature. However, the two systems share sufficient common structure to render a unified numerical analysis possible. We note that \eqref{eq:PDE_OB} and \eqref{eq:PDE_Forced} could be seen as particular cases of the following system (here the physical constants have been set to unity):

\begin{subequations}\label{eq:PDE_unified}
	\begin{alignat}{2}
%		\begin{aligned}
			-\diver \BS + \diver (\bu\otimes\bu)
			&+ \nabla p = \bm{f} -\alpha\theta\be_d
			\quad & &\text{ in }\Omega,\\
			\diver\bu &= 0\quad & &\text{ in }\Omega,\\
			-\diver (\hat{\kappa}(\theta)\nabla\theta)
			+   \diver(\bu & \theta )
			+ \alpha  \theta\bu\cdot \be_d
			= \BS \twodot \BD(\bu)\quad & &\text{ in }\Omega, \label{eq:PDE_unified_energy}
%		\end{aligned}
	\end{alignat}
\end{subequations}
where $\alpha\in \mathbb{R} $ is a parameter (when $\alpha=0$ one recovers the system \eqref{eq:PDE_Forced}). This
is the system we will study in this work.
}

The \txtb{systems \eqref{eq:PDE_Forced} and \eqref{eq:PDE_unified}} must be closed with a constitutive relation that relates the \txtb{deviatoric} stress $\BS$ and the symmetric velocity gradient $\BD(\bu)$\txtb{; one could for instance employ here the Newtonian relation $\BS=2\hat{\mu}(\theta)\BS$, as was done in \eqref{eq:PDE_OB}.} % The most commonly considered closure is the Newtonian constitutive relation $\BS = 2\hat{\mu}(\theta)\BD(\bu)$, where $\hat{\mu}\colon\R\to\R$ is the viscosity.
In this work we consider much more general implicit constitutive relations of the form $\BG(\BS,\BD(\bu),\theta) = \bm{0}$ and $\BH(\BS,\BD(\bu)) = \bm{0}$. The framework of implicitly constituted fluids is a generalisation of classical continuum mechanics that allows the study of a much wider class of materials in a thermodynamically consistent manner (see \cite{Rajagopal:2003,Rajagopal2006,Rajagopal2008}).

The first rigorous existence results within the implicitly constituted framework for the isothermal system can be found in \cite{Bulicek:2009,Bulicek:2012} (see also \cite{Blechta2019}), while an extension to a temperature-dependent system was carried out in \cite{Maringova2018}. Regarding the numerical approximation of these systems, only the isothermal case has been considered so far. The finite element approximation of the steady isothermal system was analysed in \cite{Diening:2013,refId0}, and extensions to the unsteady problem can be found in \cite{Sueli2018,Farrell2019}. An augmented Lagrangian preconditioner was proposed for a 3-field formulation of the isothermal system in \cite{FarrellGazca2019}.

Our analysis focuses on the constitutive relation defined by
\begin{equation}\label{eq:implicit_CR}
	\BG(\BS,\BD(\bu),\theta) := 2\hat{\mu}(\theta)\frac{(|\BD(\bu)|-\hat{\sigma}(\theta))^+}{|\BD(\bu)|}\BD(\bu)
	- \frac{(|\BS| - \hat{\tau}(\theta))^+}{|\BS|}\BS,
\end{equation}
where $\hat{\tau}\hat{\sigma} = 0$ (the precise assumptions on $\hat{\mu},\hat{\tau},\hat{\sigma}$ will be introduced later). The relation \eqref{eq:implicit_CR} describes a fluid with either Bingham or activated-Euler rheology in which the viscosity and activation parameters may depend on the temperature. Naturally, this family of constitutive relations also includes the Navier--Stokes model with a temperature-dependent viscosity (when $\hat{\tau}\equiv 0 \equiv \hat{\sigma}$). The relation \eqref{eq:implicit_CR} was introduced in \cite{Maringova2018}, where existence of weak solutions to the unsteady version of a similar system was shown. Our results will also cover relations with more general power-law behaviour if one restricts the system to have either an explicit constitutive relation or constant rheological parameters.

We make two main contributions in this work. First, we introduce a finite element approximation of the system \txtb{\eqref{eq:PDE_unified}} and prove convergence of the sequence of finite element approximations to a weak solution. This represents the first finite element convergence result for heat-conducting implicitly constituted fluids. For the sake of simplicity, we will neglect the viscous dissipation in the analysis. However, this can be included in the numerical algorithm without any difficulties. The main challenge associated with this term in the analysis stems from the fact that $\BS\twodot\BD(\bu)$ belongs \textit{a priori} to $L^1(\Omega)$ only,  and hence a suitable notion of renormalised solution must be employed for the temperature equation. We note that this difficulty has been circumvented in the PDE analysis of \txtb{system \eqref{eq:PDE_Forced}} in the transient case \cite{Bulicek2009,Bulicek2009a}. We would expect that by imposing certain restrictions on the mesh and for $\mathbb{P}_1$ elements, a similar convergence result would hold for an appropriately defined renormalised solution (c.f.\ \cite{Casado-Diaz}). When restricted to constant rheological parameters and the isothermal problem, the convergence result here improves on the result for $r$-graphs from \cite{Diening:2013} by extending it to cover the whole admissible range $r>\frac{2d}{d+2}$, even without pointwise divergence-free elements. This is possible by making use of reconstruction operators, which in recent years were introduced to restore the pressure-robustness in the finite element formulations (see e.g.\ \cite{John2017}).

The second main contribution is the development of a preconditioner based on an augmented Lagrangian approach for linearisations of the discretisation of \eqref{eq:PDE_OB}, including the viscous dissipation term. After Newton linearisation the system takes the following form
\begin{equation}\label{eq:block_structure}
    \begin{bmatrix}
    A & B^\top\\ B & 0
    \end{bmatrix}
    \begin{bmatrix}
    \bz \\ p
    \end{bmatrix}
    =
    \begin{bmatrix}
	    \bm{f} \\ g
    \end{bmatrix},
\end{equation}
where $\bz = (\theta,\bu)^\top$ or $\bz = (\BS,\theta,\bu)^\top$, depending on whether a 3-field or a 4-field formulation is employed, and $B$ represents the divergence operator acting on the velocity space. After performing Gaussian elimination on the blocks, the problem of solving \eqref{eq:block_structure} reduces to solving smaller systems involving $A$ and the Schur complement $S := -B A^{-1}B^\top$. In many cases, such as in a velocity-pressure formulation of the Stokes system, $A$ represents a symmetric and coercive operator which can be inverted efficiently, and so the challenge is to develop an effective and efficient approximation for the Schur complement inverse $\tilde{S}^{-1}$. For the Stokes system with constant viscosity $\nu$ it is known that the choice $\tilde{S}^{-1} = -\nu M_p^{-1}$, where $M_p$ is the pressure mass matrix, results in a spectrally equivalent preconditioner \cite{Silvester:1994,Mardal2011}. When the convective term is introduced to the formulation, the performance of this strategy degrades as the Reynolds number $\Re$ gets larger (meaning that the number of Krylov subspace iterations per nonlinear iteration grows with $\Re$) \cite{Elman2006}. This loss of robustness occurs also with other well-known preconditioners, such as the PCD \cite{Kay:2006} and LSC \cite{Elman2006} preconditioners (see e.g.\ \cite{Elman2014}). Block preconditioners based on PCD for the system \eqref{eq:PDE_OB} without viscous dissipation were proposed in \cite{Howle2012,Ke2017}, where it was observed that the number of linear iterations increased strongly with the Rayleigh number $\Ra$.

Alternatively, one can consider the system with an augmented Lagrangian term, with $\gamma > 0$:
\begin{equation}\label{eq:block_augmented}
\begin{bmatrix}
A + \gamma B^{\top}M_p^{-1}B & B^{\top}\\B &0
\end{bmatrix}
\begin{bmatrix}
\bz \\ p
\end{bmatrix}
=
\begin{bmatrix}
	\bm{f} + \gamma B^{\top}M_p^{-1} g \\ g
\end{bmatrix},
\end{equation}
which has the same solution as \eqref{eq:block_structure}, since $B\bz = g$. The advantage of this is that using the Sherman--Morrison--Woodbury formula (see e.g.\ \cite{Bacuta2006}), the Schur complement can be approximated in a straightforward way:
\begin{align*}
	S^{-1} &= (-B(A + \gamma B^{\top}M_p^{-1}B)^{-1}B^{\top})^{-1} = -(BA^{-1}B^{\top})^{-1} - \gamma M_p^{-1} \\
	       &     \approx -(\nu + \gamma)M_p^{-1},
\end{align*}
and the approximation gets better as $\gamma\to \infty$. The difficulty now becomes solving the linear system associated with top block $A + \gamma B^\top M_p^{-1}B$ efficiently, since the augmented Lagrangian term possesses a large kernel (the set of all discretely divergence-free velocities). This approach was used for the 2D Navier--Stokes system by Benzi and Olshanskii \cite{Benzi2006} and later extended to three dimensions by Farrell, Mitchell and Wechsung \cite{Farrell2019a}. The strategy for efficiently solving the top block in these works was based on ideas developed by Sch\"{o}berl in the context of nearly incompressible elasticity \cite{Schoeberl1999,Schoeberl:1999}, where it became clear that constructing robust relaxation and transfer operators is essential for obtaining a $\gamma$-robust multigrid algorithm.

These ideas will be applied here to develop a preconditioner for the \txtbl{anisothermal} system based on a discretisation using the Scott--Vogelius pair for the velocity and pressure, which has the advantage of preserving the divergence constraint exactly (to machine precision and solver tolerances). This builds on previous work for the Navier--Stokes system \cite{Farrell2020} and a stress-velocity-pressure formulation for isothermal \txtb{flow of} non-Newtonian fluids with implicit rheology \cite{FarrellGazca2019}. An augmented Lagrangian-based preconditioner (AL) for buoyancy-driven flow was already presented in \cite{Ke2018} for a stabilised $\mathbb{P}_1$--$\mathbb{P}_1$ velocity-pressure pair, in which the augmented velocity block was substituted by $A + \gamma B^\top \mathrm{diag}(M_p)^{-1}B$ and handled by GMRES preconditioned with algebraic multigrid; in that work it was shown that the AL preconditioner performed better than non-augmented variants, at least for Prandtl and Rayleigh numbers in the ranges $0.04\leq \Pr \leq 1$, $500\leq \Ra \leq 10000$. Numerical experiments with the preconditioner will show good performance with the Navier--Stokes and power-law models for a wider range of non-dimensional numbers, even with temperature-dependent viscosity, heat conductivity, and viscous dissipation. It is remarkable that the robustness properties of the preconditioner hold in these cases, given that the available parameter-robust multigrid theory pioneered by Sch\"oberl does not apply, since the block $A$ is non-symmetric and \txtb{possibly} non-coercive.

\section{Preliminaries}
\subsection{Function spaces}
Throughout this work we will employ standard notation for Sobolev and Lebesgue spaces (e.g.\ $(W^{k,s}(\Omega),\|\cdot\|_{W^{k,s}(\Omega)})$ and $(\Lp{q},\|\cdot\|_{\Lp{q}})$). The space $W^{k,r}_0(\Omega)$, for $r\in[1,\infty)$, is defined as the closure of the space of smooth functions with compact support $C_0^\infty(\Omega)$ with respect to $\|\cdot\|_{k,r}$; its dual space will be denoted by $W^{-1,r'}(\Omega)$, where $r'$ is the H\"{o}lder-conjugate of the number $r$, i.e.\ $1/r' + 1/r = 1$. When $r=2$ we will write $W^{k,2}(\Omega) = H^{k}(\Omega)$ and $\Wmonep[]{2} = \Hmone[]$. Let us also define the following useful subspaces for $r>1$ and $\Gamma \subset \partial\Omega$:
\begingroup
\allowdisplaybreaks
\begin{gather*}
	\Lmean{r} := \left\{q \in L^r(\Omega) \colon \int_\Omega q = 0 \right\},\\
	\Wonezdiv{r} := \overline{\{\bv \in \Cinfd\, :\, \diver\bv =0\}}^{\|\cdot\|_{W^{1,r}(\Omega)}},\\
	\Sobs{1}{r}{\Gamma}{} := \overline{\{w \in C^\infty(\Omega)\, :\, w|_{\Gamma}=0 \}}^{\|\cdot\|_{W^{1,r}(\Omega)}},\\
\Lsym{r} := \{\btau \in {\Lp{r}}^{d\times d}\colon \btau^\top = \btau \},\\
\Lsymtr{r} := \{\btau \in \Lsym{r} \colon \tr{\btau} = 0 \},\\
\Wtrzero{r} := \{w|_{\Gamma} \colon w\in \Wonep[]{r},\, w=0\text{ on }\partial\Omega\setminus \overline{\Gamma}   \}.
\end{gather*}
\endgroup
The operator $\tr$ in the definition of $\Lsymtr{r}$ denotes the trace of a $d\times d$ matrix. The letter $c$ will be used in various estimates to denote a generic positive constant whose value might change from line to line (the dependence on the parameters will be made explicit whenever necessary).

\subsection{Implicit constitutive relations} \label{sec:implcr}
Let us assume now that the material parameters $\hat{\mu},\hat{\tau},\hat{\sigma},\hat{\kappa}$ are continuous functions of one variable such that
\begin{equation}
\begin{gathered}\label{eq:assumptions_materialparameters}
	0 \leq \hat{\tau}(s),\hat{\sigma}(s) \leq c_0,\\
	c_1 \leq \hat{\mu}(s),\hat{\kappa}(s) \leq c_2,\\
	\hat{\tau}(s)\hat{\sigma}(s) =0,
\end{gathered}
\end{equation}
for all $s\in\R$, and some positive constants $c_0,c_1,c_2$. It is not difficult to show that under these assumptions, the relation \eqref{eq:implicit_CR} defines a monotone and coercive 2-graph \cite[Lemma 3]{Maringova2018}.

\begin{lemma}\label{lm:monotonicity_coercivity_temp}
	Let $\BG\colon \Rds \times \Rds \times \R \to \R$ be the function defined by \eqref{eq:implicit_CR} and suppose that $\hat{\mu},\hat{\tau},\hat{\sigma}\in C(\R)$ satisfy \eqref{eq:assumptions_materialparameters}. Then there exist two constants $\alpha,\beta>0$ such that
	\begin{equation}\label{eq:monotonicity_temp}
		\BS\twodot\BD \geq \alpha (|\BS|^2 + |\BD|^2) - \beta
	\end{equation}
	\begin{equation}
		(\BS - \overline{\BS})\twodot (\BD - \overline{\BD}) \geq 0,
	\end{equation}
	for any $(\BS,\BD,\theta),(\overline{\BS},\overline{\BD},\theta)\in \Rds\times\Rds\times \R$ such that $\BG(\BS,\BD,\theta)=\bm{0}=\BG(\overline{\BS},\overline{\BD},\theta)$.
\end{lemma}

In the same spirit as \cite{Diening:2013,Farrell2019}, in the numerical scheme we will employ a sequence of continuous explicit approximations of the implicit constitutive relation \eqref{eq:implicit_CR}. Let us define for $n\in\mathbb{N}$ the approximations as follows:
\begin{equation}
	\begin{split}\label{eq:approximate_constrel}
	\Dcr^n(\BS,\theta) &:= \min \left\{n + \frac{1}{2\hat{\mu}(\theta)},
	\frac{\frac{1}{2\hat{\mu}(\theta)} (|\BS| - \hat{\tau}(\theta))^+ + \hat{\sigma}(\theta)}{|\BS|}\right\},\\
	\Scr^n(\BD,\theta) &:= \min \left\{n + 2\hat{\mu}(\theta),
	\frac{2\hat{\mu}(\theta) (|\BD| - \hat{\sigma}(\theta))^+ + \hat{\tau}(\theta)}{|\BD|}\right\}.
	\end{split}
\end{equation}
Either of the two can be chosen, depending on whether one wishes to consider explicit approximations of the \txtb{deviatoric} stress in terms of the symmetric velocity gradient or vice-versa. The functions $\Dcr^n$ and $\Scr^n$ satisfy the same monotonicity and coercivity conditions as those stated in Lemma \ref{lm:monotonicity_coercivity_temp}, uniformly in $n$. More importantly, the following localised Minty's lemma is available for these approximations, which will be useful when proving that the limit of the numerical approximations satisfies the constitutive relation.

\begin{lemma}[\cite{Maringova2018}, Lemma 6]\label{lm:minty}
	Let $M\subset \Omega$ be measurable and let $\BG$ be defined by \eqref{eq:implicit_CR}. Now suppose that $\{\BD^n\}_{\mathbb{N}}$ and $\{\theta^n\}_{\mathbb{N}}$ are sequences of measurable functions on $\Omega$ and let $\BS^n:= \Scr^n(\BD^n,\theta^n)$. Assume that the following conditions hold:
	\begin{align*}
		\Scr^n(\BD^n,&\theta^n) = \bm{0} && \text{a.e. in } M,\\
		\BS^n &\rightharpoonup \BS  && \text{weakly in }L^2(M)^{d\times d},\\
		\BD^n &\rightharpoonup \BD && \text{weakly in }L^2(M)^{d\times d},\\
\theta^n &\rightharpoonup \theta && \text{a.e. in }M,\\
\limsup_{n\to\infty} \int_M \BS^n  \twodot &\BD^n \leq \int_M \BS\twodot \BD. &&
	\end{align*}
Then $\BG(\BS,\BD,\theta)=\bm{0}$ and $\BS^n\twodot\BD^n \rightharpoonup \BS\twodot\BD$ weakly in $L^1(M)$. An analogous statement holds for $\Dcr^n$.
\end{lemma}

If the rheological parameters are constant (i.e.\ do not depend on the temperature), it is possible to generalise the convergence result to cover implicit relations of the form $\BH(\cdot,\BS,\BD(\bu)) = \bm{0}$, where $\BH\colon \Omega\times \Rds\times \Rds\to \Rds$, that satisfy the coercivity condition \eqref{eq:monotonicity_temp} with an exponent other than 2; this would for instance capture the Herschel--Bulkley constitutive relation. For convenience, the assumptions will be written in terms of the graph induced by $\BH$, which is defined in the standard way:
\begin{equation*}
	(\BD,\BS)\in \mathcal{A}(\cdot) \Longleftrightarrow \BH(\cdot,\BS,\BD) = \bm{0}.
\end{equation*}

\begin{assumption}\label{as:rgraph}
	The graph $\mathcal{A}$ is a maximal monotone $r$-graph for some $r>\frac{2d}{d+2}$. More precisely, the following properties hold for almost every $x\in\Omega$:
\begin{itemize}[leftmargin = 0.8cm]
	\item ($\mathcal{A}$ contains the origin). $(\bm{0},\bm{0})\in \mathcal{A}(x)$;
	\item ($\mathcal{A}$ is a monotone graph). For every $(\BD_1,\BS_1),(\BD_2,\BS_2)\in\mathcal{A}(x)$,
		\begin{equation*}
			(\BS_1 - \BS_2)\twodot(\BD_1 - \BD_2) \geq 0;
		\end{equation*}
	\item ($\mathcal{A}$ is maximal monotone). If $(\BD,\BS)\in \Rds\times \Rds$ is such that
		\begin{equation*}
			(\hat{\BS}-\BS)\twodot(\hat{\BD}-\BD)\geq 0
			\quad \text{for all }(\hat{\BD},\hat{\BS})\in\mathcal{A}(x),
		\end{equation*}
		then $(\BD,\BS)\in\mathcal{A}(x)$;
	\item  ($\mathcal{A}$ is an $r$-graph). There is a non-negative function $m\in L^1(\Omega)$ and a constant $c>0$ such that
		\begin{equation*}
			\BS\colon \BD \geq -m + c(|\BD|^r + |\BS|^{r'})\quad \text{for all }(\BD,\BS)\in \mathcal{A}(x);
		\end{equation*}
	\item (Measurability). The set-valued map $x\mapsto \mathcal{A}$ is $\mathcal{L}(\Omega)$--$(\mathcal{B}(\Rds)\otimes \mathcal{B}(\Rds))$ measurable; here $\mathcal{L}(\Omega)$ denotes the family of Lebesgue measurable subsets of $\Omega$ and $\mathcal{B}$ is the family of Borel subsets of $\Rds$;
	\item (Compatibility). For any $(\BD,\BS)\in\mathcal{A}(x)$ we have that
		\begin{equation*}
			\tr(\BD)=0\Longleftrightarrow \tr(\BS)=0.
		\end{equation*}
\end{itemize}
\end{assumption}
%}

\subsection{Finite element spaces}\label{sc:FEM_spaces}
Let $\{\mathcal{T}_n\}_{n\in \mathbb{N}}$ be a family of shape-regular triangulations such that the mesh size $h_n := \max_{K\in \mathcal{T}_n} h_K$ tends to zero as $n\to \infty$, where $h_K$ denotes the diameter of an element $K\in\mathcal{T}_n$. We define the following conforming families of finite element spaces:
\begin{align*}
	\Sigma^n &:= \left\{\bsigma \in \Lsym{\infty} \colon \bsigma|_K \in \mathbb{P}_{\mathbb{S}}(K)^{d\times d},\, K\in\mathcal{T}_n \right\}, \\
V^n &:= \left\{\bv \in \Sobs{1}{\infty}{0}{d} \colon \bv|_K \in \mathbb{P}_{\mathbb{V}}(K)^d,\, K\in\mathcal{T}_n \right\}, \\
M^n &:= \left\{q \in \Lp{\infty} \colon q|_K \in \mathbb{P}_{\mathbb{M}}(K),\, K\in\mathcal{T}_n \right\}, \\
U^n &:= \left\{w \in \Sobs{1}{\infty}{\Gamma_D}{} \colon w|_K \in \mathbb{P}_{\mathbb{U}}(K),\, K\in\mathcal{T}_n \right\},
\end{align*}
where $\mathbb{P}_{\mathbb{S}}(K),\mathbb{P}_{\mathbb{V}}(K),\mathbb{P}_{\mathbb{M}}(K),\mathbb{P}_{\mathbb{U}}(K)$ are spaces of polynomials on the element $K\in \mathcal{T}_n$. It will be convenient to define the following subspaces:
\begin{gather*}
	M_0^n := M^n \cap \Lmean{2},% \quad \Signsym := \Sigma^n \cap \Lsym{r'},
	\quad \Signtr := \Sigma^n \cap L^2_{\sym,\mathrm{tr}}(\Omega)^{d\times d}, \\
	\Vndiv := \left\{\bv \in V^n \colon \int_\Omega q\diver \bv = 0\quad \forall q\in M^n \right\}.
\end{gather*}

\begin{assumption}[Approximability]\label{as:approximability}
For every $s\in [1,\infty)$ we have that
\begin{align*}
	\inf_{\overline{\bm{v}}\in V^n}\|\bm{v}-\overline{\bm{v}}\|_{W^{1,s}(\Omega)} &\rightarrow 0 \quad \text{ as }n\rightarrow\infty\quad \forall\,\bm{v}\in W^{1,s}_0(\Omega)^d,\\
	\inf_{\overline{q}\in M^n}\|q-\overline{q}\|_{L^{s}(\Omega)} &\rightarrow 0 \quad \text{ as }n\rightarrow\infty\quad \forall\,q\in L^s(\Omega),\\
	\inf_{\overline{\bm{\sigma}}\in \Sigma^n}\|\bm{\sigma}-\overline{\bm{\sigma}}\|_{L^{s}(\Omega)} &\rightarrow 0 \quad \text{ as }n\rightarrow\infty\quad \forall\,\bm{\sigma}\in L^s(\Omega)^{d\times d},\\
	\inf_{\overline{w}\in U^n}\|w-\overline{w}\|_{W^{1,s}(\Omega)} &\rightarrow 0 \quad \text{ as }n\rightarrow\infty\quad \forall\,w\in W^{1,s}_{\Gamma_D}(\Omega).
\end{align*}
\end{assumption}
\begin{assumption}[Fortin Projector $\Pi^n_\Sigma$]\label{as:Projector_Stress}
For each $n\in\mathbb{N}$ there is a linear projector $\Pi^n_\Sigma \colon L_{\sym}^1(\Omega)^{d\times d}\to \Sigma^n$ such that:
\begin{itemize}[leftmargin = 0.8cm]
	\item (Preservation of divergence). For any $\bm{\sigma}\in \Lsym{1}$ we have that
\begin{equation*}
\int_\Omega \bm{\sigma}:\BD(\bv) = \int_\Omega \Pi^n_\Sigma(\bm{\sigma}):\BD(\bm{v}) \quad \forall\, \bm{v}\in \Vndiv.
\end{equation*}
\item ($L^{s}$--stability). For every $s\in (1,\infty)$ there is a constant $c>0$, independent of $n$, such that:
\begin{equation*}
	\|\Pi^n_\Sigma \bm{\sigma}\|_{L^{s}(\Omega)} \leq c \|\bm{\sigma}\|_{L^{s}(\Omega)}\qquad \forall\, \bm{\sigma}\in \Lsym{s}.
\end{equation*}
\end{itemize}
\end{assumption}
\begin{assumption}[Fortin Projector $\Pi^n_V$]\label{as:Projector_Velocity}
For each $n\in\mathbb{N}$ there is a linear projector $\Pi^n_V :W^{1,1}_0(\Omega)^d\rightarrow V^n$ such that the following properties hold:
\begin{itemize}[leftmargin = 0.8cm]
\item (Preservation of divergence). For any $\bm{v}\in W^{1,1}_0(\Omega)^d$ we have that
\begin{equation*}
\int_\Omega q\,\diver\bv = \int_\Omega q\,\diver(\Pi^n_V\bv) \quad\, \forall\, q\in M^n.
\end{equation*}
\item ($W^{1,s}$--stability). For every $s\in (1,\infty)$ there is a constant $c>0$, independent of $n$, such that:
\begin{equation*}
\|\Pi^n_V\bm{v}\|_{W^{1,s}(\Omega)} \leq c \|\bm{v}\|_{W^{1,s}(\Omega)}\qquad \forall \,\bm{v}\in W^{1,s}_0(\Omega)^d.
\end{equation*}
\end{itemize}
\end{assumption}
\begin{assumption}[Projectors $\Pi^n_M, \Pi^n_U$]\label{as:Projectors_PressureTemp}
	For each $n\in\mathbb{N}$ there is a linear projector $\Pi^n_M :L^1(\Omega)\rightarrow M^n$ and a linear projector $\Pi^n_U\colon W^{1,1}_{\Gamma_D}(\Omega) \to U^n$ such that for all $s\in (1,\infty)$ there is a constant $c>0$, independent of $n$, such that:
\begin{align*}
	\|\Pi^n_M q\|_{L^{s}(\Omega)} &\leq c \|q\|_{L^{s}(\Omega)} && \forall\, q\in L^s(\Omega),\\
	\|\Pi^n_U w\|_{L^{s}(\Omega)} &\leq c \|w\|_{W^{1,s}(\Omega)} && \forall\, w\in W^{1,s}_{\Gamma_D}(\Omega).
\end{align*}
\end{assumption}

The stability and approximability assumptions above imply immediately that for any $s\in[1,\infty)$ we have:
\begin{align}\label{eq:fem_projector_convergence}
	\begin{split}
		\|\bm{\sigma} - \Pi^n_\Sigma\bm{\sigma}\|_{L^{s}(\Omega)} & \rightarrow 0\quad \text{ as }n\rightarrow\infty\quad \forall\, \bm{\sigma}\in L_\text{sym}^s(\Omega)^{d\times d},\\
		\|\bm{v} - \Pi^n_V\bm{v}\|_{W^{1,s}(\Omega)} &\rightarrow 0\quad \text{ as }n\rightarrow\infty\quad \forall \,\bm{v}\in W_0^{1,s}(\Omega)^{d},\\
		\|q - \Pi^n_Mq\|_{L^{s}(\Omega)} &\rightarrow 0\quad \text{ as }n\rightarrow\infty\quad \forall\, q\in L^{s}(\Omega),\\
		\|w - \Pi^n_Uw\|_{W^{1,s}(\Omega)} &\rightarrow 0\quad \text{ as }n\rightarrow\infty\quad \forall\, w\in W^{1,s}_{\Gamma_D}(\Omega).
\end{split}
\end{align}
In addition, the assumptions guarantee that the velocity-pressure and stress-velocity pairs are inf-sup stable: for any $s\in (1,\infty)$ there are two constants $\beta_s,\gamma_s>0$, independent of $n$, such that the following inf-sup conditions are satisfied:
\begin{align}
	\adjustlimits \inf_{q\in M^n\setminus\{0\}}\sup_{\bv\in V^n\setminus\{0\}} \frac{\int_\Omega q\,\diver\bv}{\|\bv\|_{W^{1,s}(\Omega)}\|q\|_{L^{s'}(\Omega)}} \geq \beta_s,\label{eq:infsupvel}\\
	\adjustlimits\inf_{\bv\in V^n_{\diver}\setminus\{0\}}\sup_{\btau\in \Sigma^n_{\sym}\setminus\{0\}}
\frac{\int_\Omega \bm{\tau}\twodot\BD(\bm{v})}{\|\bm{\tau}\|_{L^{s'}(\Omega)}\|\bm{v}\|_{W^{1,s}(\Omega)}} \geq \gamma_s.\label{eq:infsupstress}
\end{align}

In the literature there are several well-known examples of velocity-pressure pairs $V^n$--$M^n$ that satisfy the approximability and stability assumptions above. They include, among others, the MINI element, the Taylor--Hood element $\mathbb{P}_k$--$\mathbb{P}_{k-1}$, and the conforming Crouzeix--Raviart element (see e.g.\ \cite{Boffi2013,Girault1986,Crouzeix1973}). The Scott--Vogelius pair $\mathbb{P}_k$--$\mathbb{P}_{k-1}^{\mathrm{disc}}$ is another example that in addition has the remarkable property that discretely divergence-free functions are also pointwise divergence-free \cite{Scott1985}. This element can be shown to be inf-sup stable for instance on barycentrically refined meshes \cite{Qin1994,Zhang2005}, and the preconditioner to be introduced in Section \ref{sec:AL_preconditioner} will be based on a discretisation using this pair. As for the stress variable, if the velocity space consists of continuous piecewise polynomials of degree $k$ (as is the case of the Scott--Vogelius element), then a space satisfying Assumption \ref{as:Projector_Stress} is \cite{Farrell2019}:
\begin{equation}\label{eq:space_discrete_stresses}
	\Sigma^n = \{\bm{\sigma}\in \Lsym{\infty} \, :\, \bm{\sigma}|_K\in\mathbb{P}_{k-1}(K)^{d\times d},\text{ for all }K\in \mathcal{T}_n\}.
\end{equation}
The space of discrete temperatures $U^n$ is not required to satisfy any inf-sup stability conditions, and so it suffices to choose any $H^1$-conforming space for which the expected order of accuracy is consistent with that of the other variables.

\subsection{Convective term}
A useful property in the analysis of systems describing incompressible fluids is that the convective term vanishes when testing with the divergence-free velocity itself. This is a consequence of the identity
\begin{equation}
	-\int_\Omega (\bv\otimes \bv)\twodot \BD(\bv) = 0
	\quad \text{for all }\bv\in \Cinfd \text{ with }\diver\bv =0.
\end{equation}
Such an identity will not be satisfied in general with only discretely divergence-free elements. In order to recover this cancellation property at the discrete level let us define a skew-symmetric form of the convective term as follows:
\begin{equation*}
\renewcommand{\arraystretch}{1.5}
\mathcal{B}(\bm{u},\bm{v},\bm{w}) := \left\{
\begin{array}{cc}
- \displaystyle\int_\Omega\bm{u}\otimes\bm{v}\twodot\nabla \bm{w}, & \textrm{ if } \Vndiv \subset W^{1,1}_{0,\textrm{div}}(\Omega)^d,\\
  \displaystyle\frac{1}{2}\int_\Omega  \bm{u}\otimes\bm{w}\twodot\nabla\bm{v}-\bm{u}\otimes\bm{v}\twodot\nabla \bm{w}, & \textrm{ otherwise}.\\
\end{array}
\right.
\end{equation*}
This new trilinear form now satisfies $\mathcal{B}(\bv,\bv,\bv)=0$ for any $\bv\in W^{1,\infty}_0(\Omega)^d$, regardless of whether $\bv$ is divergence-free or not, and it reduces to the original trilinear form $-\int_\Omega (\bu\otimes \bw)\twodot \nabla \bw$ if $\diver \bv=0$.

Let us now define
\begin{equation*}
	\tilde{r}:= \min \{r', r^*/2\}, \text{ where }
	r^* := \left\{\begin{array}{cc}\frac{dr}{d-r} & \text{if }r<d,\\
	 \infty, & \text{otherwise}.\end{array} \right.
\end{equation*}
Observe that the condition $\tilde{r}>1$ is equivalent to $r> \frac{2d}{d+2}$, which is the natural condition required to have a well-defined weak form of the convective term, because it ensures that $\Wonep{r}  \hookrightarrow L^2(\Omega)^d$. In this case, for exactly divergence-free functions $\bu,\bv,\bw\in \Vndiv$ one has that
\begin{equation}
	|\mathcal{B}(\bu, \bv, \bw)| \leq \int_\Omega |\bu\otimes\bv\twodot \nabla \bw|
	\leq c \|\bu\|_{\Wonep[]{r}}\|\bv\|_{\Wonep[]{r}}\|\bw\|_{\Wonep[]{\tilde{r}'}}.
\end{equation}
Otherwise one needs the stronger assumption $r>\frac{2d}{d+1}$; this ensures that there is an $s\in (1,\infty)$ such that $\frac{1}{r} + \frac{1}{2\tilde{r}} + \frac{1}{s} = 1$ and so (c.f.\ \cite{Diening:2013})
\begin{equation}\label{eq:bound_convective_skew}
	\begin{split}
		\int_\Omega |\bu \otimes \bw \twodot\nabla \bv|& \leq \|\bu\|_{\Lp{2\tilde{r}}} \|\bv\|_{\Wonep[]{r}} \|\bw\|_{\Lp{s}}\\
							       &	\leq c \|\bu\|_{\Wonep[]{r}}\|\bv\|_{\Wonep[]{r}}\|\bw\|_{\Wonep[]{\tilde{r}'}},
\end{split}
\end{equation}
for any $\bu,\bv\in \Sobs{1}{r}{}{d},\bw\in \Sobs{1}{\tilde{r}'}{}{d}$. Thus we deduce that the trilinear form $\mathcal{B}(\cdot,\cdot,\cdot)$ is bounded on $\Wonep{r}\times \Wonep{r}\times \Wonep{\tilde{r}'}$ if $r>\frac{2d}{d+2}$ when using exactly divergence-free elements and if $r>\frac{2d}{d+1}$ otherwise. This does not pose a problem when working with the constitutive relation \eqref{eq:implicit_CR} (for which $r=2$), but for relations with more general $r$-growth the more demanding requirement that $r>\frac{2d}{d+1}$ would impose a restriction on the convergence result that can be obtained (see \cite[Thm.\ 18]{Diening:2013}). In order to circumvent this issue we shall make use of a reconstruction operator.

\begin{assumption}[Reconstruction operator $\pi^n$]\label{as:reconstruction_operator}
	Let $X^n$ be an auxiliary $H(\diver;\Omega)$-conforming finite element space. There exists a map $\pi^n\colon \Wonep{1}\to V^n + X^n$ (usually called a reconstruction operator) that satisfies:
\begin{itemize}[leftmargin = 0.8cm]
	\item (Preservation of Divergence). If $\bv \in \Vndiv$ then $\diver (\pi^n \bv) = 0$ pointwise.
	\item (Consistency). For every $\bv\in V^n$ and $K\in \mathcal{T}_n$ it holds that
		\begin{equation*}
			\|\bv - \pi^n \bv\|_{L^s(K)}\leq c h_K^m |\bv|_{W^{m,s}(K)}, \quad \text{ for }s\in[1,\infty),\, m\in \{0,1,2\}.
		\end{equation*}
\end{itemize}
\end{assumption}

Operators with the properties described above have been constructed in \cite{Linke2014,Linke2016a,Linke2016,Linke2017,John2017} for elements with discontinuous pressures; the construction is based on the interpolation operators associated with the Raviart--Thomas and Brezzi--Douglas--Marini elements. A slightly more complicated construction for elements with continuous pressures was introduced in \cite{Lederer2017}; however, this construction is computationally expensive and so might not be advantageous in practice. These reconstruction operators have been employed to obtain pressure-robust discretisations by ``repairing'' the $L^2$-orthogonality between discretely divergence-free functions and gradient fields; see \cite{John2017} for more details. In order to exploit the advantages of this framework one has to replace the $L^2$ inner products in the discrete formulation in the following way:
\begin{equation}
\int_\Omega \bw\cdot\bv  \mapsto \int_\Omega \bw \cdot \pi^n\bv,
\end{equation}
where $\bv\in V^n$ is a test function. As for the convective term, let us define
\begin{equation}
\renewcommand{\arraystretch}{1.5}
\tilde{\mathcal{B}}_n(\bu,\bv,\bw) := \left\{
\begin{array}{cc}
- \displaystyle\int_\Omega\bm{u}\otimes\bm{v}\twodot\nabla \bm{w}, & \textrm{ if } \Vndiv \subset W^{1,1}_{0,\textrm{div}}(\Omega)^d,\\
 - \int_\Omega \bu\otimes \pi^n \bv\twodot \nabla \bw, & \textrm{ otherwise}.\\
\end{array}
\right.
\end{equation}
From the properties of $\pi^n$ stated in Assumption \ref{as:reconstruction_operator} one readily sees that the trilinear form $\tilde{\mathcal{B}}_n$ is bounded on $\Wonep{r}\times \Wonep{r}\times \Wonep{\tilde{r}'}$, and that $\tilde{\mathcal{B}}_n(\bv,\bv,\bv)=0$ for any $\bv\in \Vndiv$.

For the advective term for the temperature one can analogously define the trilinear form
\begin{equation*}
\renewcommand{\arraystretch}{1.5}
\mathcal{C}(\bm{u},\theta,\eta) := \left\{
\begin{array}{cc}
- \displaystyle\int_\Omega\bm{u}\theta \cdot\nabla \eta, & \textrm{ if } \Vndiv \subset W^{1,1}_{0,\textrm{div}}(\Omega)^d,\\
  \displaystyle\frac{1}{2}\int_\Omega  \bm{u}\eta\cdot\nabla\theta - \bm{u}\theta \cdot \nabla \eta, & \textrm{ otherwise},\\
\end{array}
\right.
\end{equation*}
which is well defined and bounded on $\Wonep{r}\times \Hone[] \times \Wonep[]{\infty}$ assuming that $r>\frac{2d}{d+2}$. In addition, this form satisfies $\mathcal{C}(\bu,\eta,\eta) = 0$ for any $\eta\in \Wonep[]{\infty}$, regardless of whether $\bu$ is divergence-free or not. The form $\mathcal{C}$ does not impose additional restrictions like $\mathcal{B}$ does for small $r$, but a trilinear form using a reconstruction operator $\tilde{\mathcal{C}}_n$  could be used instead (and defined analogously).%, if preferred:

\section{Finite Element Approximation} \label{sec:discretisation}
Let us now set the physical constants to unity for ease of readability (appropriate non-dimensional forms of the system will be employed in Section \ref{sec:experiments}). \txtb{For simplicity the body force $\bm{f}$ is also set to zero; the result however remains valid for any $\bm{f}\in H^{-1}(\Omega)^d$.} Suppose that $\theta_b \in \Hhalfzero[\Gamma_D] := \Sobs[\Gamma_D]{1/2}{2}{00}{} $, and let $\hat{\theta}_b\in \Hone[]$ be such that $\hat{\theta}_b|_{\Gamma_D} = \theta_b$. We can now define the weak formulation of the system (without viscous heating).

\noindent
\textbf{Formulation A$_{\mathbf{0}}$.} Find $(\BS,\theta,\bu,p)\in \Lsymtr{2}\times (\hat{\theta}_b + \SobsH{1}{\Gamma_D}{}) \times \SobsH{1}{0}{d} \times \Lmean{2}$ such that:
\begin{subequations}\label{eq:WeakFormulation}
\begin{alignat}{2}
\int_\Omega \BS:\BD(\bv) -\int_\Omega \bu\otimes\bu&:\BD(\bv) -\int_\Omega p  \diver\bv = \int_\Omega \theta \bv\cdot\be_d \quad & \forall\,\bv\in C^\infty_0(\Omega)^d,\\
	-&\int_\Omega q  \diver\bu = 0 & \forall\, q\in C^\infty_0(\Omega),\\
	\int_\Omega \hat{\kappa}(\theta)\nabla&\theta\cdot \nabla\eta - \bu\theta\cdot\nabla\eta = 0 & \forall\, \eta \in C^\infty_{\Gamma_D}(\Omega),\\
	\BG(&\BS,\BD(\bu),\theta) = \bm{0} & \textrm{a.e. in }\Omega.\label{eq:weak_form_constrel}
\end{alignat}
%For the problem with a constitutive relation independent of the temperature we define Formulation $\textrm{B}_{0}$ by replacing \eqref{eq:weak_form_constrel} with
%\begin{equation}
%	\BH(\cdot,\BS,\BD(\bu)) = \bm{0} \qquad \text{ a.e. in }\Omega.
%\end{equation}
\end{subequations}

Let $\hat{\theta}_b^n$ be the standard Scott--Zhang interpolant of $\hat{\theta}_b$ into $\hat{U}^n$, where $\hat{U}^n$ is the same finite element space as $U^n$, but without strongly imposed boundary conditions. We have everything in place to state the finite element approximation of the problem.

\noindent
\textbf{Formulation A$^\mathbf{n}_{\mathbf{0}}$.} Find $(\theta^n,\bu^n,p^n) \in (\hat{\theta}^n_b + U^n)\times V^n \times M^n_0$ such that:
\begin{subequations}\label{eq:FEFormulation}
\begin{alignat}{2}
	\int_\Omega \Scr^n(\BD(\bu^n),\theta^n)\twodot\BD(\bv) +\mathcal{B}(&\bu^n,\bu^n,\bv) -\int_\Omega p^n  \diver\bv = \int_\Omega \theta^n  \bv\cdot\be_d \quad & \forall\,\bv\in V^n,\\
	-&\int_\Omega q  \diver\bu^n = 0 & \forall\, q\in M^n,\\
	\int_\Omega \hat{\kappa}(\theta^n )\nabla(\theta^n&)\cdot \nabla\eta
	+  \mathcal{C}(\bu^n,\theta^n,\eta)  = 0 & \forall\, \eta \in U^n.
\end{alignat}
\end{subequations}
In case one wishes to compute the \txtb{deviatoric} stress directly, a 4-field formulation may be employed instead. We refer to this formulation as Formulation B$^n_0$. We will prove that the solutions to the discrete formulations A$^n_0$ and B$^n_0$ converge to a weak solution of Formulation A$_0$.

\noindent
\textbf{Formulation B$^\mathbf{n}_{\mathbf{0}}$.} Find $(\BS^n,\theta^n,\bu^n,p^n) \in \Sigma^n\times (\hat{\theta}^n_b + U^n)\times V^n \times M^n_0$ such that:
\begin{subequations}\label{eq:FEFormulation2}
\begin{alignat}{2}
	\int_\Omega (&\Dcr^n(\BS^n,\theta^n ) - \BD(\bu^n))\twodot  \btau =0  &\forall\,\btau \in \Sigma^n,\label{eq:discrete_CR_B0}\\
	\int_\Omega \BS^n\twodot\BD(\bv) +\mathcal{B}(&\bu^n,\bu^n,\bv) -\int_\Omega p^n  \diver\bv = \int_\Omega \theta^n \bv\cdot\be_d \quad & \forall\,\bv\in V^n, \label{eq:discrete_momentum_B0}\\
	&-\int_\Omega q  \diver\bu^n = 0 & \forall\, q\in M^n,\label{eq:discrete_mass_B0}\\
	\int_\Omega \hat{\kappa}(\theta^n )&\nabla\theta^n\cdot \nabla\eta
	+  \mathcal{C}(\bu^n,\theta^n)  = 0 & \forall\, \eta \in U^n.\label{eq:discrete_temp_B0}
\end{alignat}
\end{subequations}
We define Formulations $\tilde{\textrm{A}}^n_0$ and $\tilde{\textrm{B}}^n_0$ as the analogues of the formulations $\textrm{A}_0^n$ and $\textrm{B}_0^n$, respectively, in which we replace $\mathcal{B}$ and  $\mathcal{C}$ by $\tilde{\mathcal{B}}_n$ and $\tilde{\mathcal{C}}_n$. The following lemma asserts that all of these formulations have a solution.

\begin{lemma}\label{lm:existence_discrete}
	Suppose the material parameters satisfy condition \eqref{eq:assumptions_materialparameters} and suppose that $\{U^n,V^n,M^n\}_{n\in\mathbb{N}}$ (respectively $\{\Sigma^n,U^n,V^n,M^n\}_{n\in\mathbb{N}}$) is a family of finite element spaces satisfying Assumptions \ref{as:approximability} and \ref{as:Projector_Velocity}--\ref{as:Projectors_PressureTemp} (resp. \ref{as:approximability}--\ref{as:Projectors_PressureTemp}). In the case of formulations $\tilde{\mathrm{A}}^n_0$ and $\tilde{\mathrm{B}}^n_0$ suppose further that Assumption \ref{as:reconstruction_operator} holds. Then, for every $n\in \mathbb{N}$, Formulations $\mathrm{A}^n_0$ and $\tilde{\mathrm{A}}_0^n$ (resp. $\mathrm{B}^n_0$ and $\tilde{\mathrm{B}}^n_0$) admit a solution $(\theta^n,\bu^n,p^n)\in (\hat{\theta}^n_b + U^n)\times V^n\times M^n_0$ (resp.\  $(\BS^n,\theta^n,\bu^n,p^n) \in \Sigma^n \times (\hat{\theta}^n_b + U^n) \times V^n \times M^n_0$). Moreover, the following \textit{a priori} estimate holds:
\begin{subequations}\label{eq:apriori}
	\begin{equation}
		\|\bu^n\|_{\Hone[]} + \|\theta^n\|_{\Hone[]}
		+ \|p^n\|_{\Lp{2}}  + \|\BS^n\|_{\Lp{2}} \leq c,
	\end{equation}
where the constant $c$ is independent of $n$; we denote $\BS^n := \Scr^n(\BD(\bu^n),\theta^n)$ in the case of Formulations $\mathrm{A}_0^n$ and $\tilde{\mathrm{A}}^n_0$. In addition, for Formulations $\mathrm{B}^n_0$ and $\tilde{\mathrm{B}}^n_0$ we have
\begin{equation}
	\|\Dcr^n(\BS^n,\theta^n)\|_{\Lp{2}} \leq c.
\end{equation}
\end{subequations}
\end{lemma}
\begin{proof}
	We will carry out the proof for Formulation $\mathrm{B}^n_0$; the proof for the other formulations is analogous with some simplifications. The existence proof will make use of a fixed point argument. Let $\theta^n_0$ be an arbitrary nonzero element of $\hat{\theta}^n_b+U^n$ and define, for $j\in\mathbb{N}$, the function $\theta^n_j\in \hat{\theta}^n_b+U^n$ as follows: given $\theta^n_{j-1}$ we first find $(\BS^n_j,\bu^n_j,p^n_j)\in \Sigma^n\times V^n\times M_0^n$ by solving
\begin{subequations}\label{eq:discrete_fixedp1}
\begin{alignat}{2}
	\int_\Omega (&\Dcr^n(\BS^n_j,\theta^n_{j-1}) - \BD(\bu^n_j))\twodot  \btau =0  &\forall\,\btau \in \Sigma^n,\\
% \frac{1}{j}\int_\Omega \BD(\bu_j^n)\colon \BD(\bv) +
	\int_\Omega \left( \frac{1}{j}\BD(\bu^n_j) + \BS^n_j\right)\twodot\BD(\bv) +\mathcal{B}(&\bu^n_j,\bu^n_j,\bv) -\int_\Omega p^n_j  \diver\bv = \int_\Omega \theta^n_{j-1} \bv\cdot\be_d \: & \forall\,\bv\in V^n,\\
	&-\int_\Omega q  \diver\bu^n_j = 0 & \forall\, q\in M^n,
\end{alignat}
\end{subequations}
and then $\theta^n_j$ is defined as $\hat{\theta}^n_b + \tilde{\theta}^n_j$, where $\tilde{\theta}^n_j\in U^n$ is the solution of the nonlinear problem
\begin{equation}\label{eq:discrete_fixedp2}
	\int_\Omega \hat{\kappa}(\tilde{\theta}^n_j + \hat{\theta}^n_b)\nabla(\tilde{\theta}^n_j + \hat{\theta}^n_b)\cdot \nabla\eta
	+  \mathcal{C}(\bu^n_j,\tilde{\theta}^n_j + \hat{\theta}^n_b,\eta)
 = 0 \qquad \forall\, \eta \in U^n.
\end{equation}
In order to show that the problem \eqref{eq:discrete_fixedp1} is well-posed, let us define a mapping $F^n_j\colon \Sigma^n\times \Vndiv \to (\Sigma^n \times \Vndiv)^*$ by
\begin{align*}
	\langle F^n_j(\bsigma,\bv); (\btau,\bw)\rangle :=
	&\int_\Omega  ( \Dcr^n(\bsigma,\theta^n_{j-1})\twodot \btau
	- \BD(\bv)\twodot\btau
	+ \frac{1}{j}\BD(\bv)\twodot\BD(\bw) \\
						       &	+ \bsigma\twodot \BD(\bw)
+						       \mathcal{B}(\bv,\bv,\bw)
- \theta^n_j \bv\cdot \be_d).
\end{align*}
By using the coercivity of $\Dcr^n$ and the fact that $\mathcal{B}(\bv,\bv,\bv)=0$, one obtains using the inequalities of Young, Korn and Poincar\'{e} that there exists a $\delta(j)>0$ such that
\begin{equation*}
	\langle F^n_j(\bsigma,\bv),(\bsigma,\bv) \rangle >0 \quad \text{ if } \quad
	\|(\bsigma,\bv)\| = \delta(j).
\end{equation*}
A corollary of Brouwer's fixed point theorem \cite[Ch.\ 4, Cor.\ 1.1]{Girault1986} guarantees the existence of functions $(\BS^n_j,\bu^n_j)\in \Sigma^n\times \Vndiv$ satisfying $F^n_j(\BS^n_j,\bu^n_j)=0$ (which is equivalent to \eqref{eq:discrete_fixedp1} with divergence-free test functions) and such that $\|(\BS^n_j,\bu^n_j)\|\leq \delta(j)$. The existence of  $p^n_j\in M^n_0$ then follows from the inf-sup condition \eqref{eq:infsupvel}. A similar argument can be used to prove the well-posedness of the problem \eqref{eq:discrete_fixedp2}.

Now, the inf-sup condition \eqref{eq:infsupstress} and the discrete form of the constitutive relation \eqref{eq:discrete_CR_B0} allow us to control, uniformly in $j$ and $n$, the norm of the velocity in terms of the stress:
\begin{equation}
	\gamma_2 \|\bu^n_j\|_{\Hone[]} \leq \|\BS^n_j\|_{\Lp{2}}.
\end{equation}
Therefore, testing \eqref{eq:discrete_fixedp1} with $(\BS^n_j,\bu^n_k,p^n_j)$ yields the estimate
\begin{equation}\label{eq:apriori_fixedp}
	\|\Dcr^n(\BS^n_j,\theta^n_{j-1})\|^2_{\Lp{2}}
	+ \|\BS^n_j\|^2_{\Lp{2}}
	+ \|\bu^n_j\|^2_{\Hone[]}
	\leq  c \|\theta^n_{j-1}\|^2_{\Lp{2}},
\end{equation}
where $c>0$ is independent of $j$ and $n$. The inf-sup condition \eqref{eq:infsupvel} and the discrete momentum equation in turn imply an estimate for the pressure:
\begin{equation}
	\|p^n_j\|^2_{\Lp{2}} \leq c\|\theta^n_{j-1}\|^2_{\Lp{2}}.
\end{equation}
Furthermore, testing \eqref{eq:discrete_fixedp2} with $\theta^n_j-\hat{\theta}^n_b$ results in
\begin{equation}
	\|\theta^n_j\|^2_{\Hone[]} \leq c \|\bu^n_j\|^2_{\Hone[]}.
\end{equation}
Hence, up to a subsequence, we have as $j\to\infty$ that
\begin{align}
	\Dcr^n(\BS^n_j,\theta^n_{j-1}) &\rightharpoonup \overline{\BD}^n && \text{weakly in }\Lsym{2}, \notag \\
	\BS^n_j &\to \BS^n && \text{strongly in }\Lsym{2}, \notag \\
	\bu^n_j &\to \bu^n && \text{strongly in }\Hone, \\
	p^n_j &\to p^n && \text{strongly in }\Lp{2}, \notag \\
	\theta^n_j &\to \theta^n && \text{strongly in }\Hone[], \notag
\end{align}
where we used the fact that weak and strong convergence are equivalent in finite-dimensional spaces. Since $\Dcr^n$ is continuous and the convergences are strong, one can straightforwardly identify $\overline{\BD}^n = \Dcr^n(\BS^n,\theta^n)$ and pass to the limit to show that $(\BS^n,\theta^n,\bu^n,p^n)$ solve Formulation $\mathrm{B}^n_0$. Now, testing Formulation $\mathrm{B}^n_0$ with $(\BS^n,\theta^n-\hat{\theta}^n_b,\bu^n,p^n)$ allows one to obtain the estimate \eqref{eq:apriori}. Note that the inf-sup conditions were essential to obtain estimates that are uniform in $n$.
\end{proof}

Having shown that the discrete problems admit solutions, we now consider the question of convergence.

\begin{theorem}\label{thm:convergence_temp_dependent}
	Suppose the same assumptions as in  Lemma \ref{lm:existence_discrete} hold and suppose that $\{(\theta^n,\bu^n,p^n)\}_{n\in\mathbb{N}}$ (respectively $(\{\BS^n,\theta^n,\bu^n,p^n\}_{\mathbb{N}})$) is a sequence of solutions of Formulation $\mathrm{A}^n_0$ or $\tilde{\mathrm{A}}^n_0$ (resp.\ Formulation $\mathrm{B}^n_0$ or $\tilde{\mathrm{B}}^n_0$). Then there exists a solution $(\BS,\theta,\bu,p)\in \Lsymtr{2}\times (\hat{\theta}_b+ H^1_{\Gamma_D}(\Omega)) \times \SobsH{1}{0}{d} \times \Lmean{2}$ of Formulation $\mathrm{A}_0$ such that, up to a subsequence, as $n\to\infty$:
	\begin{equation}\label{eq:convergences}
	\begin{split}
	\BS^n &\rightharpoonup \BS \qquad \text{weakly in }\Lsym{2}, \\
	\bu^n &\rightharpoonup \bu \qquad  \text{weakly in }\Hone,\\
	p^n &\rightharpoonup p \qquad \text{weakly in }\Lp{2},  \\
	\theta^n &\rightharpoonup \theta \qquad \text{weakly in }\Hone[],
\end{split}
\end{equation}
where in the case of Formulations $\mathrm{A}^n_0$ and $\tilde{\mathrm{A}}^n_0$ we denote $\BS^n:= \Scr^n(\BD(\bu^n),\theta^n)$.
\end{theorem}
\begin{proof}
	We will once again focus on Formulation $\mathrm{B}^n_0$, since the other cases are completely analogous. From the \textit{a priori} estimate \eqref{eq:apriori} and the fact that $\hat{\theta}^n_b \to \hat{\theta}_b$ in $H^1(\Omega)$, we immediately obtain the convergences \eqref{eq:convergences} (for a not relabelled subsequence) for some $(\BS,\theta,\bu,p) \in \Lsym{2} \times (\hat{\theta}_b + H^1_{\Gamma_D}(\Omega)) \times \SobsH{1}{0}{d} \times \Lmean{2}$, and that
	\begin{equation}
		\Dcr^n(\BS^n,\theta^n) \rightharpoonup \overline{\BD}\quad \text{weakly in }\Lsym{2}.
	\end{equation}
All that is left to prove is that the limiting functions are a solution of Formulation $\mathrm{A}_0$.

Let $\btau\in \Lsym{2}$ be arbitrary. Then \eqref{eq:convergences} and \eqref{eq:fem_projector_convergence} result in
\begin{equation}
	0 = \int_\Omega (\Dcr^n(\BS^n,\theta^n)-\BD(\bu^n))\twodot \Pi^n_{\Sigma}\btau \xrightarrow[\: n\to\infty \:]{}
	\int_\Omega (\overline{\BD} - \BD(\bu))\twodot \btau ,
\end{equation}
and therefore $\overline{\BD}=\BD(\bu)$ almost everywhere. Similarly, for an arbitrary $q\in \Lmean{2}$ one obtains that
\begin{equation}
	0 = \int_\Omega \diver\bu^n \,\Pi^n_M q \xrightarrow[\: n\to \infty \:]{}
	\int_\Omega \diver \bu \, q,
\end{equation}
and so $\bu$ is pointwise divergence-free. One can pass to the limit in \eqref{eq:discrete_momentum_B0} and \eqref{eq:discrete_temp_B0}  in a similar manner, but perhaps the convective terms are worth looking at in more detail. To that end, first note that the Sobolev embedding theorem ensures that (up to a subsequence) we have, for any $p\in [1,2^*)$,
\begin{align}\label{eq:strong_convergence}
	\bu^n &\to \bu &&\text{strongly in }{\Lp{p}}^d, \notag \\
	\theta^n &\to \theta &&\text{strongly in }{\Lp{p}},\\
	\theta^n &\to \theta &&\text{a.e. in }\Omega. \notag
\end{align}
The strong convergence of $\bu^n$ suffices to prove that, for an arbitrary $\bv\in H_0^1(\Omega)^d$:
\begin{equation}
\mathcal{B}(\bu^n,\bu^n,\Pi^n_V\bv) \xrightarrow[\: n\to\infty \:]{}
\frac{1}{2}\int_\Omega \bu \otimes \bv\twodot \nabla \bu
- \bu\otimes \bu \twodot \nabla\bv
= -\int_\Omega \bu\otimes \bu \twodot \BD\bv,
\end{equation}
where the last equality is a consequence of the fact that $\diver\bu = 0$.
Now, from testing the discrete momentum equation with $\bu^n$ and taking \eqref{eq:strong_convergence} into account we observe that
\begin{equation}\label{eq:limsup_temp_dependent}
	\limsup_{n\to\infty} \int_\Omega \BS^n\twodot\BD(\bu^n)
	= \lim_{n\to \infty} \int_\Omega \theta^n \bu^n\cdot\be_d
= \int_\Omega \theta \bu\cdot \be_d
= \int_\Omega \BS\twodot \BD(\bu),
\end{equation}
and hence by Lemma \ref{lm:minty} we conclude that $\BG(\BS,\BD(\bu),\theta) = \bm{0}$. Finally, by taking traces on both sides of the constitutive relation we also obtain that $\tr\BS = 0$ and so $\BS\in \Lsymtr{2}$, which concludes the proof.
\end{proof}

In the proof of Theorem \ref{thm:convergence_temp_dependent} it becomes clear that the only bottleneck that prevents one from considering constitutive laws with more general $r$-coercivity (e.g.\ a power-law with temperature dependent consistency), is the fact that Lemma \ref{lm:minty} is tied to the particular function $\BG$ defined in \eqref{eq:implicit_CR}. Using Minty's trick it is possible to show that if an explicit constitutive relation is available, an analogous convergence result will hold.

\begin{assumption}\label{as:explicit_S}
	Let $\Scr:\Omega\times \Rds \times \R \to \Rds$ be a continuous function satisfying for some $r>\frac{2d}{d+2}$:
\begin{itemize}[leftmargin = 0.8cm]
	\item (Monotonicity). For every $\btau_1,\btau_2 \in \Rds$:
		\begin{equation*}
			(\Scr(\btau_1,s)-\Scr (\btau_2,s))\twodot (\btau_1 - \btau_2)\geq 0 \text{ for fixed }s\in\R;
		\end{equation*}
	\item (Coercivity). There is a non-negative function $m\in \Lp{1}$ and a constant $c>0$ such that
		\begin{equation*}
			\Scr(\btau,s)\twodot \btau \geq -m + c(|\Scr(\btau,s)|^{r'} + |\btau|^r)\quad \text{ for all }\btau\in \Rds,s\in \R;
		\end{equation*}
	\item (Growth). There is a function $n\in \Lp{r'}$ and a constant $c>0$ such that
		\begin{equation*}
			|\Scr(\btau,s)| \leq c(|\btau|^{r'-1} + n);
		\end{equation*}
	\item (Compatibility). For a fixed $s\in\R$ we have that $\tr(\Scr(\btau,s)) = 0$ if and only if $\tr(\btau) = 0$, for any $\btau\in \Rds$.
\end{itemize}
\end{assumption}

When $r < \frac{3d}{d+2}$ the velocity $\bu$ is not an admissible test function anymore and so obtaining an identity such as \eqref{eq:limsup_temp_dependent} is not straightforward. This difficulty can be overcome by testing instead with a discrete Lipschitz truncation of the error $e^n := \bu - \bu^n$. The discrete Lipschitz truncation was introduced in \cite{Diening:2013}, and the idea is that it turns $\be^n$ into a Lipschitz function belonging to $V^n$ in such a way that the size of the set where the truncation does not equal the original function can be controlled. We note that the construction of this discrete Lipschitz truncation requires a refined version of Assumption \ref{as:Projector_Velocity}.

\begin{assumption}[Fortin Projector $\Pi^n_V$]\label{as:Projector_Velocity_local}
For each $n\in\mathbb{N}$ there is a linear projector $\Pi^n_V :W^{1,1}_0(\Omega)^d\rightarrow V^n$ such that it preserves the divergence in the same sense as in Assumption \ref{as:Projector_Velocity}, but the stability condition is replaced by:
\begin{itemize}[leftmargin = 0.8cm]
\item (Local $W^{1,1}$-stability). For every $s\in (1,\infty)$ there is a constant $c>0$, independent of $n$, such that
\begin{equation*}
	\frac{1}{|K|}\int_K |\nabla \Pi^n_V\bv|
	\leq c \frac{1}{|\Omega_K^n|}\int_{\Omega_K^n} |\nabla\bv|
	\qquad \forall \,\bm{v}\in W^{1,s}_0(\Omega)^d, K\in \mathcal{T}_n,
\end{equation*}
where $\Omega_K^n$ denotes the patch of elements in $\mathcal{T}_n$ whose intersection with $K$ is nonempty.
\end{itemize}
\end{assumption}

It can be shown that the local $W^{1,1}$-stability from Assumption \ref{as:Projector_Velocity_local} implies the global $W^{1,s}$-stability of Assumption \ref{as:Projector_Velocity} \cite{Belenki:2012,Diening:2013}. Some examples of finite elements satisfying Assumption \ref{as:Projector_Velocity_local} include the conforming Crouzeix--Raviart element, the MINI element, the Bernardi--Raugel element, the $\mathbb{P}_2$--$\mathbb{P}_0$ and the Taylor--Hood pair $\mathbb{P}_k$--$\mathbb{P}_{k-1}$ for $k\geq d$ \cite{Belenki:2012}; the lowest order Taylor--Hood pair in 3D also satisfies the assumption if the mesh has a certain macroelement structure \cite{Girault2015}. As for exactly divergence-free elements, this assumption can also be verified for low order Guzm\'{a}n--Neilan elements and the Scott--Vogelius pair \cite{Diening:2013,Tscherpel2018}.

\begin{corollary}\label{cor:explicit_relation_convergence}
	Let $r>\frac{2d}{d+2}$ and let $\Scr:\Rds\times \R\to\Rds$ be a function satisfying Assumption \ref{as:explicit_S} and suppose that $\{U^n,V^n,M^n\}_{n\in\mathbb{N}}$ is a family of finite element subspaces satisfying Assumptions \ref{as:approximability}, \ref{as:Projectors_PressureTemp}, \ref{as:reconstruction_operator}, and \ref{as:Projector_Velocity_local}. Then, for any $n\in\mathbb{N}$, the finite element formulation obtained by replacing $\Scr^n$ by $\Scr$ in Formulation $\tilde{\mathrm{A}}^n_0$ admits a solution $(\theta^n,\bu^n,p^n)\in (\hat{\theta}^n_b + U^n)\times V^n\times M^n_0$ and we have, up to subsequences, that
	\begin{align*}
		\bu^n &\rightharpoonup \bu &\text{weakly in }\Wonep{r},\\
		p^n &\rightharpoonup p &\text{weakly in }\Lp{\tilde{r}},\\
		\theta^n &\rightharpoonup \theta &\text{weakly in }H^1(\Omega),\\
		\Scr(\BD(\bu^n),\theta^n) &\rightharpoonup \BS &\text{weakly in }\Lsym{r'},
	\end{align*}
	where $(\BS,\theta,\bu,p)\in \Lsymtr{r'}\times(\hat{\theta}_b +\SobsH{1}{\Gamma_D}{})\times W^{1,r}_0(\Omega)^d\times \Lmean{\tilde{r}}$ is a solution of Formulation $\mathrm{A}_0$.
\end{corollary}
\begin{proof}
	The proof is entirely analogous to the proofs of Lemma \ref{lm:existence_discrete} and Theorem \ref{thm:convergence_temp_dependent}, with a couple of small differences. Firstly, the \emph{a priori} estimate \eqref{eq:apriori} changes to
	\begin{equation}
		\|\bu^n\|_{\Wonep{r}} + \|\theta^n\|_{\Hone[]}
		+ \|p^n\|_{\Lp{\tilde{r}}}  + \|\BS^n\|_{\Lp{r'}} \leq c,
	\end{equation}
which implies the desired weak convergences. On the other hand, since $r>\frac{2d}{d+2}$, for a small enough $\varepsilon>0$ we have that $r>\frac{(2+\varepsilon)d}{d + (2+\varepsilon)}$, which implies that $\bu^n\to\bu$ strongly in ${\Lp{2+\varepsilon}}^d$ as $n\to \infty$. Furthermore, from the consistency condition in Assumption \ref{as:reconstruction_operator} we see that
\begin{equation*}
	\|\pi^n\bu^n - \bu\|_{L^{2+\varepsilon}(K)} \leq \|\bu^n - \bu\|_{L^{2+\varepsilon}(K)}
	+ c h_K^{1 + d(\frac{1}{2+\varepsilon} - \frac{1}{r})}\|\bu^n\|_{W^{1,r}(K)},
\end{equation*}
where we have used a standard local inverse inequality; the exponent of $h_K$ is positive by the choice of $\varepsilon$, which implies that $\pi^n\bu^n \to \bu$ strongly in ${\Lp{2+\varepsilon}}^d$ as $n\to\infty$. This is enough to pass to the limit in the convective term:
\begin{equation}
	\tilde{\mathcal{B}}_n(\bu^n,\bu^n,\Pi^n\bv) \xrightarrow[\: n\to\infty \:]{} -\int_\Omega \bu\otimes \bu \twodot \BD(\bv),
\end{equation}
for any $\bv\in W_0^{1,(\frac{2+\varepsilon}{2})'}(\Omega)^d$.
As for the identification of the constitutive relation, by testing the discrete momentum equation with the discrete Lipschitz truncation of the error $\be^n:=\bu-\bu^n$ it is possible to prove that (see \cite{Tscherpel2018} for a similar argument)
\begin{equation}\label{eq:limsup_explicit}
	\limsup_{n\to\infty} \int_\Omega \Scr(\BD(\bu^n),\theta^n)\twodot \BD(\bu^n) \leq
	\int_\Omega \BS\twodot \BD(\bu).
\end{equation}
Furthermore, from the growth condition of $\Scr$ and the dominated convergence theorem (note that, up to a subsequence, we have that $\theta^n\to\theta$ almost everywhere, c.f.\ \eqref{eq:strong_convergence}) we see that, for any $\btau\in \Lsym{r}$,
\begin{equation}\label{eq:strong_convergence_explicit}
\Scr(\btau,\theta^n) \to \Scr(\btau,\theta)\quad \text{strongly in }{\Lp{r'}}^{d\times d},
\end{equation}
as $n\to\infty$. Combining the monotonicity of $\Scr$ with \eqref{eq:limsup_explicit} and \eqref{eq:strong_convergence_explicit} yields for an arbitrary $\btau\in {\Lsym{r}}$:
\begin{align*}
	0 &\leq \limsup_{n\to\infty} \int_\Omega (\Scr(\BD(\bu^n),\theta^n) - \Scr(\btau,\theta^n))\twodot (\BD(\bu^n) - \btau)\\
	  &\leq \int_\Omega (\BS - \Scr(\btau,\theta))\twodot (\BD(\bu^n) - \btau).
\end{align*}
Choosing $\btau = \BD(\bu) \pm \varepsilon \bsigma$ with an arbitrary $\bsigma\in C_0^\infty(\Omega)^{d\times d}$ and letting $\varepsilon\to 0$ concludes the proof.
\end{proof}

\begin{remark}
	The use of the discrete Lipschitz truncation is only necessary when the velocity $\bu$ is not an admissible test function in the momentum equation, which occurs when $r<\frac{3d}{d+2}$. If $r\geq \frac{3d}{d+2}$ then one can substitute Assumption \ref{as:Projector_Velocity_local} with Assumption \ref{as:Projector_Velocity}. It is also important to note that if the trilinear form $\mathcal{B}$ is used instead, the stronger assumption $r>\frac{2d}{d+1}$ is required (see \eqref{eq:bound_convective_skew}).
\end{remark}

\begin{remark}
	If the constitutive relation can be written in the form $\BD(\bu) = \Dcr(\BS,\theta)$, where $\Dcr$ satisfies analogous conditions to the ones stated in Assumption \ref{as:explicit_S}, then the corresponding 4-field formulation will also satisfy an analogous convergence result. An example of a constitutive relation captured by these assumptions is the Ostwald--de Waele power-law model with $r>\frac{2d}{d+2}$:
	\begin{align*}
 		\Scr(\BD,\theta) &:= K(\theta)|\BD|^{r-2}\BD,\\
		\Dcr(\BS,\theta) &:= \frac{1}{K(\theta)}\left|\frac{\BS}{K(\theta)}\right|^{r'-2}\BS,
	\end{align*}
	where $K:\R\to\R$ is a continuous function satisfying $c_1\leq K(s) \leq c_2$ for any $s\in\R$, where $c_1,c_2$ are two positive constants.
\end{remark}

As mentioned in Section \ref{sec:implcr}, if the rheological parameters are not temperature-dependent, the convergence result can cover very general constitutive relations defined by maximal monotone $r$-graphs (which include, for instance, Herschel--Bulkley fluids). For this problem let us define Formulation $\textrm{C}_{0}$ in exactly the same way as Formulation A$_0$, but replacing \eqref{eq:weak_form_constrel} with
\begin{equation}
	\BH(\cdot,\BS,\BD(\bu)) = \bm{0} \qquad \text{ a.e. in }\Omega.
\end{equation}

In order to introduce the finite element formulation, the only necessary ingredient is an approximation to the graph $\mathcal{A}$, for which a result analogous to Lemma \ref{lm:minty} holds. This is the case e.g.\ for the generalised Yosida approximation described in \cite{Tscherpel2018}:
\begin{equation}\label{eq:approximation_Yosida}
	\Dcr^n(x,\BS):= \{\BD\in \Rds\, :\, (\BD,\BS)\in \mathcal{A}^n(x)\},
\end{equation}
where the approximate graph $\mathcal{A}^n$ is defined as follows
\begin{equation}
	\mathcal{A}^n(x):= \{\left(\BD ,\BS + \tfrac{1}{n}|\BD|^{r-2}\BD  \right)\in \Rds\times \Rds \, :\, (\BD,\BS)\in \mathcal{A}(x) \},
\end{equation}
where $x\in\Omega$. The relation \eqref{eq:approximation_Yosida} defines in fact a single-valued function that can be employed in the definition of the finite element formulation. Formulation $\tilde{\mathrm{C}}_0^n$ is then defined in the same way as Formulation $\tilde{\mathrm{B}}^n_0$, but with $\Dcr^n(\BS^n,\theta^n)$ replaced by \eqref{eq:approximation_Yosida}. However, it is worth pointing out that in the numerical computations one can simply work with the implicit function directly by writing
\begin{equation}
	\int_\Omega \BH(\cdot,\BS^n,\BD(\bu^n))\twodot \btau = 0 \quad \forall\,\btau\in\Sigma^n,
\end{equation}
instead of \eqref{eq:discrete_CR_B0}.

\begin{corollary}\label{cor:convergence_constant_parameters}
	Let $r>\frac{2d}{d+2}$ and let $\BH\colon \Omega\times \Rds\times \Rds\to \Rds$ be a function satisfying Assumption \ref{as:rgraph}. Suppose that $\{\Sigma^n,U^n,V^n,M^n\}_{n\in\mathbb{N}}$ is a family of finite element subspaces satisfying Assumptions \ref{as:approximability}, \ref{as:Projector_Stress}, \ref{as:Projectors_PressureTemp}, \ref{as:reconstruction_operator}, and \ref{as:Projector_Velocity_local}. Then, for any $n\in\mathbb{N}$, Formulation $\tilde{\mathrm{C}}^n_0$ admits a solution $(\BS^n,\theta^n,\bu^n,p^n)\in \Sigma^n\times (\hat{\theta}^n_b+ U^n)\times V^n \times M^n$, and we have, up to subsequences, that
	\begin{align*}
		\bu^n &\rightharpoonup \bu &\text{weakly in }\Wonep{r},\\
		p^n &\rightharpoonup p &\text{weakly in }\Lp{\tilde{r}},\\
		\theta^n &\rightharpoonup \theta &\text{weakly in }H^1(\Omega),\\
		\BS^n &\rightharpoonup \BS &\text{weakly in }\Lsym{r'},
	\end{align*}
	where $(\BS,\theta,\bu,p)\in \Lsymtr{r'}\times(\hat{\theta}_b +\SobsH{1}{\Gamma_D}{})\times W^{1,r}_0(\Omega)^d\times \Lmean{\tilde{r}}$ is a solution of Formulation $\mathrm{C}_0$.
\end{corollary}

\begin{remark}
When restricted to the isothermal case, the convergence result from Corollary \ref{cor:convergence_constant_parameters} improves the one presented in \cite{Diening:2013} in two respects: the graph is not required to be strictly monotone here, which allows models with a yield stress, for instance, and the result holds for the whole admissible range $r>\frac{2d}{d+2}$ even without the use of pointwise divergence-free elements, thanks to the modified convective term $\tilde{\mathcal{B}}_n$. In addition, the argument used here in the identification of the constitutive relation avoids the use of Young measures, simplifying the proof.
\end{remark}

\section{Augmented Lagrangian Preconditioner}\label{sec:AL_preconditioner}
In this section we employ the Scott--Vogelius pair for the velocity and pressure, and discontinuous and continuous elements for the stress and temperature, respectively, with $k\geq d$:
\begin{equation} \label{eq:sv}
	\begin{split}
	\Sigma^n &= \{\bsigma\in \Lsymtr{\infty} \, :\, \bsigma|_K\in\mathbb{P}_{k-1}(K)^{d\times d}\text{ for all }K\in \mathcal{T}_n\},\\
	U^n &= \{\eta\in W_{\Gamma_D}^{1,\infty}(\Omega)\, :\, \eta|_K\in\mathbb{P}_{k}(K)\text{ for all }K\in \mathcal{T}_n\},\\
	V^n &= \{\bw\in W^{1,\infty}_0(\Omega)^d\, :\, \bw|_K\in\mathbb{P}_{k}(K)^d\text{ for all }K\in \mathcal{T}_n\},\\
M^n &= \{q\in L_0^\infty(\Omega)\, :\, q|_K\in \mathbb{P}_{k-1}(K)\text{ for all }K\in \mathcal{T}_n\}.
\end{split}
\end{equation}
In order to ensure the inf-sup stability of the velocity-pressure pair, each level $\mathcal{T}_n$ of the mesh hierarchy is barycentrically refined, with the hierarchy itself constructed by uniform refinement, to prevent the appearance of degenerate elements (Figure \ref{fig:mesh_hierarchy}). A drawback of this approach is that the resulting mesh hierarchy is non-nested, which introduces some difficulties when dealing with the transfer operators in the multigrid algorithm.

\begin{figure}
\begin{center}
\begin{tikzpicture}[scale=1.8]
\usetikzlibrary{positioning}
	\node (A) [] at (-1,1){
		\begin{tikzpicture}[scale=2]
			\draw[line width=1.5pt] (-0.5,-0) -- (0.5,0) -- (0.,0.866) -- cycle;
		\end{tikzpicture}
		};
	\node (B) [] at (1,1) {
		\begin{tikzpicture}[scale=2]
			\draw[line width=1.5pt] (-0.5,-0) -- (0.5,0) -- (0.,0.866) -- cycle;
			\draw[line width=1.5pt] (-0,0.015) -- (0.235,0.425) -- (-0.235,0.425) -- cycle;
		\end{tikzpicture}
		};
	\draw [->] (A) -- (B);
	\node (A_bary) [] at (-1,-0.6){
		\begin{tikzpicture}[scale=2]
			\draw[line width=1.5pt] (-0.5,-0) -- (0.5,0) -- (0.,0.866) -- cycle;
			\begin{scope}[style=dashed]
				\coordinate (bar1) at (-0.5,0);
				\coordinate (bar2) at (0.5,0);
				\coordinate (bar3) at (0,0.866);
				\draw (barycentric cs:bar1=0.5,bar2=0.5,bar3=0.5) -- (barycentric cs:bar1=1,bar2=0,bar3=0);
				\draw (barycentric cs:bar1=0.5,bar2=0.5,bar3=0.5) -- (barycentric cs:bar1=0,bar2=1,bar3=0);
				\draw (barycentric cs:bar1=0.5,bar2=0.5,bar3=0.5) -- (barycentric cs:bar1=0,bar2=0,bar3=1);
			\end{scope}
		\end{tikzpicture}
		};
	\node (B_bary) [] at (1,-0.6) {
		\begin{tikzpicture}[scale=2]
			\draw[line width=1.5pt] (-0.5,-0) -- (0.5,0) -- (0.,0.866) -- cycle;
			\draw[line width=1.5pt] (-0,0.015) -- (0.235,0.425) -- (-0.235,0.425) -- cycle;
			\begin{scope}[style=dashed]
			\foreach \x in {-0.5,0.}{
				\coordinate (bar1) at (\x,0);
				\coordinate (bar2) at (\x+0.5,0);
				\coordinate (bar3) at (\x+0.25,0.433);
				\draw (barycentric cs:bar1=0.5,bar2=0.5,bar3=0.5) -- (barycentric cs:bar1=1,bar2=0,bar3=0);
				\draw (barycentric cs:bar1=0.5,bar2=0.5,bar3=0.5) -- (barycentric cs:bar1=0,bar2=1,bar3=0);
				\draw (barycentric cs:bar1=0.5,bar2=0.5,bar3=0.5) -- (barycentric cs:bar1=0,bar2=0,bar3=1);
			}
				\coordinate (bar1) at (0.,0);
				\coordinate (bar2) at (-0.25,0.433);
				\coordinate (bar3) at (0.25,0.433);
				\draw (barycentric cs:bar1=0.5,bar2=0.5,bar3=0.5) -- (barycentric cs:bar1=1,bar2=0,bar3=0);
				\draw (barycentric cs:bar1=0.5,bar2=0.5,bar3=0.5) -- (barycentric cs:bar1=0,bar2=1,bar3=0);
				\draw (barycentric cs:bar1=0.5,bar2=0.5,bar3=0.5) -- (barycentric cs:bar1=0,bar2=0,bar3=1);
				\coordinate (bar1) at (0,0.866);
				\coordinate (bar2) at (-0.25,0.433);
				\coordinate (bar3) at (0.25,0.433);
				\draw (barycentric cs:bar1=0.5,bar2=0.5,bar3=0.5) -- (barycentric cs:bar1=1,bar2=0,bar3=0);
				\draw (barycentric cs:bar1=0.5,bar2=0.5,bar3=0.5) -- (barycentric cs:bar1=0,bar2=1,bar3=0);
				\draw (barycentric cs:bar1=0.5,bar2=0.5,bar3=0.5) -- (barycentric cs:bar1=0,bar2=0,bar3=1);
			\end{scope}
		\end{tikzpicture}
		};
	\draw [->] (A_bary) -- (B_bary);
	\draw [->] (A) -- (A_bary);
	\draw [->] (B) -- (B_bary);
\end{tikzpicture}
\end{center}
\caption{\label{fig:mesh_hierarchy}
	Non-nested two-level barycentrically refined mesh hierarchy.}
\end{figure}
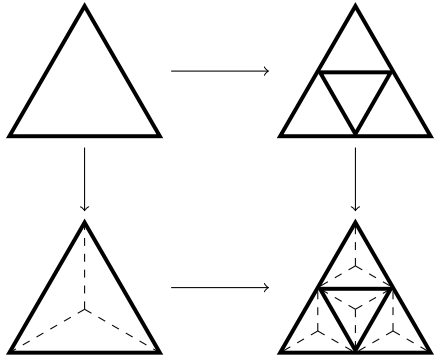
As mentioned in Section \ref{sc:FEM_spaces}, this choice of finite element space for the stress satisfies the inf-sup condition \eqref{eq:infsupstress}. In fact, since this discretisation has the property that discretely divergence-free velocities are exactly divergence-free, one can work with a traceless approximation for the \txtb{deviatoric} stress and hence fewer degrees of freedom will be required (c.f.\ \cite{FarrellGazca2019}). This exact enforcement of the divergence constraint was one of the motivations behind our choice of elements; it is known that a failure to enforce the divergence-free constraint appropriately can lead to unphysical behaviour in the solution of buoyancy-driven flow \cite{John2017}.

At this point the viscous dissipation and the adiabatic heating terms can be incorporated into the formulation. For instance, when working with the setting described by Corollary \ref{cor:explicit_relation_convergence}, in the finite element formulation we seek $(\theta^n,\bu^n,p^n)\in (\hat{\theta}_b + U^n)\times V^n\times M_0^n$ such that
\begin{subequations}\label{eq:FEM_viscous_dissipation}
\begin{alignat}{2}
	\int_\Omega \Scr(\BD(\bu^n),\theta^n)\twodot\BD(\bv)
	-\int_\Omega &(\bu^n\otimes \bu^n) \twodot \BD(\bv)
	-\int_\Omega p^n  \diver\bv = \int_\Omega \theta^n \bv\cdot\be_d \: & \forall\,\bv\in V^n, \notag \\
	&-\int_\Omega q  \diver\bu^n = 0 & \forall\, q\in M^n,\\
	\int_\Omega (\hat{\kappa}(\theta^n)\nabla\theta^n - \bu^n\theta^n)\cdot \nabla\eta
	&+ \int_\Omega \theta^n\bu^n\cdot\be_d\eta  = \int_\Omega \Scr(\BD(\bu^n),\theta^n)\twodot\BD(\bu^n)\eta & \forall\, \eta \in U^n,\notag
\end{alignat}
\end{subequations}
with analogous modifications for the other formulations. Note that the form of the convective term has been simplified since the elements are exactly divergence-free. The nonlinear finite element formulations are linearised using Newton's method; for instance, if the current guess for the solution of \eqref{eq:FEM_viscous_dissipation} is $(\tilde{\theta},\tilde{\bu},\tilde{p})$, then the method is defined by the correction step $(\tilde{\theta},\tilde{\bu},\tilde{p}) \mapsto (\tilde{\theta},\tilde{\bu},\tilde{p}) + (\theta,\bu,p)$ where $(\theta,\bu,p)$ is the solution of a linear system whose matrix has the block structure
\begin{equation}\label{eq:block_structure_3field}
	\begin{bmatrix}
		A_1&C&0\\
		E& A_2&\tilde{B}^\top\\
		0&\tilde{B}&0
 \end{bmatrix}\begin{bmatrix}
 \theta \\ \bu\\ p
 \end{bmatrix}.
\end{equation}
The blocks in \eqref{eq:block_structure_3field} are defined through the linear operators:
\begin{subequations}
	\begin{align*}
		\langle A_1 \theta,\eta \rangle := \int_\Omega \theta \hat{\kappa}'(\tilde{\theta}) \nabla \tilde{\theta}\cdot\nabla \eta
		&+ \int_\Omega \hat{\kappa}(\tilde{\theta})\nabla \theta \cdot\nabla \eta
		-\int_\Omega \tilde{\bu}\theta\cdot \nabla \eta &&  \\
								+\int_\Omega \tilde{\bu}\theta \cdot \be_d\eta
		-&\int_\Omega \Scr_\theta (\BD(\tilde{\bu}),\tilde{\theta})\twodot \BD(\tilde{\bu})\theta\eta
						&& \forall\, \theta,\eta\in U^n, \notag\\
		\langle C \bu, \eta\rangle :=
		\int_\Omega \tilde{\theta}\bu\cdot(\be_d\eta - \nabla\eta&)
		- \int_\Omega \Scr_{\BD}(\BD(\tilde{\bu}),\tilde{\theta})\BD(\bu)\twodot\BD(\bv)\eta  &&\\
	- \int_\Omega \Scr(\BD(\tilde{\bu}),&\tilde{\theta})\twodot\BD(\bu)\eta && \forall\, \bu\in V^n,\eta\in U^n,\\
   \langle E\theta, \bv\rangle :=  \int_\Omega \Scr_\theta(\BD(&\tilde{\bu}),\tilde{\theta})\theta\twodot \BD(\bv)
   - \int_\Omega \theta \bv\cdot\be_d
				  &&\forall\,\theta\in U^n,\bv\in V^n,\\
   \langle A_2\bu,\bv\rangle :=
   \int_\Omega \left( \Scr_{\BD}(\BD(\tilde{\bu}),\tilde{\theta})\right.&\left.\BD(\bu)- \tilde{\bu}\otimes\bu - \bu\otimes \tilde{\bu} \right){:}\:\BD(\bv) && \forall\, \bu,\bv\in V^n,\\
   \langle \tilde{B}\bv,q\rangle &:=  -\int_\Omega q\diver\bv
				    &&\forall\, \bv\in V^n,q\in M^n.
	\end{align*}
	We use the notation $\Scr_{\BD},\Scr_\theta$ to denote the partial derivatives of $\Scr$; for instance, for the Navier--Stokes model one would have $\Scr_{\BD}(\BD(\tilde{\bu}),\tilde{\theta})= 2\hat{\mu}(\tilde{\theta})\bm{I}$ and $\Scr_\theta(\BD(\tilde{\bu}),\tilde{\theta}) = 2 \hat{\mu}'(\tilde{\theta})\BD(\tilde{\bu})$, where $\bm{I}$ is the fourth-order identity tensor.
\end{subequations}

\subsection{Robust relaxation and prolongation}
Keeping \eqref{eq:block_structure_3field} as an illustrative example, we see that after augmentation the top block can be written in the form
\begin{equation}\label{eq:top_block}
	A + \gamma B^{\top}M_p^{-1}B =
	\begin{bmatrix}
		A_1 & C \\ E & A_2
	\end{bmatrix}
	+ \gamma \begin{bmatrix}
		0 \\ \tilde{B}^\top
	\end{bmatrix}
	M_p^{-1}\begin{bmatrix}
		0 & \tilde{B}
	\end{bmatrix},
\end{equation}
where $A$ is invertible and $\gamma B^\top M_p^{-1}B$ is symmetric and semi-definite. Let us define $Z^n := U^n\times V^n$ whenever the 3-field formulation is employed and $Z^n := \Sigma^n \times U^n \times V^n$ otherwise. Relaxation methods in multigrid algorithms are often studied as subspace correction methods \cite{Xu1992,Xu2001}. Consider the space decomposition
\begin{equation}\label{eq:subspace_decomposition}
Z^n = \sum_i Z_i^n,
\end{equation}
where the sum is not necessarily direct. The key insight from \cite{Schoeberl1999,LWXZ:2007} is that, assuming $A$ is symmetric and coercive, the subspace correction method induced by the decomposition \eqref{eq:subspace_decomposition} will be robust in $\gamma$ if the decomposition stably captures the kernel $\mathcal{N}^n$ of the semi-definite term:
\begin{equation}
	\mathcal{N}^n = \sum_i Z^n_i\cap \mathcal{N}^n.
\end{equation}
Here $\mathcal{N}^n$ consists of the elements of the form $(\theta,\bv)^\top$ and $(\bsigma,\theta,\bv)^\top$ for the 3-field and 4-field formulations, respectively, where $\bv\in \Vndiv$, and $\bsigma\in \Sigma^n$, $\theta\in U^n$ are arbitrary.  This means that the decomposition must allow for sufficiently rich subspaces such that divergence-free velocities can be written as combinations of divergence-free elements of the subspaces. A local characterisation of the kernel of the divergence for Scott--Vogelius elements on meshes with the macro element structure shown in Figure \ref{fig:mesh_hierarchy} was presented in \cite{Farrella} and used to construct a preconditioner for a system of nearly incompressible elasticity; this construction was then employed in \cite{Farrell2020} and \cite{FarrellGazca2019} to precondition the isothermal Navier--Stokes system and a 3-field non-Newtonian formulation, respectively. In \cite{Farrella} it was shown that the kernel is captured by using a decomposition based on the subspaces
\begin{equation}\label{eq:decomposition_SV}
	Z^n_i := \{\bz\in Z^n\, :\, \supp(\bz)\subset \macrostar(q_i)  \},
\end{equation}
where for a vertex $q_i$, the macrostar patch $\macrostar(q_i)$ is defined as the union of all macro cells touching the vertex (Figure \ref{fig:macrostar}). In the algorithm presented here the relaxation solves based on the decomposition \eqref{eq:decomposition_SV} are performed additively.

The work of Sch\"{o}berl \cite{Schoeberl1999} also revealed the necessity of controlling the continuity constant of the prolongation operator in order to obtain a robust solver. In our setting, this entails ensuring that the prolongation operator $P_N\colon V^N \to V^n$ mapping coarse grid functions in $V^N$ into fine grid functions in $V^n$ has the property that divergence-free velocities get mapped to (nearly) divergence-free velocities; note that when using a standard prolongation based on interpolation, the condition $\diver\bv^N = 0$ does not necessarily imply that $\diver (P_N\bv^n)=0$. For the setting described here a modified prolongation operator can be defined by computing a correction using local Stokes solves on the macro cells (see \cite{Farrella} for details).

For the formulations including the stress there is an additional difficulty: it is not obvious how to transfer piecewise discontinuous fields between non-nested meshes. Here we employ the supermesh projection described in \cite{FarrellGazca2019}. For the temperature we employ a standard interpolation-based prolongation operator.

While the macrostar iteration mentioned above results in a robust relaxation scheme for the linear elasticity problem considered in \cite{Farrella}, on its own it ceases to be effective when applied to the substantially more complex problem \eqref{eq:top_block}. However, we find that a handful of GMRES iterations preconditioned by the macrostar iteration are very effective for the problem under consideration.

\begin{figure}
\begin{center}
\begin{tikzpicture}[scale=1.5]
\usetikzlibrary{shapes.geometric}
%macro grid
\begin{scope}
\draw[step=1cm] (-2,-1) grid (2,2);
\draw (-2,2) -- (1,-1);
\draw (-1,2) -- (2,-1);
\draw (0,2) -- (2,0);
\draw (1,2) -- (2,1);
\draw (-2,0) -- (-1,-1);
\draw (-2,-1) -- (-1,-1);
\draw (-2,1) -- (0,-1);
\end{scope}

%lines to barycenters
\begin{scope}[style=help lines]
\foreach \x  in {0,1,-1,-2} {
\foreach \y  in {0,1,-1,-1} {
\coordinate (bar)   at (\x,\y);
\coordinate (barr) at (\x,\y + 1);
\coordinate (barrr)      at (\x + 1,\y);
\coordinate (cbar)   at (\x + 1,\y + 1);
\draw (barycentric cs:cbar=0.5,barr=0.5 ,barrr=0.5) -- (barycentric cs:cbar=1.,barr=0. ,barrr=0);
\draw (barycentric cs:cbar=0.5,barr=0.5 ,barrr=0.5) -- (barycentric cs:cbar=0.,barr=1 ,barrr=0);
\draw (barycentric cs:cbar=0.5,barr=0.5 ,barrr=0.5) -- (barycentric cs:cbar=0.,barr=0. ,barrr=1);
\draw (barycentric cs:bar=0.5,barr=0.5 ,barrr=0.5) -- (barycentric cs:bar=1.,barr=0. ,barrr=0);
\draw (barycentric cs:bar=0.5,barr=0.5 ,barrr=0.5) -- (barycentric cs:bar=0.,barr=1 ,barrr=0);
\draw (barycentric cs:bar=0.5,barr=0.5 ,barrr=0.5) -- (barycentric cs:bar=0.,barr=0. ,barrr=1);}
}
\end{scope}

%first hexagon
\node[regular polygon, regular polygon sides=6, draw] [dashed, ultra thick, rounded corners=8pt, rotate=-45,xscale=10,yscale=6,magenta] at (-1,0){};
\filldraw [magenta] (-1,0) circle (1.5pt);
%second hexagon
\node[regular polygon, regular polygon sides=6, draw] [dashed, ultra thick,rounded corners=8pt,rotate=-45,xscale=10,yscale=6.,magenta] at (1,1){};
\filldraw [magenta] (1,1) circle (1.5pt);
\end{tikzpicture}
\end{center}
\caption{\label{fig:macrostar}
Macrostar patches on a barycentrically refined mesh.}
\end{figure}
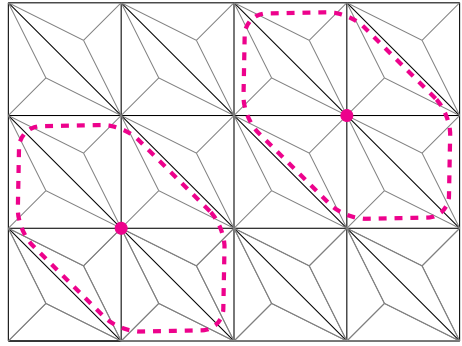

\section{Numerical experiments}\label{sec:experiments}
Let us suppose that the parameters in the constitutive relation \eqref{eq:implicit_CR} can be written as
\begin{equation}
	\frac{\hat{\mu}(\theta)}{\mu_0} = \mu(\theta),\quad
	\frac{\hat{\kappa}(\theta)}{\kappa_0} = \kappa(\theta),\quad
	\frac{\hat{\tau}(\theta)}{\tau_0} = \tau(\theta),\quad
	\frac{\hat{\sigma}(\theta)}{\sigma_0} = \sigma(\theta),\quad
\end{equation}
where $\mu_0,\kappa_0>0$ are reference values for the viscosity and heat conductivity, $\tau,\sigma\geq 0$ are reference values for the activation parameters, and $\mu,\kappa,\tau,\sigma$ are then non-dimensional functions.
In practice the Oberbeck-Boussinesq system \eqref{eq:PDE_OB} can be non-dimensionalised in distinct ways to give more importance to different physical regimes. For example, suppose that the time scale is chosen based on the diffusion of heat, and that the non-dimensional variables are introduced in the following way:
\begin{equation}
	\tilde{t}:= \frac{\alpha}{L^2}t,\quad
	\tilde{x}:= \frac{x}{L},\quad
	\tilde{\bu}:= \frac{L}{\alpha}\bu, \quad
	\tilde{p}:= \frac{L^2}{\rho_0\alpha^2}p,\quad
	\tilde{\theta} := \frac{\theta - \theta_C}{\theta_H - \theta_C},\quad
	\tilde{\BS}:= \frac{L^2}{\mu_0 \alpha}\BS,
\end{equation}
where $L$ is a characteristic length scale, $\theta_H$ is a reference temperature (e.g.~the temperature of the hot plate in a B\'{e}nard problem), and  $\alpha= \frac{\kappa_0}{\rho_0 c_p}$ is the thermal diffusion rate. The non-dimensional form of the system then reads (dropping the tildes):

\begin{subequations}\label{eq:PDE_OB_Rayleigh1}
	\begin{alignat}{2}
%		\begin{aligned}
			- \Pr \diver \BS + \diver (\bu&\otimes\bu)
			+ \nabla p = \Ra\,\Pr\,\theta\be_d
			\quad & &\text{ in }\Omega,\\
			&\diver\bu = 0\quad & &\text{ in }\Omega,\\
			-\diver (\kappa(\theta)\nabla\theta)
			+ \diver(\bu & \theta )
			+ \Di(\theta + \Theta)\bu\cdot \be_d
			= \frac{\Di}{\Ra}\BS\twodot \BD(\bu)\quad & &\text{ in }\Omega,
%		\end{aligned}
	\end{alignat}
\end{subequations}
where the Rayleigh, Prandtl, Dissipation and Theta numbers are defined respectively as
\begin{equation}
	\Ra = \frac{\beta g (\theta_H - \theta_C)L^3}{\nu \alpha},\quad
	\Pr = \frac{\nu_0}{\alpha},\quad
	\Di = \frac{\beta g L }{c_p},\quad
\Theta =\frac{\theta_C}{\theta_H - \theta_C},\quad
\end{equation}
where $\nu_0:= \frac{\mu_0}{\rho_0}$ is the reference kinematic viscosity (more non-dimensional numbers could arise with a non-Newtonian constitutive relation). %On the other hand, if one chooses to balance the pressure using the scaling $\tilde{p}:= \frac{L^2}{\rho_0 \nu_0 \alpha}p$, then the momentum equation takes instead the form \cite{Turcotte1974}
%\begin{equation}\label{eq:PDE_OB_Rayleigh2}
%	-\diver\BS + \frac{1}{\Pr}\diver(\bu\otimes\bu)
%	+\nabla p = \Ra \theta\be_d\qquad\text{ on }\Omega,
%\end{equation}
%with the other two equations unchanged; this form is used e.g.\ in the simulation of the earth's mantle convection, where sometimes the ``infinite Prandtl number'' approximation is used and the convective term is dropped altogether \cite{Schubert2001}.
Alternatively, if one assumes that the gravitational potential energy is completely transformed into kinetic energy \cite{Hewitt1975,Ostrach1958}, the characteristic velocity is chosen as $U = (gL\beta(\theta_H - \theta_C))^{1/2}$ and the resulting non-dimensional system becomes
\begin{subequations}\label{eq:PDE_OB_Grashof}
	\begin{alignat}{2}
%		\begin{aligned}
		- \frac{1}{\sqrt{\Gr}} \diver \BS + \diver &(\bu\otimes\bu)
			+ \nabla p = \theta\be_d
			\quad & &\text{ in }\Omega,\\
			&\diver\bu = 0\quad & &\text{ in }\Omega,\\
			-\frac{1}{\Pr\sqrt{\Gr}}\diver (\kappa(\theta)\nabla\theta)
			+ \diver(&\bu  \theta )
			+ \Di(\theta + \Theta)\bu\cdot \be_d
			= \frac{\Di}{\sqrt{\Gr}}\BS\twodot \BD(\bu)\quad & &\text{ in }\Omega,
%		\end{aligned}
	\end{alignat}
\end{subequations}
where the Grashof number is defined as
\begin{equation}
	\Gr = \frac{gL^3 \beta (\theta_H - \theta_C)}{\nu_0^2}.
\end{equation}

In the following section we will numerically investigate the convergence of Formulations $\textrm{A}_0^n$ and $\textrm{B}_0^n$, employing the non-dimensional form \eqref{eq:PDE_OB_Rayleigh1} (which includes the viscous dissipation terms) on a two-dimensional cavity problem by means of the method of manufactured solutions. In Section \ref{sc:heated_cavity} we will then examine the performance of the solver introduced in Section \ref{sec:AL_preconditioner} using the different forms \eqref{eq:PDE_OB_Rayleigh1} and \eqref{eq:PDE_OB_Grashof} with a heated cavity problem. The computational examples were implemented in Firedrake \cite{Rathgeber2016}, and PCPATCH \cite{Farrell} (a recently developed tool for subspace decomposition in multigrid in PETSc \cite{PETSc}) was employed for the macrostar patch solves in the multigrid algorithm. The augmented Lagrangian parameter was set to $\gamma=10^4$, and unless specified otherwise, the Newton solver was deemed to have converged when the Euclidean norm of the residual fell below $1\times 10^{-8}$ and the corresponding tolerance for the linear solver in 2D was set to $1\times 10^{-10}$ ($1\times 10^{-8}$ in 3D). In the implementation the uniqueness of the pressure was enforced by orthogonalizing against the nullspace of constants in the Krylov solver, instead of enforcing a zero mean condition.

\subsection{Convergence test}
We consider the exact solution
\begin{gather*}
	\bu_e(x,y) = \begin{pmatrix}
		2 y \sin(\pi x)\sin(\pi y)(x^2 -1) + \pi\sin(\pi x)\cos(\pi y)(x^2 -1)(y^2 -1)\\
		-2 x \sin(\pi x)\sin(\pi y)(y^2 -1) + \pi\cos(\pi x)\sin(\pi y)(x^2 -1)(y^2 -1)
	\end{pmatrix},\\
	p_e(x,y) = y^2 - x^2,\\
	\theta_e(x,y) = x^2 - y^4,
\end{gather*}
on $\Omega = (0,1)^2$.
The boundary data for the velocity and temperature are chosen so as to match the exact solution.
The constitutive relation is defined by a standard power-law
\begin{equation*}
	\BS = K(\theta)|\BD(\bu)|^{r-2}\BD(\bu),
\end{equation*}
where $r>1$, and the temperature dependence of the heat conductivity and consistency index are taken to be
\begin{align*}
	K(\theta) &:= \mathrm{e}^{-\frac{\theta}{4}},\\
	\kappa(\theta) &:= \mathrm{e}^{4\theta}.
\end{align*}

In this test we consider the system without augmentation (i.e.\ $\gamma = 0$) and solve the linear systems using a sparse direct solver from MUMPS \cite{MUMPS:1}. We employ the lowest order Taylor--Hood element for the velocity and pressure, and use continuous and discontinuous piecewise polynomials of degree 2 and 1 for the temperature and stress, respectively:

\begin{equation}
	\begin{split}
	\Sigma^n &= \{\bsigma\in \Lsym{\infty} \, :\, \bsigma|_K\in\mathbb{P}_{1}(K)^{2\times 2}\text{ for all }K\in \mathcal{T}_n\},\\
	U^n &= \{\eta\in  \theta_e + W_{0}^{1,\infty}(\Omega)\, :\, \eta|_K\in\mathbb{P}_{2}(K)\text{ for all }K\in \mathcal{T}_n\},\\
	V^n &= \{\bw\in \bu_e + W^{1,\infty}_0(\Omega)^2\, :\, \bw|_K\in\mathbb{P}_{2}(K)^2\text{ for all }K\in \mathcal{T}_n\},\\
	M^n &= \{q\in C(\overline{\Omega})\, :\, q|_K\in \mathbb{P}_{1}(K)\text{ for all }K\in \mathcal{T}_n\}.
\end{split}
\end{equation}

Tables \ref{tb:errors_Tup}--\ref{tb:errors_TSup2} show the computational errors and the experimental orders of convergence, which are defined as:
\begin{gather*}
	E^n_{\bu} := \|\bu^n - \bu_e\|_{L^r(\Omega)},\qquad
	EOC^n_{\bu} := \log_2(E^n_{\bu}/E^{n-1}_{\bu}), \\
	E^n_{p} := \|p^n - p_e\|_{L^{r'}(\Omega)},\qquad
	EOC^n_{p} := \log_2(E^n_{p}/E^{n-1}_{p}), \\
	E^n_{\theta} := \|\theta^n - \theta_e\|_{L^2(\Omega)},\qquad
	EOC^n_{\theta} := \log_2(E^n_{\theta}/E^{n-1}_{\theta}), \\
	E^n_{\BS} := \|\BS^n - \BS_e\|_{L^{r'}(\Omega)},\qquad
	EOC^n_{\BS} := \log_2(E^n_{\BS}/E^{n-1}_{\BS}). \\
\end{gather*}

\begin{table}
\centering
\captionsetup{justification=centering}
\begin{tabular}{c c c c c c}
\toprule
$h_n$ & {\# dofs} & $E^n_{\bu}$ & $EOC^n_{\bu}$ & $E^n_p$ & $EOC^n_p$ \\
\midrule
0.104 & $5.4\times 10^3$ & $4.014\times 10^{-4}$  & - & $1.124\times 10^{-3}$ &  - \\
0.052 & $2.1\times 10^4$ & $5.126\times 10^{-5}$  & 2.97 & $1.850\times 10^{-3}$ & -0.71 \\
0.026 & $8.3\times 10^4$ & $6.360\times 10^{-6}$  & 3.01 & $1.879\times 10^{-4}$ & 3.29 \\
0.013 & $3.3\times 10^5$ & $7.944\times 10^{-7}$  & 3.00 & $1.902 \times 10^{-5}$ & 3.30\\
0.006 & $1.3\times 10^6$ & $1.002\times 10^{-7}$  & 2.98 & $2.231\times 10^{-6}$ & 3.09 \\
  \bottomrule
\end{tabular}\\
\caption{\label{tb:errors_Tup}
Errors and experimental orders of convergence for the velocity and pressure obtained using Formulation $\textrm{A}^n_0$ with $r=3.5$, $\Ra = 10^4$, $\Di = 0.3$, $\Pr = 1$, and $\Theta = 0$.}
\end{table}

\begin{table}
\centering
\captionsetup{justification=centering}
\begin{tabular}{c c c c }
\toprule
$h_n$ & {\# dofs}  & $E^n_{\theta}$ & $EOC^n_{\theta}$\\
\midrule
0.104 & $5.4\times 10^3$  & $9.307\times 10^{-6}$ & - \\
0.052 & $2.1\times 10^4$  & $1.159\times 10^{-6}$ & 3.00 \\
0.026 & $8.3\times 10^4$  & $1.451\times 10^{-7}$ & 2.99 \\
0.013 & $3.3\times 10^5$  & $1.816\times 10^{-8}$ & 2.99 \\
0.006 & $1.3\times 10^6$  & $2.272\times 10^{-9}$ & 2.99 \\
  \bottomrule
\end{tabular}\\
\caption{\label{tb:errors_Tup2}
Errors and experimental orders of convergence for the temperature obtained using Formulation $\textrm{A}^n_0$ with $r=3.5$, $\Ra = 10^4$, $\Di = 0.3$, $\Pr = 1$, and $\Theta = 0$.}
\end{table}

\begin{table}
\centering
\captionsetup{justification=centering}
\begin{tabular}{c c c c c c}
\toprule
$h_n$ & {\# dofs} & $E^n_{\bu}$ & $EOC^n_{\bu}$ & $E^n_p$ & $EOC^n_p$ \\
\midrule
0.104 & $1.2\times 10^4$ & $2.405\times 10^{-4}$  & - & $1.632\times 10^{-3}$ &  - \\
0.052 & $4.9\times 10^4$ & $3.103\times 10^{-5}$  & 2.95 & $2.229\times 10^{-4}$ & 2.87 \\
0.026 & $1.9\times 10^4$ & $8.599\times 10^{-6}$  & 1.85 & $1.238\times 10^{-4}$ & 0.85 \\
0.013 & $7.8\times 10^5$ & $9.828\times 10^{-7}$  & 3.13 & $1.877 \times 10^{-5}$ & 2.72\\
0.006 & $3.1\times 10^6$ & $1.408\times 10^{-7}$  & 2.80 & $3.840\times 10^{-6}$ & 2.29 \\
  \bottomrule
\end{tabular}\\
\caption{\label{tb:errors_TSup}
Errors and experimental orders of convergence for the velocity and pressure obtained using Formulation $\textrm{B}^n_0$ with $r=1.6$, $\Ra = 10^4$, $\Di = 0.3$, $\Pr = 1$, and $\Theta = 0$.}
\end{table}

\begin{table}
\centering
\captionsetup{justification=centering}
\begin{tabular}{c c c c c c}
\toprule
$h_n$ & {\# dofs} & $E^n_{\theta}$ & $EOC^n_{\theta}$ & $E^n_{\BS}$ & $EOC^n_{\BS}$ \\
\midrule
0.104 & $1.2\times 10^4$ & $9.307\times 10^{-6}$  & - & $3.128$ &  - \\
0.052 & $4.9\times 10^4$ & $1.167\times 10^{-6}$  & 2.99 & $1.649\times 10^{-2}$ & 7.56 \\
0.026 & $1.9\times 10^4$ & $2.256\times 10^{-7}$  & 2.37 & $7.053\times 10^{-3}$ & 1.22 \\
0.013 & $7.8\times 10^5$ & $2.651\times 10^{-8}$  & 3.09 & $2.760 \times 10^{-3}$ & 1.35\\
0.006 & $3.1\times 10^6$ & $3.713\times 10^{-9}$  & 2.86 & $1.121\times 10^{-3}$ & 1.30 \\
  \bottomrule
\end{tabular}\\
\caption{\label{tb:errors_TSup2}
Errors and experimental orders of convergence for the temperature and stress obtained using Formulation $\textrm{B}^n_0$ with $r=1.6$, $\Ra = 10^4$, $\Di = 0.3$, $\Pr = 1$, and $\Theta = 0$.}
\end{table}

Given the generality of the implicitly constituted framework, one cannot expect higher regularity or uniqueness results in general (especially if the formulation includes viscous dissipation), and so there is little hope of guaranteeing convergence using error estimates. However, Tables \ref{tb:errors_Tup}--\ref{tb:errors_TSup2} show that the formulations introduced in this work can exhibit orders of convergence that are close to the orders one would anticipate, given the choice of elements.

\subsection{Heated cavity}\label{sc:heated_cavity}
The problem \eqref{eq:PDE_OB_Grashof} is solved on the unit square/cube $\Omega = (0,1)^d$ with boundary data
\begin{equation*}
	\bu = \bm{0} \quad \text{on }\partial\Omega,\quad
	\nabla\theta \cdot\bm{n} = 0 \quad\text{on }\partial\Omega\setminus(\Gamma_H\cup\Gamma_C),\quad
	\theta = \left\{\begin{array}{cc}
			1, & \text{on }\Gamma_H,\\
			0, & \text{on }\Gamma_C,
	\end{array}\right.
\end{equation*}
where $\Gamma_H := \{x_1 = 0\}$ and $\Gamma_C :=\{x_1 = 1\}$. For the problems with temperature-dependent viscosity and conductivity we choose the following functional dependences:
\begin{subequations}\label{eq:temp_dependence}
	\begin{align}
		\mu(\theta) &:= \mathrm{e}^{-\frac{\theta}{10}}, \label{eq:temp_dependent_viscosity}\\
		\kappa(\theta) &:= \frac{1}{2} + \frac{\theta}{2} + \theta^2.\label{eq:temp_dependent_conductivity}
	\end{align}
\end{subequations}
The viscosity defined by \eqref{eq:temp_dependent_viscosity} decreases with temperature, as is the case with most liquids \cite{Ferro2002}; heat conductivities of the form \eqref{eq:temp_dependent_conductivity} are a good fit for most liquid metals and gases \cite{Emery1999}. Let us denote the problem solved with $\mu(\theta)\equiv 1 \equiv \kappa(\theta)$ by (P1), the one using \eqref{eq:temp_dependent_viscosity} and $\kappa(\theta)\equiv 1$ by (P2), and by (P3) the one using both forms in \eqref{eq:temp_dependence}. As we wish to investigate the performance of the preconditioner proposed in Section \ref{sec:AL_preconditioner}, we employ the Scott--Vogelius element pair \eqref{eq:sv} with velocity degree $k \ge d$.

A simple continuation algorithm was used to reach the different values of the parameters; for instance, the solution corresponding to a Rayleigh number $\Ra$ was used as an initial guess in Newton's method for the problem with $\Ra + \Ra_{\mathrm{step}}$, where $\Ra_{\mathrm{step}}$ is some predetermined step. In some cases (most notably shear-thinning fluids) the use of advective stabilisation was essential; here we have added to the formulation an advective stabilisation term based on penalising the jumps between facets \cite{Burman2008,Douglas1976}:
\begin{equation}\label{eq:burman_stabilisation}
	S_h(\bv,\bw) := \sum_{K\in\mathcal{M}_h} \frac{1}{2}\int_{\partial K}\delta \, h^2_{\partial K}\,\jump{\nabla\bv} : \jump{\nabla\bw},
\end{equation}
\noindent
where $\jump{\bz}$ denotes the jump of $\bz$ across $\partial K$,  $h_{\partial K}$ is the diameter of each face in $\partial K$, and $\delta$ is an arbitrary stabilization parameter. In the numerical experiments the stabilization parameter was chosen to be cell-dependent and set to $5\times10^{-3}\|\tilde{u}\|_{L^\infty(K)}$; an analogous term was added to the temperature equation. The choice of stabilisation \eqref{eq:burman_stabilisation} was preferred over the more common SUPG stabilisation because the latter introduces additional couplings between the velocity and the pressure in the momentum equation, and between the velocity, stress and temperature in the energy equation, which can spoil the convergence of the nonlinear solver (this was already observed in the isothermal case in \cite{Farrell2020}). \txtb{One} disadvantage is that \eqref{eq:burman_stabilisation} introduces an additional kernel consisting of $C^1$ functions, that might not be captured by the relaxation; this means that unless $k\geq 3$ in 2D or $k\geq 5$ in 3D, a slight loss of robustness might be expected \cite{Farrell2020}. \txtb{Another disadvantage is that the stabilisation term \eqref{eq:burman_stabilisation} reduces the sparsity of the system, hence increasing the computational cost. In the temperature equation it is possible to employ SUPG stabilisation, but in the examples an analogous stabilisation term to \eqref{eq:burman_stabilisation} was used; we found, in fact, that interior penalty and SUPG stabilisation yielded almost identical results.}

Tables \ref{tb:Grashof2D_P1}--\ref{tb:Grashof2D_P3} show the average number of Krylov iterations for the problem with non-dimensional form \eqref{eq:PDE_OB_Grashof} and increasingly large Grashof number, comparing with different values of the Dissipation number; Tables \ref{tb:Grashof3D_P1}--\ref{tb:Grashof3D_P3} show the same for the three-dimensional problem. \txtb{In the two-dimensional problem the mesh was constructed to be finer towards the vertical walls, to allow for a better capturing of the thermal layer that arises (see Figure \ref{fig:meshes}). For Grashof numbers around $10^8$ the solver in 2D starts to break down; this can be seen e.g.\ in Tables \ref{tb:Grashof2D_P2} and \ref{tb:Grashof2D_P3}. In this example, the cause of this seems to be the failure of the coarse grid correction to resolve finer features of the flow; improving the resolution of the coarse grid employed delays the failure of the solver. An analogous statement holds for the three-dimensional problem.}

\begin{figure}
\centering
\subfloat[C][{\centering After one refinement.}]{{%
	\includegraphics[width=0.5\textwidth]{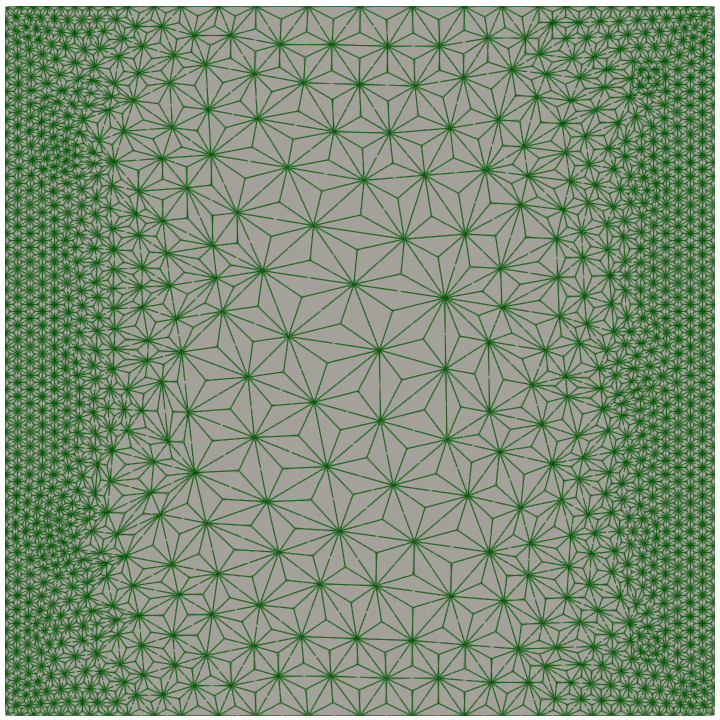}%
	}}%
	\subfloat[D][{\centering After two refinements.}]{{%
	\includegraphics[width=0.5\textwidth]{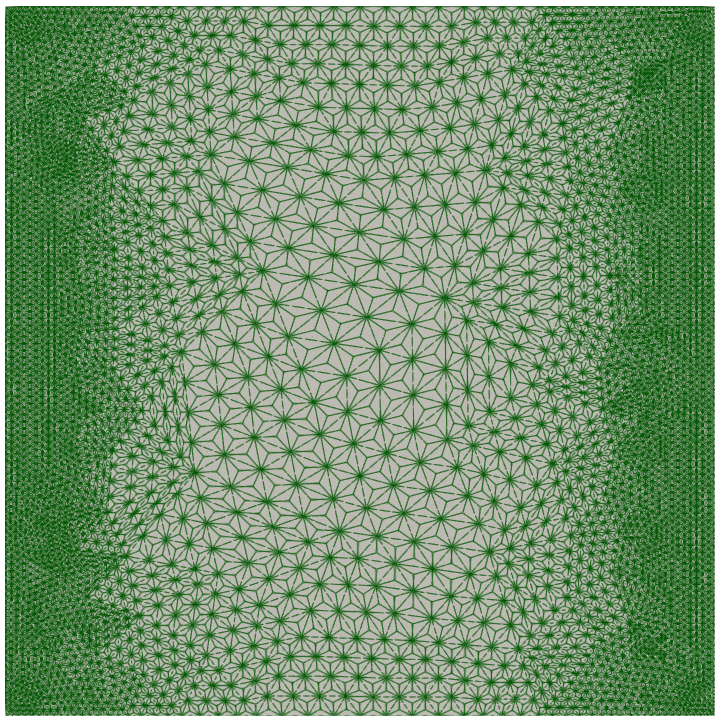}%
	}}%
	\caption{Computational meshes for the 2D heated cavity problem.}%
	\label{fig:meshes}
\end{figure}

It can be observed that the iteration count remains under control, and the previously mentioned loss of robustness occurs when $k=2$. Figure \ref{fig:Grashof2d} shows the streamlines and temperature contours for the problem (P3); it can be observed that the presence of the viscous dissipation term has a stabilising effect on the flow. Table \ref{tb:power_law} shows the number of iterations for the problem \eqref{eq:PDE_OB_Rayleigh1} using the temperature-dependent power-law relation
\begin{equation}\label{eq:example_power_law}
	\BS = \Scr(\BD(\bu),\theta) := \mathrm{e}^{-\frac{\theta}{10}} |\BD(\bu)|^{r-2}\BD(\bu),
\end{equation}
using $r=1.6$ and $\Di =2.0$, and the streamlines are shown in Figure \ref{fig:Rayleigh2d} alongside the ones of the Newtonian problem ($r=2$). In this case the tolerances for the linear and nonlinear iterations were set to $1\times 10^{-10}$. \txtb{In Table \ref{tb:power_law} it is apparent that the outer iteration counts are not as robust for this very challenging problem, although the performance is arguably still adequate. We note that the robustness in the iteration count can be improved by increasing the number of relaxation sweeps and multigrid cycles, but at this parameter range the less robust choice still results in an overall smaller computational time.}

\begin{table}
\centering
\captionsetup{justification=centering}
\begin{tabular}{c c c c  c c c c}
\toprule
 \multirow{2}{*}{$\Di$} & \multirow{2}{*}{$k$} &  \multirow{2}{*}{\# refs}& \multirow{2}{*}{\# dofs} & \multicolumn{4}{c}{$\Gr$ }\\ [0.1ex]
			& &  & & $5\times 10^4$ & $10^6$ & $10^7$ & ${10}^8$  \\
\midrule
 \multirow{4}{*}{$0$} & \multirow{2}{*}{2} & 1 & $5.5\times 10^4$  & 4.6 (5) & 4.5 (2) & 6 (2) & 20.3 (3) \\
											&  & 2 & $2.2\times 10^5$  & 4.4 (5) & 4 (2) & 6 (2) & 13.3 (3) \\
% \midrule
											& \multirow{2}{*}{3} & 1 & $1.2\times 10^5$ & 2.4 (5) & 2.5 (2) & 3 (2) & 6.33 (3) \\
											&  & 2 & $4.7\times 10^5$ & 2.4 (5) & 2 (2) & 2 (2) & 5 (3) \\
\midrule
 \multirow{4}{*}{$0.6$}& \multirow{2}{*}{2} & 1 & $5.5\times 10^4$ & 4.75 (4) & 4 (2) & 5.5 (2) & 11.5 (2) \\
											 &  & 2 & $2.2\times 10^5$ & 4.5 (4) & 4 (2)  & 4.5 (2) & 9 (2) \\
											 & \multirow{2}{*}{3} & 1 & $1.2\times 10^5$ & 2.75 (4) & 2 (2) & 2.5 (2) & 5 (2) \\
											 &  & 2 & $4.7\times 10^5$ & 2.5 (4) & 1.5 (2) & 2 (2) & 3.5 (2) \\
\midrule
 \multirow{4}{*}{$1.3$}& \multirow{2}{*}{2} & 1 & $5.5\times 10^4$ & 4.5 (4) & 3.5 (2) & 5 (2) & 7 (2) \\
											 &  & 2 & $2.2\times 10^5$ & 4.5 (4) & 3.5 (2) & 4 (2) & 5 (2) \\
											 & \multirow{2}{*}{3} & 1 & $1.2\times 10^5$ & 2.75 (4) & 2 (2) & 2 (2) & 3.5 (2) \\
											 &  & 2 & $4.7\times 10^5$ & 2.5 (4) & 1.5 (2) & 1.5 (2) & 2.5 (2) \\
\midrule
 \multirow{4}{*}{$2.0$}& \multirow{2}{*}{2} & 1 & $5.5\times 10^4$& 4.25 (4) & 3.5 (2) & 4.5 (2) & 5.5 (2) \\
											 &  & 2 & $2.2\times 10^5$ & 4.5 (4) & 3.5 (2) & 4 (2) & 4 (2) \\
											 & \multirow{2}{*}{3} & 1 & $1.2\times 10^5$ & 2.75 (4) & 2 (2) & 2 (2) & 3 (2) \\
											 &  & 2 & $4.7\times 10^5$ & 2.5 (4) & 1.5 (2) & 1.5 (2) & 2.5 (2) \\
  \bottomrule
\end{tabular}\\
\caption{\label{tb:Grashof2D_P1}
Average number of Krylov iterations per Newton step \txtb{and number of Newton iterations at the specified continuation step (in parentheses)}, as $\Gr$ increases for the 2D problem (P1) with $\Pr = 1$, obtained using \txtb{one} multigrid cycle with \txtb{6} relaxation sweeps.}
\end{table}

\begin{table}
\centering
\captionsetup{justification=centering}
\begin{tabular}{c c c c  c c c c}
\toprule
 \multirow{2}{*}{$\Di$} & \multirow{2}{*}{$k$} &  \multirow{2}{*}{\# refs}& \multirow{2}{*}{\# dofs} & \multicolumn{4}{c}{$\Gr$ }\\ [0.1ex]
			& &  & & $5\times 10^4$ & $10^6$ & $10^7$ & ${10}^8$  \\
\midrule
 \multirow{4}{*}{$0$} & \multirow{2}{*}{2} & 1 & $5.5\times 10^4$  & 4.5 (5) & 4.5 (2) & 6 (3) & 17.3 (3) \\
											&  & 2 & $2.2\times 10^5$  & 4.6 (5) & 4 (2) & 6 (3) &  - \\
% \midrule
											& \multirow{2}{*}{3} & 1 & $1.2\times 10^5$ & 2.4 (5) & 2.5 (2) & 2.67 (3) & 6.67 (3) \\
											&  & 2 & $4.7\times 10^5$ & 2.4 (5) & 2 (2) & 2 (3) & 4.67 (3) \\
\midrule
 \multirow{4}{*}{$0.6$}& \multirow{2}{*}{2} & 1 & $5.5\times 10^4$ & 5 (4) & 4 (2) & 5.5 (2) & 11.5 (2) \\
											 &  & 2 & $2.2\times 10^5$ & 4.5 (4) & 4 (2)  & 4.5 (2) & 13.5 (2) \\
											 & \multirow{2}{*}{3} & 1 & $1.2\times 10^5$ & 2.75 (4) & 2 (2) & 3 (2) & 5.5 (2) \\
											 &  & 2 & $4.7\times 10^5$ & 2.5 (4) & 2 (2) & 2 (2) & 4 (2) \\
\midrule
 \multirow{4}{*}{$1.3$}& \multirow{2}{*}{2} & 1 & $5.5\times 10^4$ & 4.5 (4) & 4 (2) & 5 (2) & 7 (2) \\
											 &  & 2 & $2.2\times 10^5$ & 4.5 (4) & 3.5 (2) & 4 (2) & 5.5 (2) \\
											 & \multirow{2}{*}{3} & 1 & $1.2\times 10^5$ & 2.75 (4) & 2 (2) & 2.5 (2) & 3.5 (2) \\
											 &  & 2 & $4.7\times 10^5$ & 2.5 (4) & 1.5 (2) & 2 (2) & 2.5 (2) \\
\midrule
 \multirow{4}{*}{$2.0$}& \multirow{2}{*}{2} & 1 & $5.5\times 10^4$& 4.5 (4) & 3.5 (2) & 4.5 (2) & 5 (2) \\
											 &  & 2 & $2.2\times 10^5$ & 4.5 (4) & 3.5 (2) & 4 (2) & 4 (2) \\
											 & \multirow{2}{*}{3} & 1 & $1.2\times 10^5$ & 2.75 (4) & 2 (2) & 2 (2) & 3 (2) \\
											 &  & 2 & $4.7\times 10^5$ & 2.5 (4) & 1.5 (2) & 1.5 (2) & 2.5 (2) \\
  \bottomrule
\end{tabular}\\
\caption{\label{tb:Grashof2D_P2}
Average number of Krylov iterations per Newton step \txtb{and number of Newton iterations at the specified continuation step (in parentheses)}, as $\Gr$ increases for the 2D problem (P2) with $\Pr = 1$, obtained using \txtb{one} multigrid cycle with \txtb{6} relaxation sweeps \txtb{(a dash indicates failure to converge in less than 200 linear iterations)}.}
\end{table}

\begin{table}
\centering
\captionsetup{justification=centering}
\begin{tabular}{c c c c  c c c c}
\toprule
 \multirow{2}{*}{$\Di$} & \multirow{2}{*}{$k$} &  \multirow{2}{*}{\# refs}& \multirow{2}{*}{\# dofs} & \multicolumn{4}{c}{$\Gr$ }\\ [0.1ex]
			& &  & & $5\times 10^4$ & $10^6$ & $10^7$ & ${10}^8$  \\
\midrule
 \multirow{4}{*}{$0$} & \multirow{2}{*}{2} & 1 & $5.5\times 10^4$  & 4.8 (5) & 4.5 (2) & 6 (3) & 14 (3) \\
											&  & 2 & $2.2\times 10^5$  & 5 (5) & 4.5 (2) & 5.33 (3) &  - \\
% \midrule
											& \multirow{2}{*}{3} & 1 & $1.2\times 10^5$ & 2.6 (5) & 2.5 (2) & 5 (3) & 8.33 (3) \\
											&  & 2 & $4.7\times 10^5$ & 2.4 (5) & 2 (2) & 3.67 (3) & 10 (2) \\
\midrule
 \multirow{4}{*}{$0.6$}& \multirow{2}{*}{2} & 1 & $5.5\times 10^4$ & 4.5 (4) & 4 (2) & 5.5 (2) & 11 (2) \\
											 &  & 2 & $2.2\times 10^5$ & 4.5 (4) & 4 (2)  & 4.5 (2) & 9 (2) \\
											 & \multirow{2}{*}{3} & 1 & $1.2\times 10^5$ & 2.75 (4) & 2 (2) & 2.5 (2) & 4.5 (2) \\
											 &  & 2 & $4.7\times 10^5$ & 2.5 (4) & 1.5 (2) & 2 (2) & 3.5 (2) \\
\midrule
 \multirow{4}{*}{$1.3$}& \multirow{2}{*}{2} & 1 & $5.5\times 10^4$ & 4.5 (4) & 4 (2) & 5 (2) & 6.5 (2) \\
											 &  & 2 & $2.2\times 10^5$ & 4.5 (4) & 3.5 (2) & 4 (2) & 5 (2) \\
											 & \multirow{2}{*}{3} & 1 & $1.2\times 10^5$ & 2.5 (4) & 2 (2) & 2 (2) & 3.5 (2) \\
											 &  & 2 & $4.7\times 10^5$ & 2.25 (4) & 1.5 (2) & 1.5 (2) & 2.5 (2) \\
\midrule
 \multirow{4}{*}{$2.0$}& \multirow{2}{*}{2} & 1 & $5.5\times 10^4$& 4.25 (4) & 3.5 (2) & 4 (2) & 5 (2) \\
											 &  & 2 & $2.2\times 10^5$ & 4.5 (4) & 3.5 (2) & 4 (2) & 4 (2) \\
											 & \multirow{2}{*}{3} & 1 & $1.2\times 10^5$ & 2.5 (4) & 2 (2) & 2 (2) & 3 (2) \\
											 &  & 2 & $4.7\times 10^5$ & 2 (4) & 1.5 (2) & 1.5 (2) & 2 (2) \\
  \bottomrule
\end{tabular}\\
\caption{\label{tb:Grashof2D_P3}
Average number of Krylov iterations per Newton step \txtb{and number of Newton iterations at the specified continuation step (in parentheses)}, as $\Gr$ increases for the 2D problem (P3) with $\Pr = 1$, obtained using \txtb{one} multigrid cycle with \txtb{6} relaxation sweeps \txtb{(a dash indicates failure to converge in less than 200 linear iterations)}.}
\end{table}

\begin{figure}
\centering
\subfloat[C][{\centering Streamlines for $\Di = 2$.}]{{%
	\includegraphics[width=0.5\textwidth]{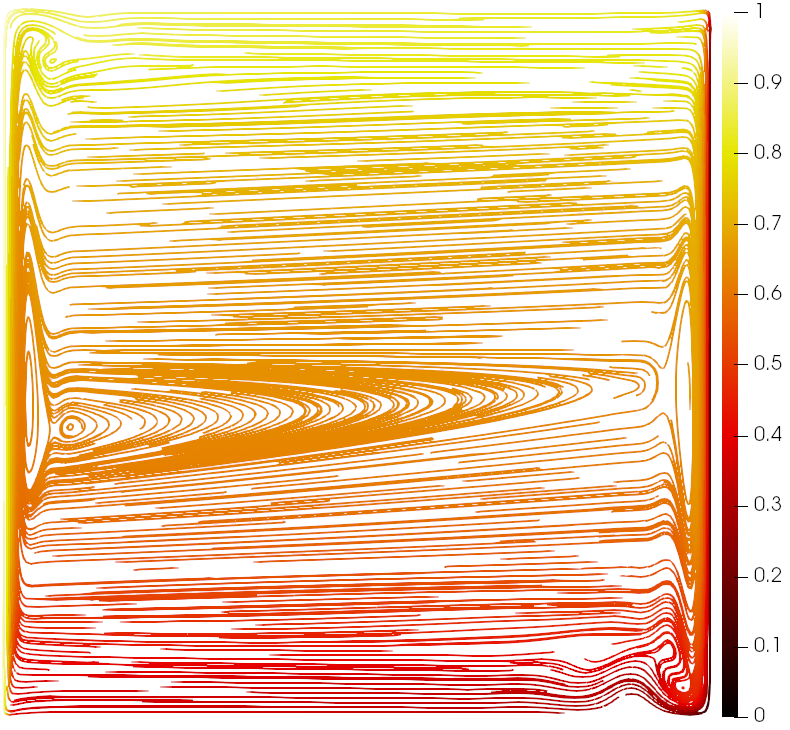}%
	}}%
	\subfloat[D][{\centering Streamlines for $\Di = 0$.}]{{%
	\includegraphics[width=0.5\textwidth]{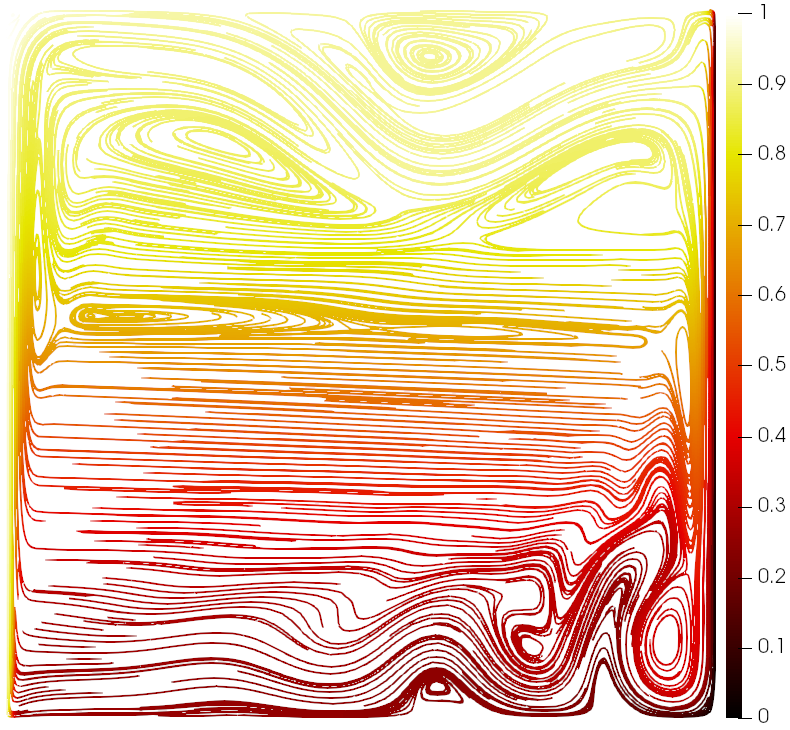}%
	}}\\
\subfloat[A][{\centering Temperature contours for $\Di = 2$.}]{{%
	\includegraphics[width=0.5\textwidth]{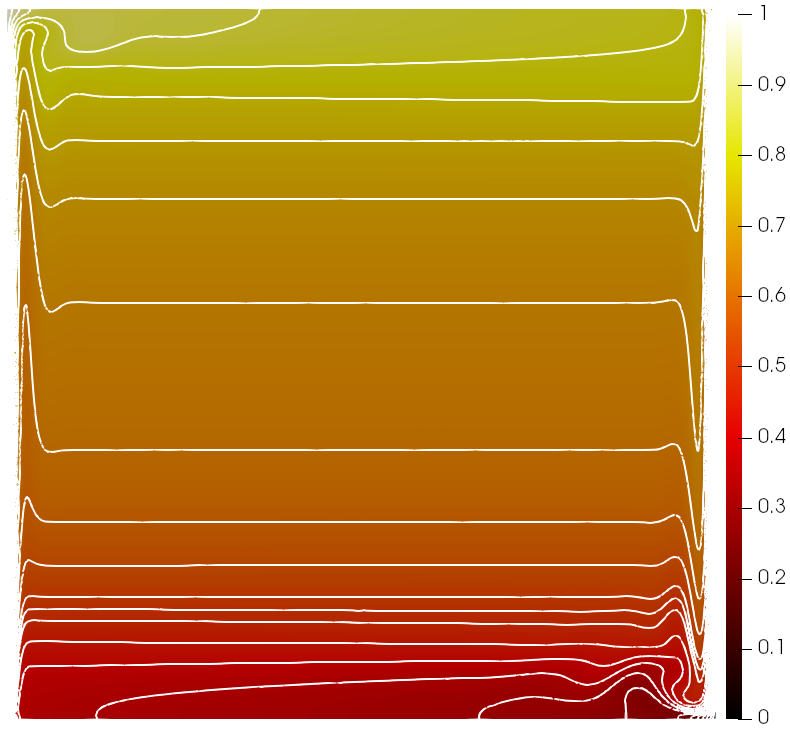}%
	}}%
	\subfloat[B][{\centering Temperature contours for $\Di = 0$.}]{{%
	\includegraphics[width=0.5\textwidth]{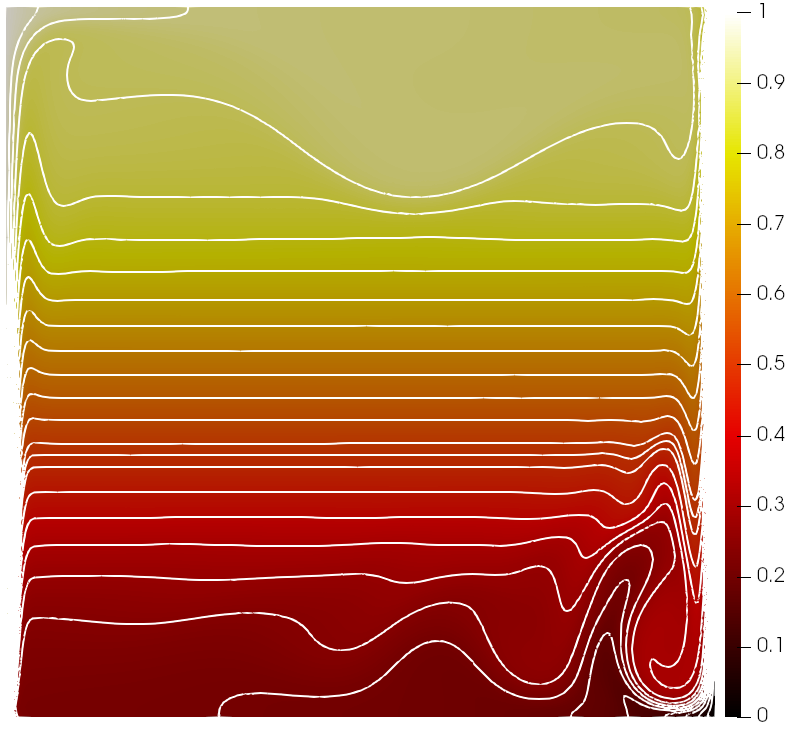}%
	}}\\
	\caption{Streamlines and temperature contours for the heated cavity with temperature dependent viscosity and heat conductivity with $\Gr= 10^8$.}%
	\label{fig:Grashof2d}
\end{figure}

\begin{table}
\centering
\captionsetup{justification=centering}
\begin{tabular}{c c c  c c c c}
\toprule
\multirow{2}{*}{$\Di$} &  \multirow{2}{*}{\# refs}& \multirow{2}{*}{\# dofs} & \multicolumn{4}{c}{$\Gr$ }\\ [0.1ex]
  &  & & $2.52\times 10^5$ & $6.30\times 10^5$ & $9.45\times 10^5$ & $1.26\times 10^6$  \\
\midrule
\multirow{2}{*}{0} & 1 & $3.2\times 10^5$ & 3.33 (3) & 4 (3) & 4.5 (2) & 9 (2) \\
									 & 2 & $2.6\times 10^6$ & 6 (3) & 4.5 (3) & 3.5 (2) & 3.5 (2) \\
 \midrule
\multirow{2}{*}{0.6} & 1 & $3.2\times 10^5$ & 3.33 (3) & 4 (3) & 4 (2) & 10.5 (2) \\
										 & 2 & $2.6\times 10^6$ & 4.33 (3) & 5 (3) & 4.5 (2) & 4.5 (2) \\
 \midrule
\multirow{2}{*}{1.3} & 1 & $3.2\times 10^5$ & 3.33 (3) & 4 (3) & 4 (2) & 10.5 (2) \\
										 & 2 & $2.6\times 10^6$ & 6 (3) & 4.5 (3) & 4.5 (2) & 4 (2) \\
 \midrule
\multirow{2}{*}{2} & 1 & $3.2\times 10^5$ & 3 (4) & 4 (3) & 4.5 (2) & 12 (2) \\
									 & 2 & $2.6\times 10^6$ & 6 (4) & 4.5 (3) & 3.5 (2) & 3.5 (2) \\
  \bottomrule
\end{tabular}\\
\caption{\label{tb:Grashof3D_P1}
Average number of Krylov iterations per Newton step as $\Gr$ increases for the 3D problem (P1) with $\Pr = 1$ and $k=3$, obtained using 2 multigrid cycles with 4 relaxation sweeps.}
\end{table}

\begin{table}
\centering
\captionsetup{justification=centering}
\begin{tabular}{c c c  c c c c}
\toprule
\multirow{2}{*}{$\Di$} &  \multirow{2}{*}{\# refs}& \multirow{2}{*}{\# dofs} & \multicolumn{4}{c}{$\Gr$ }\\ [0.1ex]
  &  & & $2.52\times 10^5$ & $6.30\times 10^5$ & $9.45\times 10^5$ & $1.26\times 10^6$  \\
\midrule
\multirow{2}{*}{0} & 1 & $3.2\times 10^5$ & 3.67 (4) & 4 (2) & 5 (2) & 13.5 (2) \\
									 & 2 & $2.6\times 10^6$ & 5 (4) & 5.5 (2) & 5.5 (2) & 5.5 (2) \\
 \midrule
\multirow{2}{*}{0.6} & 1 & $3.2\times 10^5$ & 3.33 (4) & 4 (3)  & 4.5 (2)  & 14 (2)  \\
  & 2 & $2.6\times 10^6$ & 4.33 (4)  & 5.5 (3)  & 4.5 (2)  & 5 (2)  \\
 \midrule
\multirow{2}{*}{1.3} & 1 & $3.2\times 10^5$ & 3.33 (4)  & 4 (3)  & 5 (2)  & 16.5 (2)  \\
  & 2 & $2.6\times 10^6$ & 6 (4)  & 4.5 (3)  & 4.5 (2)  & 4.5 (2)  \\
 \midrule
\multirow{2}{*}{2} & 1 & $3.2\times 10^5$ & 3.67 (4)  & 4 (3)  & 5 (2)  & 13.5 (2)  \\
  & 2 & $2.6\times 10^6$ & 6 (4)  & 4.5 (3)  & 3.5 (2)  & 4 (2)  \\
  \bottomrule
\end{tabular}\\
\caption{\label{tb:Grashof3D_P2}
Average number of Krylov iterations per Newton step as $\Gr$ increases for the 3D problem (P2) with $\Pr = 1$ and $k=3$, obtained using 2 multigrid cycles with 4 relaxation sweeps.}
\end{table}

\begin{table}
\centering
\captionsetup{justification=centering}
\begin{tabular}{c c c  c c c c}
\toprule
\multirow{2}{*}{$\Di$} &  \multirow{2}{*}{\# refs}& \multirow{2}{*}{\# dofs} & \multicolumn{4}{c}{$\Gr$ }\\ [0.1ex]
  &  & & $2.52\times 10^5$ & $6.30\times 10^5$ & $9.45\times 10^5$ & $1.26\times 10^6$  \\
\midrule
\multirow{2}{*}{0} & 1 & $3.2\times 10^5$ & 3.67 (4)  & 5 (3)  & 7.5 (2)  & 19 (2)  \\
  & 2 & $2.6\times 10^6$ & 5 (4)  & 6 (3)  & 6 (2)  & 9.5 (2)  \\
 \midrule
\multirow{2}{*}{0.6} & 1 & $3.2\times 10^5$ & 3.33 (4)  & 4 (3)  & 4 (2)  & 10.5 (2)  \\
  & 2 & $2.6\times 10^6$ & 4.33 (4)   & 5.5 (3)  & 4.5 (2)  & 7 (2)  \\
 \midrule
\multirow{2}{*}{1.3} & 1 & $3.2\times 10^5$ & 3.33 (5)  & 4 (3)  & 10 (2)  & 28.5 (2)  \\
  & 2 & $2.6\times 10^6$ & 6 (5)  & 4.5 (3)  & 4.5 (2)  & 4 (2)  \\
 \midrule
\multirow{2}{*}{2} & 1 & $3.2\times 10^5$ & 3 (5)  & 4 (3)  & 11.5 (2)  & 41.5 (2)  \\
  & 2 & $2.6\times 10^6$ & 6 (5)  & 4.5 (3)  & 3.5 (2)  & 3.5 (2)  \\
  \bottomrule
\end{tabular}\\
\caption{\label{tb:Grashof3D_P3}
Average number of Krylov iterations per Newton step as $\Gr$ increases for the 3D problem (P3) with $\Pr = 1$ and $k=3$, obtained using 2 multigrid cycles with 4 relaxation sweeps.}
\end{table}

\begin{figure}
\centering
\subfloat[C][{\centering Problem (P1).}]{{%
	\includegraphics[width=0.5\textwidth]{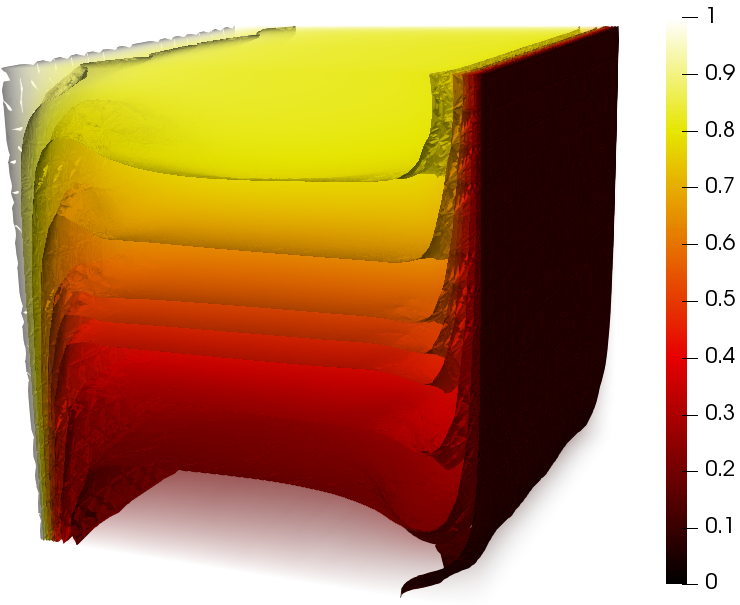}%
	}}%
	\subfloat[D][{\centering Problem (P3).}]{{%
	\includegraphics[width=0.5\textwidth]{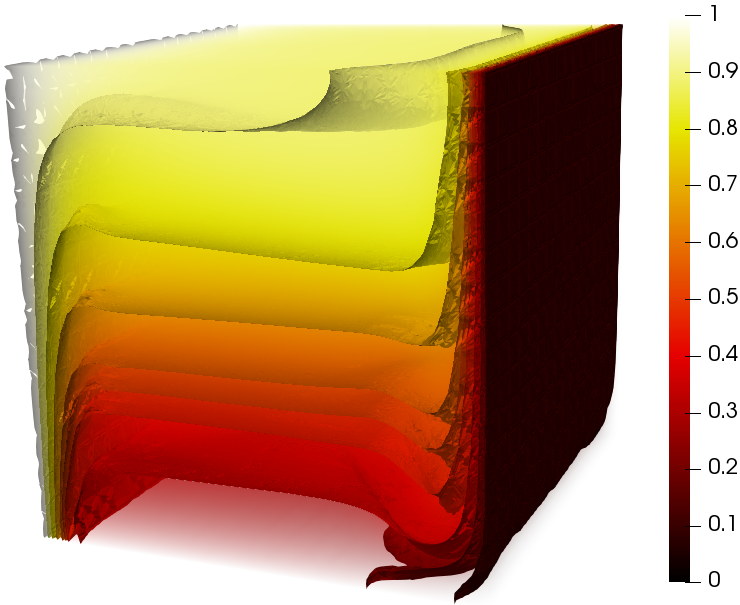}%
	}}%
	\caption{Temperature contours for the 3D heated cavity with $\Gr = 1.26\times 10^{6}$.}%
	\label{fig:Rayleigh3d}
\end{figure}

\begin{table}
\centering
\captionsetup{justification=centering}
\begin{tabular}{c c c c  c c c c}
\toprule
 \multirow{2}{*}{$\Di$} & \multirow{2}{*}{$k$} &  \multirow{2}{*}{\# refs}& \multirow{2}{*}{\# dofs} & \multicolumn{4}{c}{$\Ra$ }\\ [0.1ex]
			& &  & & $1$ & $1000$ & $20000$ & $30000$  \\
\midrule
 \multirow{4}{*}{$0$} & \multirow{2}{*}{2} & 1 & $5.5\times 10^4$  & 4.5 (4) & 7.2 (10) & 23.25 (4) & 28 (4) \\
											&  & 2 & $2.2\times 10^5$  & 5 (4) & 6.75 (4) & 6.67 (3) & 38 (4) \\
% \midrule
											& \multirow{2}{*}{3} & 1 & $1.2\times 10^5$ & 2.75 (4) & 4.44 (9) & 12.75 (4) & 15.25 (4) \\
											&  & 2 & $4.7\times 10^5$ & 2.75 (4) & 4.7 (10) & 12 (4) & 14.5 (4) \\
\midrule
 \multirow{4}{*}{$2.0$}& \multirow{2}{*}{2} & 1 & $5.5\times 10^4$ & 3 (4) & 7.22 (9) & 18 (4) & 21.25 (4) \\
											 &  & 2 & $2.2\times 10^5$ & 2 (4) & 7.55 (11)  & 20.25 (4) & 25 (4) \\
											 & \multirow{2}{*}{3} & 1 & $1.2\times 10^5$ & 1 (4) & 4.1 (10) & 10.5 (4) & 12.25 (4) \\
											 &  & 2 & $4.7\times 10^5$ & 1 (4) & 4.33 (12) & 10 (4) & 14.25 (4) \\
%\multirow{2}{*}{$k$} &  \multirow{2}{*}{\# refs}& \multirow{2}{*}{\# dofs} & \multicolumn{4}{c}{$\Ra$ }\\ [0.1ex]
%  & & & 5000 & 10000 & 15000 & 20000  \\
%\midrule
%\multirow{3}{*}{2} & 1 & $1.8\times 10^4$ & 3.64 & 5.25 & 6.42 & 6.38 \\
%  & 2 & $7.2\times 10^4$ & 3.78 & 5.78 & 7 & 9.75 \\
%  & 3 & $2.9\times 10^5$ & 3.22 & 4.8 & 6.3 & 8.3 \\
% \midrule
%\multirow{3}{*}{3} & 1 & $7.3\times 10^4$ & 2.57 & 3.11 & 3.5 & 4.25 \\
%  & 2 & $1.6\times 10^5$ & 2.5 & 2.8 & 3.33 & 4.75 \\
%  & 3 & $6.5\times 10^5$ & 1.9 & 2.22 & 2.44 & 4 \\
  \bottomrule
\end{tabular}\\
\caption{\label{tb:power_law}
Average number of Krylov iterations per Newton step as $\Ra$ increases for the constitutive relation \eqref{eq:example_power_law}  with $r=1.6$, obtained using \txtb{one} multigrid cycle with \txtb{5} relaxation sweeps.}
\end{table}

\begin{figure}
\centering
\subfloat[C][{\centering Streamlines for $r = 2$.}]{{%
	\includegraphics[width=0.5\textwidth]{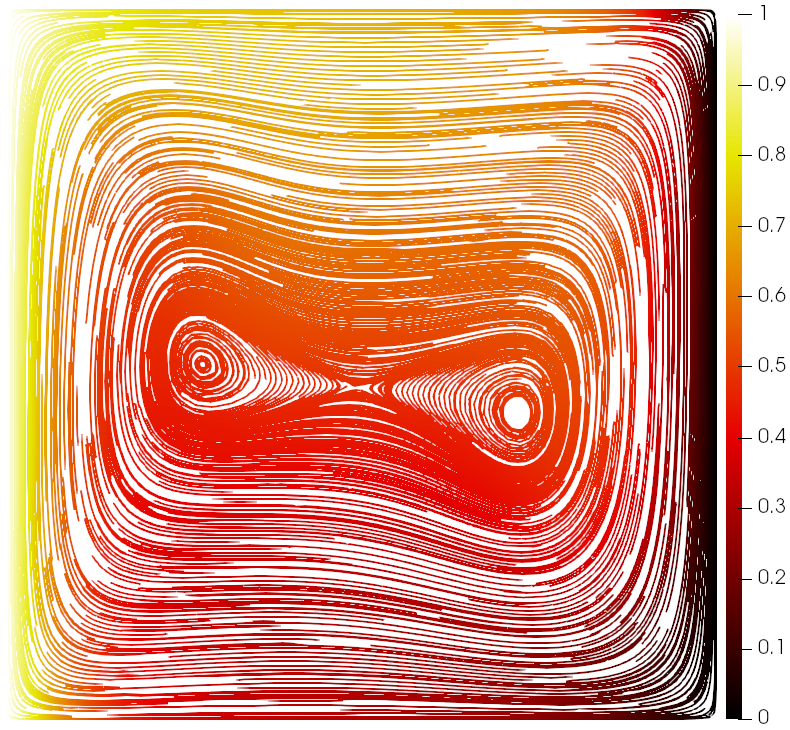}%
	}}%
	\subfloat[D][{\centering Streamlines for $r = 1.6$.}]{{%
	\includegraphics[width=0.5\textwidth]{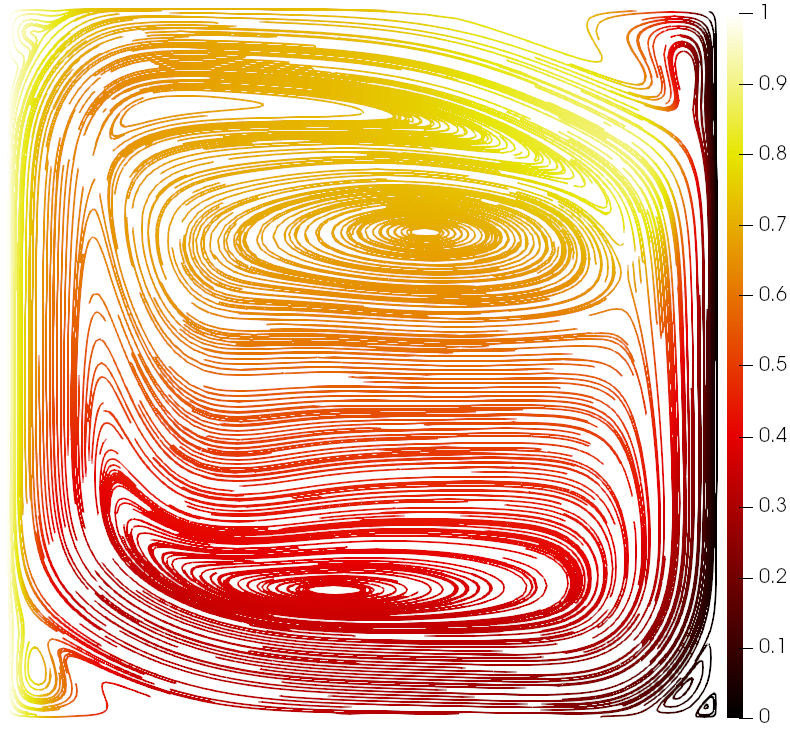}%
	}}\\
\subfloat[A][{\centering Temperature contours for $r = 2$.}]{{%
	\includegraphics[width=0.5\textwidth]{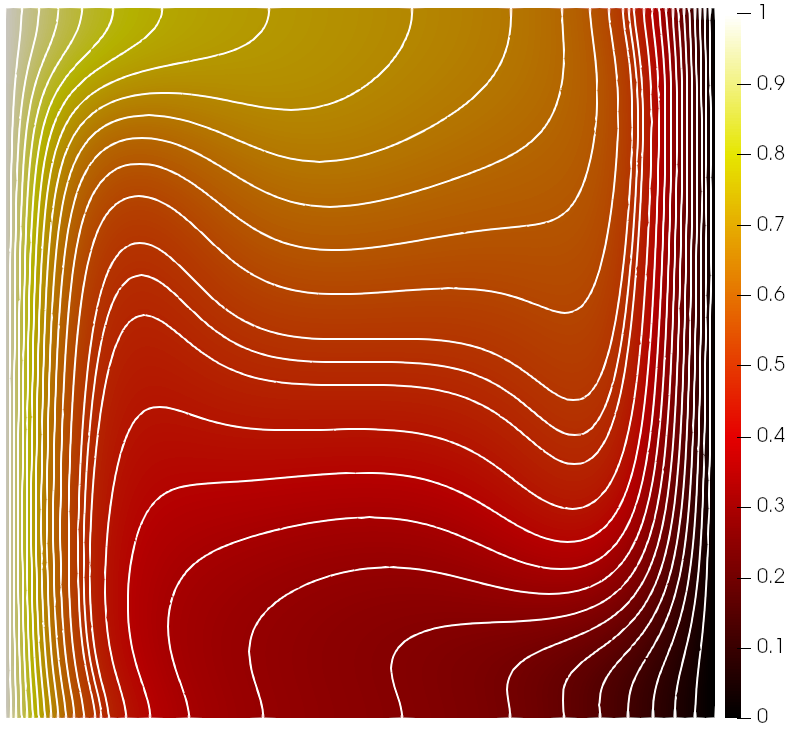}%
	}}%
	\subfloat[B][{\centering Temperature contours for $r = 1.6$.}]{{%
	\includegraphics[width=0.5\textwidth]{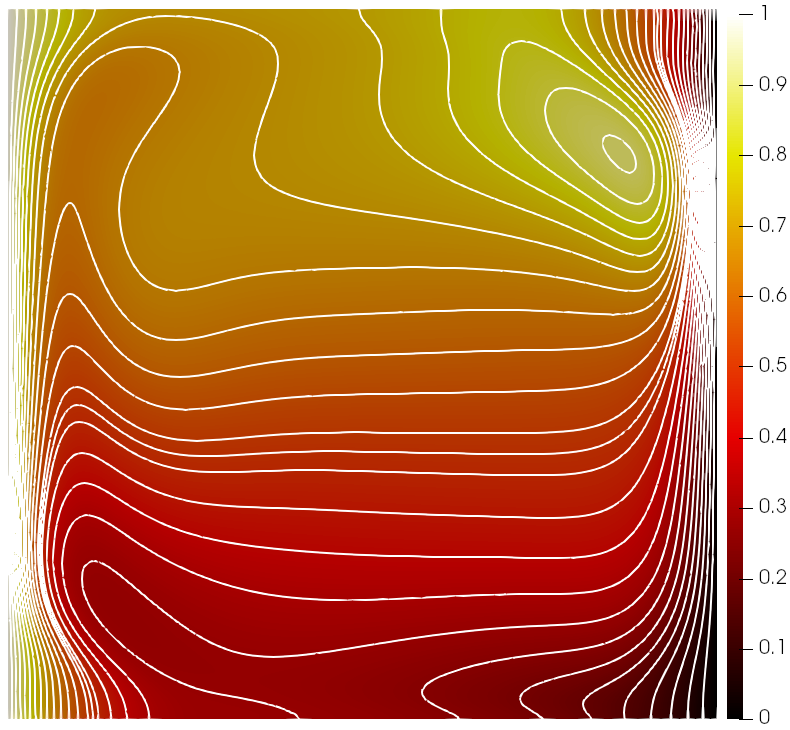}%
	}}\\
	\caption{Streamlines and temperature contours for the heated cavity with the power-law constitutive relation \eqref{eq:example_power_law}, $\Ra = 3\times 10^4$ and $\Di = 2.0$.}%
	\label{fig:Rayleigh2d}
\end{figure}

{\color{black}
	\subsection{Bingham--activated Euler channel}\label{sec:Euler_channel}

	The constitutive relation \eqref{eq:implicit_CR} is very general. As previously mentioned, it encompasses both the Bingham constitutive relation for viscoplastic fluids ($\hat{\sigma}\equiv 0$), and its activated Euler counterpart ($\hat{\tau}\equiv 0$) that was proposed only very recently \cite{Blechta2019}. Since the rheological parameters are allowed to depend on the temperature, it is interesting to consider a scenario for which the constitutive relation transitions between these two constrasting types of response as a function of temperature. We consider such an example in this section, to stress-test the generality of \eqref{eq:implicit_CR}, and its ability to capture this transition. We believe this to be the first time that the transition between Bingham and activated Euler relations has been simulated.

	The implicit constitutive relation \eqref{eq:implicit_CR} is not described by a Fr\'echet differentiable function and therefore Newton's method cannot be directly applied; for this reason we will employ in this example the following regularised version of the constitutive relation:
\begin{equation}\label{eq:EulerBingham_reg}
	\BG(\BS,\BD(\bu),\theta) := 2\hat{\mu}(\theta)\frac{\max_\varepsilon \{0,|\BD(\bu)|-\hat{\sigma}(\theta)\}}{|\BD(\bu)|}\BD(\bu)
- \frac{\max_\varepsilon \{0, |\BS| - \hat{\tau}(\theta)\}}{|\BS|}\BS,
\end{equation}
where the regularised maximum function is defined for $\varepsilon >0$ as:
\begin{equation}\label{eq:regularised_max}
	\mathrm{max}_\varepsilon \{x,y\} := \frac{1}{2}\left(x + y + \sqrt{(x-y)^2 + \varepsilon^2}\right)\qquad x,y\in \mathbb{R}.
\end{equation}
The rheological parameters will be taken to be of the form
\begin{gather*}
	\hat{\mu}(\theta) = \mu_*, \qquad \mu_*>0,\\
	\hat{\sigma}(\theta) = \mathrm{max}_\varepsilon\left\{0, \mathrm{min}_\varepsilon\left\{\sigma_*,  \frac{\sigma_*}{\theta_2 - \theta_1}(\theta- \theta_2) + \sigma_*  \right\}\right\} ,\qquad \sigma_*,\theta_2,\theta_1 >0,\, \theta_1\neq \theta_2,\\
	\hat{\tau}(\theta) = \mathrm{max}_\varepsilon\left\{0, \mathrm{min}_\varepsilon\left\{\tau_*,  \frac{\tau_*}{\theta_4 - \theta_3}(\theta- \theta_4) + \tau_*  \right\}\right\},  \qquad \tau_*,\theta_3,\theta_4>0,\, \theta_3\neq \theta_4,
\end{gather*}
where now the regularised minimum function is defined for $\varepsilon>0$ as:
\begin{equation}\label{eq:regularised_min}
	\mathrm{min}_\varepsilon \{x,y\} := \frac{1}{2}\left(x + y - \sqrt{(x-y)^2 + \varepsilon^2}\right)\qquad x,y\in \mathbb{R}.
\end{equation}
In particular, this means e.g.\ that, if $\theta_1>\theta_2$, then the activation parameter $\hat{\sigma}$ will vanish for $\theta>\theta_1$, and will equal $\sigma_*$ for $\theta<\theta_2$.

In this example we consider the system \eqref{eq:PDE_Forced} with $\rho_* = 1 = c_v$. The problem is posed on the straight channel $\Omega := (0,30)\times (-1,1)$ with the following boundary datum for the temperature:
\begin{equation*}
\theta = \left\{
	\begin{array}{cc}
		\theta_H, & \textrm{on }\Gamma_H,\\
		0, & \textrm{on }\Gamma_C,\\
		-\frac{\theta_H}{10}x_1 + 2\Theta_H, & \text{on }\partial\Omega \setminus(\Gamma_H\cup \Gamma_C),
	\end{array}
	\right.
\end{equation*}
where $\theta_H>0$, and $\Gamma_H:=\{(x_1,x_2)^\top \in \partial\Omega \, :\, x_1\leq 10\}$ and $\Gamma_C:=\{(x_1,x_2)^\top\in \partial\Omega \, :\, x_1\geq 20\}$.  As for the velocity, we impose the following boundary conditions:
\begin{equation}\label{eq:bcs_channel}
\begin{gathered}
	(\BS - p\bm{I})(1,0)^\top \cdot (1,0)^\top= 0,\, \bu\cdot (0,1)^\top =0
	\text{  on }\Gamma_{\textrm{out}},\quad
	\bu = \bu_{B} \text{  on }\Gamma_{\textrm{in}},\\
	\bu = \bm{0} \text{  on }\partial\Omega \setminus (\Gamma_{\textrm{in}}\cup\Gamma_{\textrm{out}}),
\end{gathered}
\end{equation}
where $\Gamma_{\textrm{in}}:= \{x_1=0\}$, $\Gamma_{\textrm{out}}:= \{x_1=30\}$, and $\bu_B$ is the fully developed Poiseuille flow for the isothermal problem with the Bingham constitutive relation corresponding to the yield-stress $\tau_*$, for which the exact solution is available (see e.g.\ \cite{Grinevich2009}).

The constitutive relation \eqref{eq:implicit_CR} describes such contrasting and low-regularity behaviour, especially when both the Bingham and Euler regimes are present, that higher order discretisations (such as those employing the Scott--Vogelius element) are not appropriate. In preliminary experiments with the discretisation proposed in Section \ref{sec:discretisation}, we observed oscillations near the transition regions that spoiled the convergence of the nonlinear solver. For this section we therefore switch to a discretisation based on a piecewise-constant approximation of the stress (and symmetric velocity gradient). The 4-field formulation \eqref{eq:FEFormulation2} was discretised with a $\mathbb{P}_1^d$--$\mathbb{P}_1$ element for the velocity and pressure, $\mathbb{P}_1$ elements for the temperature, and piecewise constant elements $\mathbb{P}^{d\times d}_0$ for the stress. In order to guarantee the stability of the formulation, the following term is added to the right-hand-side of the divergence constraint \eqref{eq:discrete_mass_B0} (c.f.\ \cite{Hughes1986}):
\begin{equation*}
 2 h_n^2 \int_\Omega (\diver (\bu^n\otimes \bu^n) + \nabla p^n)\cdot \nabla q
\end{equation*}
We solve the resulting linear systems with the sparse direct solver from MUMPS, without an augmented Lagrangian term ($\gamma = 0$).

Figure \ref{fig:euler_bingham} shows the solution obtained with 202372 degrees of freedom, and the values $\mu_* = \tfrac{1}{2}$, $\sigma_*=0.1$, $\tau_*=0.025$, $\varepsilon = 10^{-4}$, $\theta_1 = 3$, $\theta_2 = 1$, $\theta_3=7$, $\theta_4=9$, and $\theta_H =10$. The figure includes a plot of the effective viscosity, which for this problem we define as:
\begin{equation}
	\mu_{\mathrm{eff}}(\BS,\bu,\theta) := 2\hat{\mu}(\theta)  \frac{\max_\varepsilon \{0,|\BD(\bu)|-\hat{\sigma}(\theta)\}}{|\BD(\bu)|} \left(10^{-12} +  \frac{\max_\varepsilon \{0,|\BS|-\hat{\tau}(\theta)\}}{|\BS|} \right)^{-1}.
\end{equation}

The solution in Figure \ref{fig:euler_bingham} satisfies the Bingham constitutive relation in the hotter regions, and eventually transitions to an activated Euler model when the temperature decreases. Although it is unclear whether a model with such a transition between the viscoplastic and inviscid regimes could be of practical use, the example showcases the ability of implicit constitutive relations to capture enormously contrasting behaviour; note in particular that the effective viscosity transitions over 14 orders of magnitude.

\begin{figure}
\centering
	\subfloat[D][{\centering Temperature}]{{%
	\includegraphics[width=\textwidth]{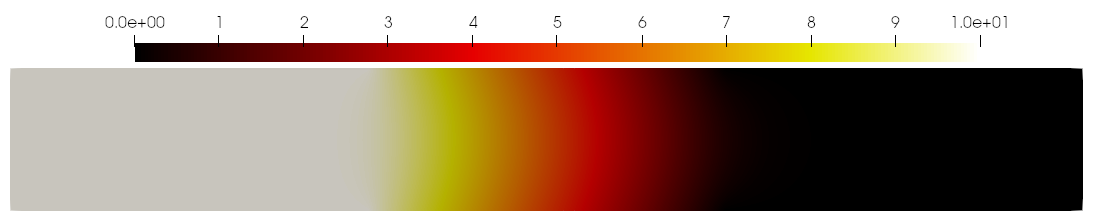}%
	}}\\
	\subfloat[E][{\centering Velocity Magnitude}]{{%
	\includegraphics[width=\textwidth]{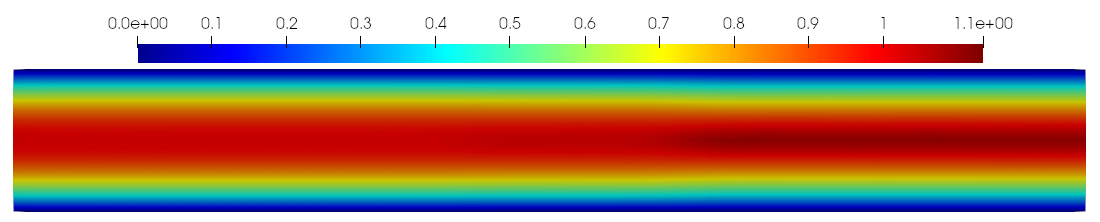}%
	}}\\
\subfloat[A][{\centering Effective Viscosity}]{{%
	\includegraphics[width=\textwidth]{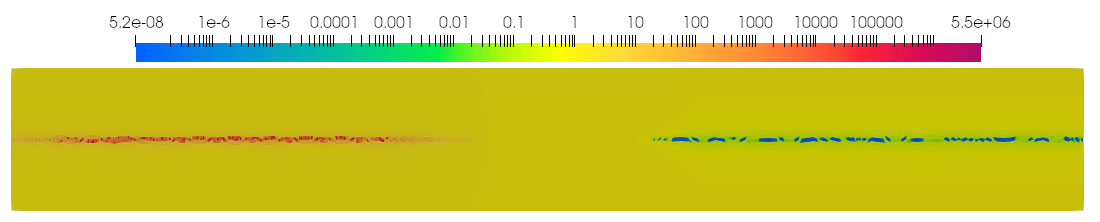}%
	}}%
	\caption{Solution for the Bingham--Euler channel problem.}%
	\label{fig:euler_bingham}
\end{figure}
}

\subsection{Bingham \txtb{fluid} in a cooling channel}
\txtb{In this example we study whether the Scott--Vogelius discretisation and multigrid preconditioner we have proposed can effectively simulate Bingham fluids with less extreme behaviour. We therefore return to the discretisation and solver described in Sections \ref{sec:discretisation} and \ref{sec:AL_preconditioner}.}
Let $\Omega := (0,40)\times (-1,1)$, and consider the following boundary conditions for the temperature:
\begin{equation*}
	\nabla\theta \cdot\bm{n} = 0 \quad\text{on }\partial\Omega\setminus(\Gamma_H\cup\Gamma_C),\quad
	\theta = \left\{\begin{array}{cc}
			\theta_H, & \text{on }\Gamma_H,\\
			0, & \text{on }\Gamma_C,
	\end{array}\right.
\end{equation*}
where $\theta_H>0$, and $\Gamma_H := \{(x_1,x_2)^\top \in \partial\Omega : x_1 \leq 10\}$ and $\Gamma_C := \{(x_1,x_2)^\top\in \partial\Omega : x_2 \in \{-1,1\},\, 10 < x_1\}$.  The Bingham constitutive relation for viscoplastic fluids is obtained by setting $\hat{\sigma} \equiv 0$ in \eqref{eq:implicit_CR}. \txtb{As mentioned in Section \ref{sec:Euler_channel}, this relation is not described by a Fr\'{e}chet differentiable function and some regularisation is needed.} The non-dimensional form then reads: % therefore Newton's method cannot be directly applied; for this reason we will introduce a regularised version of the constitutive relation. In this example we will consider a forced convection regime, in which the buoyancy effects are not taken into account. The non-dimensional form of the system then reads:
\begin{subequations}\label{eq:PDE_Bingham}
	\begin{alignat}{2}
			- \diver \BS + \Re \diver (\bu&\otimes\bu)
			+ \nabla p = 0
			\quad & &\text{ in }\Omega,\\
			\diver\bu &= 0\quad & &\text{ in }\Omega,\\
			- \frac{1}{\Pe} \diver (\nabla\theta)
			+ \diver(\bu & \theta )
			= \frac{\Br}{\Pe}\BS\twodot \BD(\bu)\quad & &\text{ in }\Omega,\\
			\sqrt{\varepsilon^2 + |\BD(\bu)|^2}\BS = (\Bn\, \tau(\theta)& + 2\mu(\theta)|\BD(\bu)|)\BD(\bu) & & \text{ in }\Omega,
	\end{alignat}
\end{subequations}
where $\varepsilon>0$ is the regularisation parameter, and the Reynolds, P\'{e}clet, Bingham and Brinkman numbers are defined as
\begin{equation}
	\Re = \frac{\rho_0 U R}{\nu_0},\quad
	\Pe = \frac{\rho_0 c_p U R}{\kappa_0},\quad
	\Bn = \frac{\tau_0 R}{\nu_0 U},\quad
	\Br = \frac{\nu_0 U^2}{\kappa_0 \theta_H},\quad
\end{equation}
where $R$ is the radius of the pipe, $U$ is the average velocity at the inlet and $\tau_0$ is the value of the yield stress at the inlet. Two choices for the (non-dimensional) viscosity and yield stress are considered here:
\begin{alignat*}{2}
	\text{Problem (Q1): } \mu(\theta) &:= a_1 \theta + a_2,\quad  && \tau(\theta) := 1. \\
	\text{Problem (Q2): } \mu(\theta) &:= 1,  && \tau(\theta) := b_1\theta + b_2.
\end{alignat*}
The values of $a_1$ and $a_2$ are chosen so that the viscosity is unity at the inlet and increases by a factor of 20 at the outlet (which means that the effective Bingham number decreases by the same factor). The constants $b_1$ and $b_2$ are such that the Bingham number is $1.5$ at the inlet, and $9$ at the outlet when a temperature drop of $15$ is applied. \txtb{For the velocity we impose the boundary condition \eqref{eq:bcs_channel}, with a Bingham number of $\Bn=1.5$ at the inlet.} %As for the velocity, we impose the following boundary conditions:
%\begin{gather*}
%	(\BS - p\bm{I})(1,0)^\top \cdot (1,0)^\top= 0,\, \bu\cdot (0,1)^\top =0
%	\text{  on }\Gamma_{\textrm{out}},\quad
%	\bu = \bu_{B} \text{  on }\Gamma_{\textrm{in}},\\
%	\bu = \bm{0} \text{  on }\partial\Omega \setminus (\Gamma_{\textrm{in}}\cup\Gamma_{\textrm{out}}),
%\end{gather*}
%where $\Gamma_{\textrm{in}}:= \{x_1=0\}$, $\Gamma_{\textrm{out}}:= \{x_1=40\}$, and $\bu_B$ is the fully developed Poiseuille flow for the isothermal problem with $\Bn = 1.5$, for which the exact solution is available (see e.g.\ \cite{Grinevich2009}).
In order to obtain better initial guesses for Newton's method, secant continuation was employed: given two previously computed solutions $\bz_1,\bz_2$ corresponding to the parameters $\varepsilon_1,\varepsilon_2$, respectively, the initial guess for Newton's method at $\varepsilon$ is chosen as
\begin{equation*}
	\frac{\varepsilon-\varepsilon_2}{\varepsilon_2-\varepsilon_1}(\bz_2-\bz_1) + \bz_2.
\end{equation*}

For this (arguably more complex) problem, the multigrid algorithm for the top block ceases to be effective for values smaller than $\varepsilon=0.001$; while the relaxation performs reasonably well, the prolongation operator is not robust with respect to $\varepsilon$. Nevertheless, the augmented Lagrangian strategy is useful, as it allows us to apply a sparse direct solver to the top block instead of the entire matrix; Tables \ref{tb:channel_visc} and \ref{tb:channel_ystress} show the average number of Krylov iterations per Newton step obtained when using a sparse direct solver for the top block, with the same Schur complement approximation. Extending the multigrid algorithm to this more challenging problem would likely require substantial theoretical developments, i.e.~the extension of Sch\"{o}berl's framework to non-symmetric and non-coercive problems.

\begin{table}
\centering
\captionsetup{justification=centering}
\begin{tabular}{c c c  c c c c}
\toprule
\multirow{2}{*}{$\gamma$} &  \multirow{2}{*}{\# refs}& \multirow{2}{*}{\# dofs} & \multicolumn{4}{c}{$\varepsilon$ }\\ [0.1ex]
			  & & & $1\times 10^{-3}$ & $1\times 10^{-4}$  & $2\times 10^{-5}$ & $1\times 10^{-5}$  \\
\midrule
\multirow{4}{*}{$10^3$} & 1 & $2.7\times 10^4$ & 12.8 (5) & 22 (3) & 51 (1) & 48 (1) \\
												& 2 & $1.0\times 10^5$ & 14.8 (6) & 33.5 (4) & 55 (2) & 49 (1) \\
												& 3 & $4.3\times 10^5$ & 13.5 (6) & 17 (4) & 25 (2) & 17 (1) \\
												& 4 & $1.7\times 10^6$ & 11.71 (6) & 8.8 (4) & 13 (2) & 12 (1) \\
 \midrule
\multirow{4}{*}{$10^5$} & 1 & $2.7\times 10^4$ & 2.6 (6) & 2 (3) & 2 (1) & 1.33 (1) \\
												& 2 & $1.0\times 10^5$ & 2.6 (6) & 2.25 (4) & 1.4 (2) & 1.15 (1) \\
												& 3 & $4.3\times 10^5$ & 2 (6) & 1.33 (4) & 1.14 (2) & 1 (1) \\
												& 4 & $1.7\times 10^6$ & 1.75 (6) & 1.33 (4) & 1.15 (2) & 1.07 (1) \\
  \bottomrule
\end{tabular}\\
\caption{\label{tb:channel_visc}
Average number of Krylov iterations per Newton step as $\varepsilon$ decreases for Problem (Q1) with $k=2$, $\Pe=10$, $\theta_H=10$, $\Br=0.1$.}
\end{table}

\begin{table}
\centering
\captionsetup{justification=centering}
\begin{tabular}{c c c  c c c c}
\toprule
\multirow{2}{*}{$\gamma$} &  \multirow{2}{*}{\# refs}& \multirow{2}{*}{\# dofs} & \multicolumn{4}{c}{$\varepsilon$ }\\ [0.1ex]
			  & & & $1\times 10^{-3}$ & $1\times 10^{-4}$  & $5\times 10^{-5}$ & $2\times 10^{-5}$  \\
\midrule
\multirow{4}{*}{$10^3$} & 1 & $2.7\times 10^4$ & 12 (7) & 17.5 (6) & 26.6 (2) & * \\
												& 2 & $1.0\times 10^5$ & 11.3 (8) & 16.25 (6) & 17 (2) & 23 (1) \\
												& 3 & $4.3\times 10^5$ & 12.88 (8) & 13.67 (6) & 14.5 (2) & 13 (1) \\
												& 4 & $1.7\times 10^6$ & 6.48 (8) & * & * & * \\
 \midrule
\multirow{4}{*}{$10^5$} & 1 & $2.7\times 10^4$ & 1.78 (7) & 2.16 (7) & 1.75 (2) & 1.2 (1) \\
												& 2 & $1.0\times 10^5$ & 1.6 (7) & 1.3 (7) & 1.16 (1) & 1.07 (1) \\
												& 3 & $4.3\times 10^5$ & 1.78 (7) & 1.17 (7) & 1.05 (1) & 1 (1) \\
												& 4 & $1.7\times 10^6$ & 1.31 (7) & 1.13 (7) & 1.03 (1) & 1 (1) \\
  \bottomrule
\end{tabular}\\
\caption{\label{tb:channel_ystress}
Average number of Krylov iterations per Newton step as $\varepsilon$ decreases for Problem (Q2) with $k=2$, $\Pe=10$, $\theta_H=10$, $\Br=0$. The symbol * means that the maximum permitted number of nonlinear iterations was reached.}
\end{table}

Figures \ref{fig:channel_visc} and \ref{fig:channel_ystress} show the temperature field and the yielded/unyielded regions of the fluid. The results are qualitatively similar to those found in \cite{Vinay2005}, where an algorithm based on the augmented Lagrangian method was applied to a similar problem (neglecting the convective term and viscous dissipation). While it is known that a method based on regularisation, such as the one applied here, is not the most appropriate if one wishes to locate the exact position of the yield surfaces, it can still be useful to obtain the general features of the flow. For example, the solutions found here show no unyielded regions in the transition zone where the temperature field varies with the mean flow direction, which is the expected behaviour \cite{Vinay2005}.

\begin{figure}
\centering
	\subfloat[D][{\centering Temperature}]{{%
	\includegraphics[width=\textwidth]{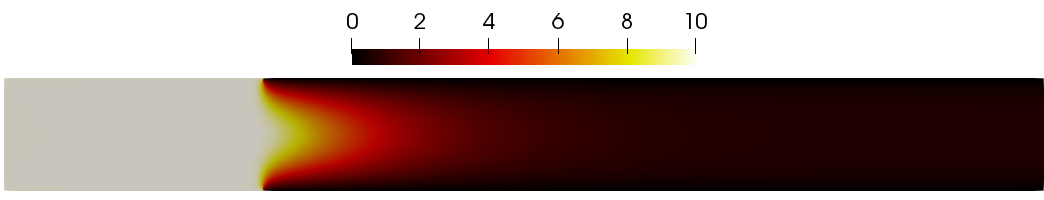}%
	}}\\
\subfloat[A][{\centering Magnitude of the symmetric velocity gradient}]{{%
	\includegraphics[width=\textwidth]{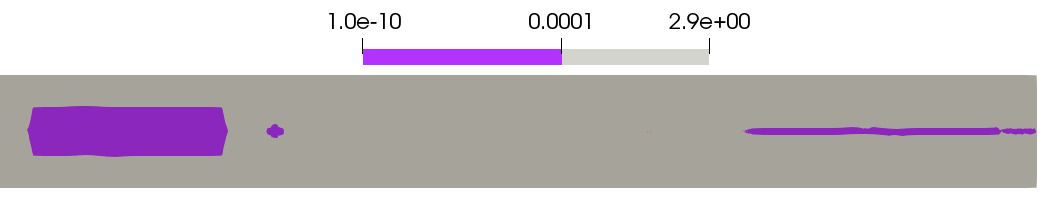}%
	}}%
	\caption{Temperature field and yielded regions for \txtb{a} Bingham \txtb{fluid} on a cooling channel (Problem (Q1)), with $\Pe=10$, $\theta_H=10$, $\Br=0.1$.}%
	\label{fig:channel_visc}
\end{figure}

\begin{figure}
\centering
	\subfloat[D][{\centering Temperature}]{{%
	\includegraphics[width=\textwidth]{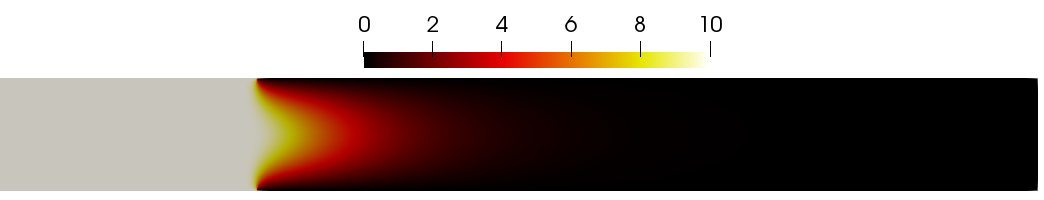}%
	}}\\
\subfloat[A][{\centering Magnitude of the symmetric velocity gradient.}]{{%
	\includegraphics[width=\textwidth]{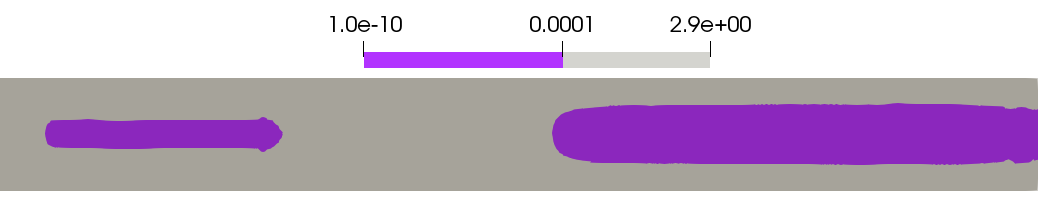}%
	}}%
	\caption{Temperature field and yielded regions for \txtb{a} Bingham \txtb{fluid} on a cooling channel (Problem (Q2)), with $\Pe = 10$, $\theta_H=10$, $\Br=0$.}%
	\label{fig:channel_ystress}
\end{figure}

\bibliographystyle{amsplain}

\begin{thebibliography}{10}

\bibitem{MUMPS:1}
P.~R. Amestoy, I.~S. Duff, J.~Coster and J.~-Y. L'Excellent.
\newblock A fully asynchronous multifrontal solver using distributed
  dynamic scheduling.
\newblock {\em SIAM J. Matrix Anal.}, 23(1):15--41, 2001.

\bibitem{Bacuta2006}
C.~Bacuta.
\newblock A unified approach for {U}zawa algorithms.
\newblock {\em SIAM J. Numer. Anal.}, 44(6):2633--2649, 2006.

\bibitem{PETSc}
S.~Balay, S.~Abhyankar, M.~F. Adams, J.~Brown, P.~Brune, K.~Buschelman,
  L.~Dalcin, V.~Eijhout, W.~D. Gropp, D.~Kaushik, M.~G. Knepley, L.~C. McInnes,
  K.~Rupp, S.~Smith, B.~F.~Zampini, H.~Zhang, and H.~Zhang.
\newblock {PETSc} users manual.
\newblock {\em Tech. Report ANL--95/11--Revision 3.8, Argonne National
  Laboratory}, 2017.
\newblock http://www.mcs.anl.gov/petsc.

\bibitem{Bayly1992}
B.~J. Bayly, C.~D. Levermore, and T.~Passot.
\newblock Density variations in weakly compressible flows.
\newblock {\em Phys. Fluids A}, 4(5):945--954, 1992.

\bibitem{Belenki:2012}
L.~Belenki, L.~Berselli, L.~Diening, and M.~R\r{u}{\v{z}}i{\v{c}}ka.
\newblock On the finite element approximation of $p$-{S}tokes systems.
\newblock {\em SIAM J. Numer. Anal.}, 50:373--397, 2012.

\bibitem{Benzi2006}
M.~Benzi and M.~A. Olshanskii.
\newblock An augmented {L}agrangian-‐based approach to the {O}seen problem.
\newblock {\em SIAM J. Sci. Comput.}, 28(6):2005--2113, 2006.

\bibitem{Blechta2019}
J.~Blechta, J.~M\'{a}lek, and K.~R. Rajagopal.
\newblock {O}n the classification of incompressible fluids and a mathematical
  analysis of the equations that govern their motion.
\newblock {\em SIAM J. Math. Anal.}, 52(2):1232--1289, 2020.

\bibitem{Boffi2013}
D.~Boffi, F.~Brezzi, and M.~Fortin.
\newblock {\em Mixed Finite Element Methods and Applications}.
\newblock Springer, 2013.

\bibitem{Boussinesq1903}
J.~Boussinesq.
\newblock {\em Th\'{e}orie Analytique de la Chaleur}.
\newblock Gauthier--Villars, Paris, 1903.

\bibitem{Bulicek2009}
M.~Bul\'{i}\v{c}ek, E.~Feireisl, and J.~M\'{a}lek.
\newblock A {N}avier-{S}tokes-{F}ourier system for incompressible fluids with
  temperature dependent material coefficients.
\newblock {\em Nonlinear Anal. Real World Appl.}, 10:992--1015, 2009.

\bibitem{Bulicek:2009}
M.~Bul\'{i}\v{c}ek, P.~Gwiazda, J.~M\'{a}lek, and
  A.~{\'{S}wierczewska-Gwiazda}.
\newblock {O}n steady flows of incompressible fluids with implicit
  power-law-like rheology.
\newblock {\em Adv. Calc. Var.}, 2:109--136, 2009.

\bibitem{Bulicek:2012}
M.~Bul\'{i}\v{c}ek, P.~Gwiazda, J.~M\'{a}lek, and
  A.~{\'{S}wierczewska-{G}wiazda}.
\newblock {O}n unsteady flows of implicitly constituted incompressible fluids.
\newblock {\em SIAM J. Math. Anal.}, 44(4):2756--2801, 2012.

\bibitem{Bulicek2009a}
M.~Bul\'{i}\v{c}ek, J.~M\'{a}lek, and K.~R. Rajagopal.
\newblock Mathematical analysis of unsteady flows of fluids with pressure,
  shear-rate, and temperature dependent material moduli that slip at solid
  boundaries.
\newblock {\em SIAM J. Math. Anal.}, 41(2):665--707, 2009.

\bibitem{Burman2008}
E.~Burman and A.~Linke.
\newblock Stabilized finite element schemes for incompressible flow using
  {S}cott--{V}ogelius elements.
\newblock {\em Appl. Numer. Math.}, 58(11):1704--1719, 2008.

\bibitem{Casado-Diaz}
J.~Casado-D\'{i}az, T.~Chac\'{o}n-Rebollo, V.~Girault, M.~G\'{o}mez-M\'{a}rmol,
  and F.~Murat.
\newblock {Finite elements approximation of second order linear elliptic
  equations in divergence form with right-hand side in {$L^1$}}.
\newblock {\em Numer. Math.}, 105:337--374, 2007.

\bibitem{Crouzeix1973}
M.~Crouzeix and P.~A. Raviart.
\newblock {Conforming and nonconforming finite element methods for solving the
  stationary Stokes equations I}.
\newblock {\em ESAIM: M2AN}, pages 33--75, 1973.

\bibitem{Diening:2013}
L.~Diening, C.~Kreuzer, and E.~S\"{u}li.
\newblock {F}inite element approximation of steady flows of incompressible
  fluids with implicit power-law-like rheology.
\newblock {\em SIAM J. Numer. Anal.}, 51(2):984--1015, 2013.

\bibitem{Douglas1976}
J.~Douglas and T.~Dupont.
\newblock Interior penalty procedures for elliptic and parabolic {G}alerkin
  methods.
\newblock {\em Computing Methods in Applied Sciences. Lecture Notes in
  Physics}, vol. 58, 1976.

\bibitem{Elman2006}
H.~Elman, V.~E. Howle, J.~Shadid, R.~Shuttleworth, and R.~Tuminaro.
\newblock Block preconditioners based on approximate commutators.
\newblock {\em SIAM J. Sci. Comput.}, 27(5):1651--1668, 2006.

\bibitem{Elman2014}
H.~C. Elman, D.~J. Silvester, and A.~J. Wathen.
\newblock {\em Finite Elements and Fast Iterative Solvers: With Applications in
  Incompressible Fluid Dynamics}.
\newblock Oxford University Press, second edition, 2014.

\bibitem{Emery1999}
A.~F. Emery and J.~W. Lee.
\newblock The effects of property variations on natural convection in a square
  enclosure.
\newblock {\em J. Heat Transfer}, 121:57--62, 1999.

\bibitem{FarrellGazca2019}
P.~E. Farrell and P.~A. Gazca-Orozco.
\newblock {A}ugmented {L}agrangian preconditioner for implicitly-constituted
  non-{N}ewtonian incompressible flow.
\newblock {\em SIAM J. Sci. Comput.}, 42(6):B1329--B1349, 2020.

\bibitem{Farrell2019}
P.~E. Farrell, P.~A. Gazca-Orozco, and E.~S\"{u}li.
\newblock Numerical analysis of unsteady implicitly constituted incompressible
  fluids: 3-field formulation.
\newblock {\em SIAM J. Numer. Anal.}, 58(1):757--787, 2020.

\bibitem{Farrell}
P.~E. Farrell, M.~G. Knepley, L.~E. Mitchell, and F.~Wechsung.
\newblock {PCPATCH}: software for the topological construction of multigrid
  relaxation methods.
\newblock {\em ArXiv Preprint: 1912.08516}, 2019.
\newblock In review.

\bibitem{Farrell2020}
P.~E. Farrell, L.~Mitchell, L.~R. Scott, and F.~Wechsung.
\newblock {A} {R}eynolds-robust preconditioner for the {R}eynolds-robust
  {S}cott--{V}ogelius discretization of the stationary incompressible
  {N}avier--{S}tokes equations.
\newblock {\em ArXiv Preprint: 2004.09398}, 2020.

\bibitem{Farrella}
P.~E. Farrell, L.~Mitchell, L.~R. Scott, and F.~Wechsung.
\newblock Robust multigrid methods for nearly incompressible elasticity using
  macro elements.
\newblock {\em ArXiv Preprint: 2002.02051}, 2020.

\bibitem{Farrell2019a}
P.~E. Farrell, L.~Mitchell, and F.~Wechsung.
\newblock An augmented {L}agrangian preconditioner for the 3{D} stationary
  incompressible {N}avier--{S}tokes equations at high {R}eynolds number.
\newblock {\em SIAM J. Sci. Comput.}, 41(5):A3073--A3096, 2019.

\bibitem{Ferro2002}
S.~Ferro and G.~Gnavi.
\newblock Effects of temperature-dependent viscosity in channels with porous
  walls.
\newblock {\em Phys. Fluids}, 14(2):839--849, 2002.

\bibitem{Girault2015}
V.~Girault, R.~H. Nochetto, and L.~R. Scott.
\newblock Max-norm estimates for {S}tokes and {N}avier--{S}tokes approximations
  in convex polyhedra.
\newblock {\em Numer. Math.}, 131(4):771--882, 2015.

\bibitem{Girault1986}
V.~Girault and P.~A. Raviart.
\newblock {\em Finite Element Methods for Navier-Stokes Equations: Theory and
  Algorithms}.
\newblock Springer Verlag, 1986.

\bibitem{Grinevich2009}
P.~P. Grinevich and M.~A. Olshanskii.
\newblock An iterative method for the {S}tokes-type problem with variable
  viscosity.
\newblock {\em SIAM J. Sci. Comput.}, 31(5):3959--3978, 2009.

\bibitem{Hewitt1975}
F.~M. Hewitt, D.~P. McKenzie, and N.~O. Weiss.
\newblock Dissipative heating in convective flows.
\newblock {\em J. Fluid Mech.}, 68(4):721--738, 1975.

\bibitem{Howle2012}
V.~E. Howle and R.~C. Kirby.
\newblock Block preconditioners for finite element discretization of
  incompressible flow with thermal convection.
\newblock {\em Numer. Linear Algebra Appl.}, 19(2):427--440, 2012.

\bibitem{Hughes1986}
T.~J.~R. Hughes, L.~P. Franca, and M.~Balestra.
\newblock A new finite element formulation for computational fluid dynamics.
V. Circumventing the {B}abu\v{s}ka--{B}rezzi condition: a stable
{P}etrov--{G}alerkin formulation of the {S}tokes problem accommodating
equal-order interpolations.
\newblock {\em Comput. Methods Appl. Mech. Eng.}, 59(1):85--99, 1986.

\bibitem{John2017}
V.~John, A.~Linke, C.~Merdon, M.~Neilan, and L.~G. Rebholz.
\newblock {On the Divergence Constraint in Mixed Finite Element Methods for
  Incompressible Flows}.
\newblock {\em SIAM Rev.}, 59(3):492--544, 2017.

\bibitem{Kagei2000}
Y.~Kagei, M.~R\r{u}\v{z}i\v{c}ka, and G.~Th\"{a}ter.
\newblock Natural convection with dissipative heating.
\newblock {\em Commun. Math. Phys.}, 214:287--313, 2000.

\bibitem{Kay:2006}
D.~Kay, D.~Loghin, and A.~Wathen.
\newblock {A} preconditioner for the steady-state {N}avier--{S}tokes equations.
\newblock {\em SIAM J. Sci. Comput.}, 24(1):237--256, 2002.

\bibitem{Ke2017}
G.~Ke, E.~Aulisa, G.~Bornia, and V.~Howle.
\newblock Block triangular preconditioners for linearization schemes of the
  {R}ayleigh--{B}\'{e}nard convection problem.
\newblock {\em Numer. Linear Algebra Appl.}, 24(5):e2096, 2017.

\bibitem{Ke2018}
G.~Ke, E.~Aulisa, G.~Dillon, and V.~Howle.
\newblock Augmented {L}agrangian-based preconditioners for steady buoyancy
  driven flow.
\newblock {\em Appl. Math. Lett.}, 82:1--7, 2018.

\bibitem{refId0}
C.~Kreuzer and E.~S\"{u}li.
\newblock Adaptive finite element approximation of steady flows of
  incompressible fluids with implicit power-law-like rheology.
\newblock {\em ESAIM: M2AN}, 50(5):1333--1369, 2016.

\bibitem{Lederer2017}
P.~L. Lederer, A.~Linke, and J.~Merdon, C.~Sch\"{o}berl.
\newblock Divergence-free reconstruction operators for pressure-robust {S}tokes
  discretizations with continuous pressure finite elements.
\newblock {\em SIAM J. Numer. Anal.}, 55(3):1291--1314, 2017.

\bibitem{LWXZ:2007}
Y.~J. Lee, J.~Wu, J.~Xu, and L.~Zikatanov.
\newblock {R}obust subspace correction methods for nearly singular systems.
\newblock {\em Math. Models Methods Appl. Sci.}, 17(11):1937--196, 2007.

\bibitem{Linke2014}
A.~Linke.
\newblock On the role of the {H}elmholtz decomposition in mixed methods for
  incompressible flows and a new variational crime.
\newblock {\em Comput. Methods Appl. Mech. Engrg.}, 268:782--800, 2014.

\bibitem{Linke2016}
A.~Linke, G.~Matthies, and L.~Tobiska.
\newblock Robust arbitrary order mixed finite element methods for the
  incompressible {S}tokes equations with pressure independent velocity errors.
\newblock {\em ESAIM: M2AN}, 50(1):289--309, 2016.

\bibitem{Linke2016a}
A.~Linke and C.~Merdon.
\newblock Pressure-robustness and discrete {H}elmholtz projectors in mixed
  finite element methods for the incompressible {N}avier--{S}tokes equations.
\newblock {\em Comput. Methods Appl. Mech. Engrg.}, 311:304--326, 2016.

\bibitem{Linke2017}
A. Linke, C.~Merdon, and W.~Wollner.
\newblock Optimal {$L^2$} velocity error estimate for a modified
  pressure-robust {C}rouzeix--{R}aviart {S}tokes element.
\newblock {\em IMA J. Numer. Anal.}, 37:354--374, 2017.

\bibitem{Mardal2011}
K.~{-A.} Mardal and R.~Winther.
\newblock Preconditioning discretizations of systems of partial differential
  equations.
\newblock {\em Numer. Linear Algebr.}, 18(1):1--40, 2011.

\bibitem{Maringova2018}
E.~Maringov\'{a} and J.~\v{Z}abensk\'{y}.
\newblock On a {N}avier--{S}tokes--{F}ourier-like system capturing transitions
  between viscous and inviscid fluid regimes and between no-slip and
  perfect-slip boundary conditions.
\newblock {\em Nonlinear Anal. Real World Appl.}, 41:152--178, 2018.

\bibitem{Necas2001}
J.~Ne\v{c}as and T.~Roubi\v{c}ek.
\newblock Buoyancy-driven viscous flow with {$L^1$}-data.
\newblock {\em Nonlinear Anal.}, 46:737--755, 2001.

\bibitem{Oberbeck1879}
A.~Oberbeck.
\newblock \"{U}ber die {W}\"{a}rmeleitung der {F}l\"{u}ssigkeiten bei der
  {B}er\"{u}cksichtigung der {S}tr\"{o}mungen infolge von
  {T}emperaturdifferenzen.
\newblock {\em Ann. Phys.}, 243(6):271--292, 1879.

\bibitem{Ostrach1958}
S.~Ostrach.
\newblock {\em Internal viscous flows with body forces}, pages 185--208.
\newblock Grenzschictsforschung. Springer Verlag, 1958.

\bibitem{Qin1994}
J.~Qin.
\newblock {\em On the convergence of some low order mixed finite elements for
  incompressible fluids}.
\newblock PhD thesis, Pennsylvania State University, 1994.

\bibitem{Rajagopal:2003}
K.~R. Rajagopal.
\newblock On implicit constitutive theories.
\newblock {\em Appl. Math.}, 48(4):279--319, 2003.

\bibitem{Rajagopal2006}
K.~R. Rajagopal.
\newblock {On implicit constitutive theories for fluids}.
\newblock {\em J. Fluid Mech.}, 550:243--249, March 2006.

\bibitem{Rajagopal2008}
K.~R. Rajagopal and A.~R. Srinivasa.
\newblock {On the thermodynamics of fluids defined by implicit constitutive
  relations}.
\newblock {\em Z. Angew. Math. Phys}, 59:715--729, July 2008.

\bibitem{Rathgeber2016}
F.~Rathgeber, D.~A. Ham, L.~Mitchell, M.~Lange, F.~Luporini, A.~T.~T. Mcrae,
  G-T. Bercea, G.~R. Markall, and P.~H.~J. Kelly.
\newblock Firedrake: automating the finite element method by composing
  abstractions.
\newblock {\em ACM Trans. Math. Softw.}, 43(3), 2016.

\bibitem{Roubicek2001}
T.~Roubi\v{c}ek.
\newblock Steady-state buoyancy-driven viscous flow with measure data.
\newblock {\em Math. Bohem.}, 126(2):493--504, 2001.

\bibitem{Schoeberl:1999}
J.~Sch\"{o}berl.
\newblock {M}ultigrid methods for a parameter dependent problem in primal
  variables.
\newblock {\em Numer. Math.}, 84(1):97–119, 1999.

\bibitem{Schoeberl1999}
J.~Sch\"{o}berl.
\newblock {\em Robust Multigrid Methods for Parameter Dependent Problems}.
\newblock PhD thesis, Johannes Kepler Universit\"{a}t Linz, 1999.

\bibitem{Scott1985}
L.~R. Scott and M.~Vogelius.
\newblock {Norm estimates for a maximal right inverse of the divergence
  operator in spaces of piecewise polynomials}.
\newblock {\em Math. Modelling Numer. Anal.}, 19(1):111--143, 1985.

\bibitem{Silvester:1994}
D.~Silvester and A.~Wathen.
\newblock {F}ast iterative solution of stabilised {S}tokes systems. {P}art
  {II}: Using general block preconditioners.
\newblock {\em SIAM J. Numer. Anal.}, 31(5):1352--1367, 1994.

\bibitem{Sueli2018}
E.~S\"{u}li and T.~Tscherpel.
\newblock Fully discrete finite element approximation of unsteady flows of
  implicitly constituted incompressible fluids.
\newblock {\em IMA J. Numer. Anal.}, dry097, 2019.

\bibitem{Tscherpel2018}
T.~Tscherpel.
\newblock {\em {FEM for the Unsteady Flow of Implicitly Constituted
  Incompressible Fluids}}.
\newblock PhD thesis, University of Oxford, 2018.

\bibitem{Turcotte1974}
D.~L. Turcotte, A.~T. Hsui, K.~E. Torrance, and G.~Schubert.
\newblock Influence of viscous dissipation on {B}\'{e}nard convection.
\newblock {\em J. Fluid Mech.}, 64(2):369--374, 1974.

\bibitem{Velarde1976}
M.~G. Velarde and R.~P\'{e}rez-Cordon.
\newblock On the (non-linear) foundations of {B}oussinesq approximation
  applicable to a thin layer of fluid {II}: {V}iscous dissipation and large
  cell gap effects.
\newblock {\em J. de Physique}, 37(3):178--182, 1976.

\bibitem{Vinay2005}
G.~Vinay, A.~Wachs, and J-F. Agassant.
\newblock Numerical simulation of non-isothermal viscoplastic waxy crude oil
  flows.
\newblock {\em J. Non-Newtonian Fluid Mech.}, 128:144--162, 2005.

\bibitem{Xu1992}
J.~Xu.
\newblock Iterative methods by space decomposition and subspace correction.
\newblock {\em {SIAM} Review}, 34(4):581--613, 1992.

\bibitem{Xu2001}
J.~Xu.
\newblock The method of subspace corrections.
\newblock {\em J. Comput. Appl. Math.}, 128(1):335--362, 2001.

\bibitem{Zhang2005}
S.~Zhang.
\newblock A new family of stable mixed finite elements for the 3{D} {S}tokes
  equations.
\newblock {\em Math. Comput.}, 74(250):543--554, 2005.

\end{thebibliography}

\end{document}